\documentclass[12pt]{amsart}
\usepackage[margin=1.2in]{geometry}
\usepackage{graphicx,latexsym,graphics}
\usepackage{amsfonts, amssymb, amsmath, amsthm}
\usepackage{algorithmic, algorithm}
\usepackage{tikz}
\usetikzlibrary{matrix,arrows}

\newtheorem{te}{Theorem}[section]
\newtheorem{pro}[te]{Proposition}
\newtheorem{co}[te]{Corollary}
\newtheorem{lemma}[te]{Lemma}

\theoremstyle{definition}
\newtheorem{de}[te]{Definition}

\newtheorem{re}[te]{Remark}

\numberwithin{equation}{section}

\allowdisplaybreaks

\newcommand{\beq}{\begin{eqnarray}}
\newcommand{\eeq}{\end{eqnarray}}

\newcommand{\beqs}{\begin{eqnarray*}}
\newcommand{\eeqs}{\end{eqnarray*}}

\newcommand{\NN}{\mathbb N}

\newcommand{\CC}{\mathbb C}
\newcommand{\RR}{\mathbb R}
\newcommand{\ZZ}{\mathbb Z}
\newcommand{\EE}{\mathcal E}
\newcommand{\DD}{\mathcal D}
\newcommand{\SSS}{\mathcal S}

\newcommand{\ds}{\displaystyle}

\begin{document}

\title[Convolution and translation-invariant spaces]{Convolution of ultradistributions and ultradistribution spaces associated to translation-invariant Banach spaces}

\author[P. Dimovski]{Pavel Dimovski}
\address{P. Dimovski, Faculty of Technology and Metallurgy, University Ss. Cyril and Methodius, Ruger Boskovic 16, 1000 Skopje, Republic of Macedonia}
\email{dimovski.pavel@gmail.com}

\author[S. Pilipovi\'{c}]{Stevan Pilipovi\'{c}}
\thanks{S. Pilipovi\'{c} is supported by the Serbian Ministry of Education, Science and Technological Development, through the project 174024}
\address{S. Pilipovi\'{c}, Department of Mathematics and Informatics, University of Novi Sad, Trg Dositeja Obradovi\'ca 4, 21000 Novi Sad, Serbia}
\email {stevan.pilipovic@dmi.uns.ac.rs}

\author[B. Prangoski]{Bojan Prangoski}
\address{B. Prangoski, Faculty of Mechanical Engineering, University Ss. Cyril and Methodius,
Karpos II bb, 1000 Skopje, Macedonia}
\email{bprangoski@yahoo.com}

\author[J. Vindas]{Jasson Vindas}
\thanks{J. Vindas gratefully acknowledges support by Ghent University, through the BOF-grant
01N01014}
\address{J. Vindas, Department of Mathematics, Ghent University, Krijgslaan 281 Gebouw S22, 9000 Gent, Belgium}
\email{jvindas@cage.UGent.be}

\subjclass[2010]{Primary 46F05; Secondary 46H25; 46F10; 46E10}
\keywords{Convolution of ultradistributions; Translation-invariant Banach space of tempered ultradistributions; Ultratempered convolutors; Beurling algebra; parametrix method}

\begin{abstract}We introduce and study a number of new spaces of ultradifferentiable functions and ultradistributions and we apply our results to the study of the convolution of ultradistributions. The spaces of convolutors $\mathcal{O}'^{\ast}_{C}(\mathbb{R}^{d})$ for tempered ultradistributions are analyzed via the duality with respect to the test function spaces $\mathcal{O}^{\ast}_{C}(\mathbb{R}^{d})$, introduced in this article. We also study ultradistribution spaces associated to translation-invariant Banach spaces of tempered ultradistributions and use their properties to provide a full characterization of the general convolution of Roumieu ultradistributions via the space of integrable ultradistributions. We show that the convolution of two Roumieu ultradistributions $T,S\in \DD'^{\{M_p\}}\left(\RR^d\right)$ exists if and only if $\left(\varphi*\check{S}\right)T\in\DD'^{\{M_p\}}_{L^1}\left(\RR^d\right)$
for every $\varphi\in\DD^{\{M_p\}}\left(\RR^d\right)$.
\end{abstract}

\maketitle

\section{Introduction}
This article is devoted to the study of various problems concerning the convolution in the setting of ultradistributions. A detailed study of some of such problems has been lacking in the theory of ultradistributions for more than 30 years. In addition, we introduce new spaces of ultradifferentiable functions and ultradistributions associated to a class of translation-invariant Banach spaces as an essential tool in this work.

In the first part of the paper we analyze the space of convolutors -- called here ultratempered convolutors -- for the space of tempered ultradistributions.
Naturally, such an investigation would be of general interest as being part of the modern theory of multipliers. In the case of tempered distributions, the space of convolutors was introduced by Schwartz \cite{S} and its full topological characterization was given years later in Horv\'{a}th's book \cite{horvath} (see also \cite{OW14}). The space of ultratempered convolutors $\mathcal O_{C}'^{\ast}(\mathbb{R}^{d})$ was recently studied in \cite{PBD}. Our first important result is the description of $\mathcal O_{C}'^{\ast}(\mathbb{R}^{d})$ through the duality with respect to the test function space $\mathcal O_C^{\ast}(\mathbb{R}^{d})$, constructed in this article. The treatment of the Roumieu case is considerably more elaborated than the Beurling one, as it involves the use of dual Mittag-Leffler lemma arguments for establishing the sought duality.

The second important achievement of the article is related to the existence of the general convolution of ultradistributions of Roumieu type. After the introduction of Schwartz' conditions for the general convolvability of distributions, many authors gave alternative definitions and established their equivalence. Notably, Shiraishi \cite{Sh} found out that the convolution of two distributions $S,T\in\mathcal{D}'(\mathbb{R}^{d})$ exists if and only if: $\left(\varphi*\check{S}\right)T\in\DD'_{L^1}\left(\RR^d\right)$
for every $\varphi\in\DD\left(\RR^d\right)$.  The existence of the convolution for Beurling ultradistributions can be treated \cite{KKP94,KK13,PilipovicC} analogously as that for Schwartz distributions. In contrast, corresponding characterizations for the convolution of Roumieu ultradistributions has been a long-standing open question in the area. It was only until recently \cite{PB} that progress in this direction was made through the study of $\varepsilon$ tensor products of $\dot{\tilde{\mathcal{B}}}^{\{M_p\}}$ and locally convex spaces. The following characterization of convolvability was shown in \cite{PB}: The convolution of two ultradistributions $T,S\in \DD'^{\{M_p\}}\left(\RR^d\right)$ exists if and only if $\left(\varphi*\check{S}\right)T\in\tilde{\DD'}^{\{M_p\}}_{L^1}\left(\RR^d\right)$
for every $\varphi\in\DD^{\{M_p\}}\left(\RR^d\right)$ and for every compact subset $K$ of $\RR^d$,
$(\varphi,\chi)\mapsto \left\langle\left(\varphi*\check{T}\right)S,\chi\right\rangle$, $\DD^{\{M_p\}}_K\times\dot{\tilde{\mathcal{B}}}^{\{M_p\}}\longrightarrow \CC$, is a continuous bilinear mapping.
The spaces $\dot{\tilde{\mathcal{B}}}^{\{M_p\}}$ and $\tilde{\DD'}^{\{M_p\}}_{L^1}\left(\RR^d\right)$ were introduced in \cite{Pilipovic}.
In this paper we shall make a significant improvement to this result, namely,  we shall show the following more transparent version of Shiraishi's result for Roumieu ultradistributions: \emph{the convolution of $T,S\in \DD'^{\{M_p\}}\left(\RR^d\right)$ exists if and only if $\left(\varphi*\check{S}\right)T\in\DD'^{\{M_p\}}_{L^1}\left(\RR^d\right)$
for every $\varphi\in\DD^{\{M_p\}}\left(\RR^d\right)$.}

Our proof of the above-mentioned result about the general convolvability of Roumieu ultradistributions is postponed to the last section of the article and it is based upon establishing the topological equality $\tilde{\DD'}^{\{M_p\}}_{L^1}={\DD'}^{\{M_p\}}_{L^1}.$ This and other topological properties of the spaces of  integrable ultradistributions can be better understood from a rather broader perspective. In this paper we introduce and study new classes of  translation-invariant ultradistribution spaces which are natural generalizations of the weighted $\mathcal{D}'^{\ast}_{L^{p}}$-spaces \cite{BFG,PilipovicK}. In the distribution setting, the recent work \cite{DPV} extends that of Schwartz on the $\mathcal{D}'_{L^{p}}$-spaces and that of Ortner and Wagner on their weighted versions \cite{OW90,Wagner}; recent applications of those ideas to the study of boundary values of holomorphic functions and solutions to the heat equation can be found in \cite{DPV1}.
The theory we present here is a generalization of that given in \cite{DPV} for distributions. Although some results are analogous to those for distributions, it should be remarked that their proofs turn out to be much more complicated since they demand the use of more sophisticated techniques and new ideas adapted to the ultradistribution setting-- especially in the Roumieu case.

The article is organized in eight sections. In Section \ref{sekcijakonvolutori} we characterize the spaces of  tempered ultradistributions $\mathcal{S}'^{\ast}(\mathbb{R}^{d})$ in terms of growth estimates for convolution averages of their elements, extending thus an important structural theorem of Schwartz \cite[Thm. VI, p. 239]{S}. Using Komatsu's approach to ultradistribution theory \cite{Komatsu1}, we define the test function spaces $\mathcal{O}^*_C(\RR^d)$ whose strong duals are algebraically isomorphic to the ultratempered convolutor spaces $\mathcal{O}'^*_C(\RR^d)$. We also obtain there structural theorems for $\mathcal{O}'^*_C(\RR^d)$.

Section \ref{UTIB} is dedicated to the analysis of translation-invariant Banach spaces of tempered ultradistributions. We are interested in the class of Banach spaces of ultradistributions that satisfy the ensuing three conditions: (I) $\mathcal{D}^*(\mathbb{R}^d)\hookrightarrow E\hookrightarrow \mathcal{D}'^*(\mathbb{R}^d)$, (II) $E$ is translation-invariant and (III) the function $\omega(h):=\| T_{-h} \| $ has at most ultrapolynomial growth. Such $E$ becomes a Banach module over the Beurling algebra $L_\omega^1$ and has nice approximation properties with respect to the translation group. In particular, we show that the translation group on $E$ is a $C_0$-semigroup (i.e., $\lim_{h\rightarrow 0}\|T_hg-g\|_{E}=0$ for each $g\in E$).  Using duality, we obtain some results concerning $E'$ which also turns out to be a Banach module over the Beurling algebra $L_{\check{\omega}}^1$, but $E'$ may fail to have many of the properties that $E$ enjoys. That motivates the introduction of a closed subspace $E'_*$ of $E'$ that satisfies the axioms (II) and (III) and it is characterized as the biggest subspace of $E'$ for which $\lim_{h\rightarrow 0}\|T_hf-f\|_{E'}=0$ for all its elements.

In Section \ref{tspace} we define our new test spaces $\mathcal{D}_E^{(M_p)}$ and $\mathcal{D}_E^{\{M_p\}}$ of Beurling and Roumieu type, respectively. In the Roumieu case we also consider another space $\tilde{\mathcal{D}}_E^{\{M_p\}}$ (in connection to it, see \cite{Komatsu3} for related spaces). We show that the elements of all these test spaces are in fact ultradifferentiable functions and the continuous and dense embeddings $\mathcal{S}^*(\mathbb{R}^d)\hookrightarrow\mathcal{D}^{*}_E\hookrightarrow E\hookrightarrow \mathcal{S}'^*(\mathbb{R}^d)$ hold. We also prove that the spaces $\mathcal{D}^{*}_E$ are topological modules over the Beurling algebra
$L^{1}_{\omega}$. The spaces $\mathcal{D}^{*}_E$ are continuously and densely embedded into the spaces $\mathcal{O}^{*}_C(\RR^d)$ introduced in Section \ref{sekcijakonvolutori}.

In Section \ref{subsection DE} we investigate the topological and structural properties of the strong dual of $\mathcal{D}^{*}_E$, denoted as $\mathcal{D}'^{*}_{E'_*}$. A structural theorem for $\mathcal{D}'^{*}_{E'_*}$ is given; there, we caracterize its elemets in terms of convolution averages and also via representations as finite sums of actions of ultradifferential operators on elements from $E'_*$. Our results enable us to embed the spaces $\mathcal{D}^{*}_E$ into the spaces of $E'_*$-valued tempered ultradistributions $\mathcal{S}'^*(\RR^d, E'_*)$. We prove that the spaces $\mathcal{D}_E^{\{M_p\}}$ and $\tilde{\mathcal{D}}_E^{\{M_p\}}$ are topologically isomorphic. When $E$ is reflexive, we show that $\mathcal{D}^{(M_p)}_{E}$ and $\mathcal{D}'^{\{M_p\}}_{E'}$ are $(FS^*)$-spaces, while $\mathcal{D}^{\{M_p\}}_{E}$ and $\mathcal{D}'^{(M_p)}_{E}$ are $(DFS^*)$-spaces.

Section \ref{examples} is devoted to the weighted spaces $\mathcal{D}^*_{L^{p}_{\eta}}$ and $\mathcal{D}'^*_{L_{\eta}^{p}}$, which we treat here as examples of the spaces $\mathcal{D}^{*}_{E}$ and  $\mathcal{D}'^*_{E'_*}$. This approach allows us to prove the topological identification of $\mathcal {D}^*_{C_\eta}$ with the spaces $\dot{\mathcal{B}}^*_\eta$ and $\dot{\tilde{\mathcal{B}}}^*_\eta$, which actually leads to the topological equality $\tilde{\DD'}^{\{M_p\}}_{L^1}={\DD'}^{\{M_p\}}_{L^1}$ and additional topological information about $\mathcal{D}^*_{L^\infty}$.

Finally, Section \ref{section convolution} deals with applications to the study of the convolution of ultradistributions.  We provide there the announced improvement to the result from \cite{PB} for the existence of the general convolution of Roumieu ultradistributions. We also obtain in this section results concerning convolution and multiplicative products on the spaces $\mathcal{D}'^*_{E'_*}$, generalizing distribution analogues from \cite{DPV}.

\section{Preliminaries}

As usual in this theory,  $M_{p}$, $p\in \NN$, $M_0=1$, denotes a sequence of positive numbers for which we assume (see \cite{Komatsu1}):
 $(M.1)$ $M_{p}^{2} \leq M_{p-1} M_{p+1}$, $p \in\ZZ_+$;
 $(M.2)$  $\ds M_{p} \leq c_0H^{p} \min_{0\leq q\leq p} \{ M_{p-q} M_{q}\}$, $p,q\in \NN$, for some $c_0,H\geq1$;
 $(M.3)$  $\sum^{\infty}_{p=q+1}   M_{p-1}/M_{p}\leq c_0q M_{q}/M_{q+1}$, $q\in \ZZ_+$. For a multi-index $\alpha\in\NN^d$, $M_{\alpha}$ means $M_{|\alpha|}$. The associated function of the sequence $M_{p}$ is given by the function $\ds M(\rho)=\sup  _{p\in\NN}\ln_+   \frac{\rho^{p}}{M_{p}}$, $\rho > 0$. It is a nonnegative, continuous, monotonically increasing function, vanishes for sufficiently small $\rho>0$, and increases more rapidly than $\ln \rho^p$ as $\rho$ tends to infinity, for any $p\in\NN$ (cf. \cite[p. 48]{Komatsu1}).

Let $U\subseteq\RR^d$ be an open set and $K\Subset U$ be a compact subset. (We will always use this notation for a compact subset of an open set.) Recall that $\EE^{\{M_p\},h}(K)$ stands for the Banach space (from now on abbreviated as $(B)$-space) of all $\varphi\in C^{\infty}(U)$ which satisfy $p_{K,h}(\varphi)=\ds\sup_{\alpha\in\NN^d}\sup_{x\in K}\frac{|D^{\alpha}\varphi(x)|}{h^{\alpha}M_{\alpha}}<\infty$ and $\DD^{\{M_p\}}_{K,h}$ stands for its subspace consisting of elements supported by $K$. Then
$$
\EE^{(M_p)}(U)=\lim_{\substack{\longleftarrow\\ K\Subset U}}\lim_{\substack{\longleftarrow\\ h\rightarrow 0}} \EE^{\{M_p\},h}(K),\,\,\,\,
\EE^{\{M_p\}}(U)=\lim_{\substack{\longleftarrow\\ K\Subset U}}
\lim_{\substack{\longrightarrow\\ h\rightarrow \infty}} \EE^{\{M_p\},h}(K),
$$
\beqs
\DD^{(M_p)}_K=\lim_{\substack{\longleftarrow\\ h\rightarrow 0}} \DD^{\{M_p\}}_{K,h},\,\,\,\, \DD^{(M_p)}(U)=\lim_{\substack{\longrightarrow\\ K\Subset U}}\DD^{(M_p)}_K,\\
\DD^{\{M_p\}}_K=\lim_{\substack{\longrightarrow\\ h\rightarrow \infty}} \DD^{\{M_p\}}_{K,h},\,\,\,\, \DD^{\{M_p\}}(U)=\lim_{\substack{\longrightarrow\\ K\Subset U}}\DD^{\{M_p\}}_K.
\eeqs

The spaces of ultradistributions and compactly supported ultradistributions of Beurling and Roumieu type are defined as the strong duals of $\DD^{(M_p)}(U)$ and $\EE^{(M_p)}(U)$, and $\DD^{\{M_p\}}(U)$ and $\EE^{\{M_p\}}(U)$ respectively. We refer to \cite{Komatsu1,Komatsu2,Komatsu3} for the properties of these spaces. Following Komatsu \cite{Komatsu1}, the common notation for $(M_p)$ and $\{M_p\}$ will be $*$. In the definitions and statements where we consider the $(M_p)$ and $\{M_p\}$ cases simultaneously, we will always first state the assertions for the Beurling case followed by the corresponding assertion for the Roumieu case in parentheses.

  We define ultradifferential operators as in \cite{Komatsu1}. The function $P(\xi ) = \sum _{\alpha \in \NN^d} c_{\alpha } \xi ^{\alpha}$, $\xi \in \RR^d$, is called an
ultrapolynomial of the class $(M_p)$ (of class $\{M_{p}\}$) if the coefficients $c_{\alpha }$ satisfy
the estimate $|c_{\alpha }|  \leq C L^{\alpha }/M_{\alpha}$, $\alpha \in \NN^d$,
for some $C,L>0$ (for every $L > 0 $ and a corresponding $C=C_{L} > 0$). Then $P(D)=\sum_{\alpha} c_{\alpha}D^{\alpha}$ is an ultradifferential operator of the class $*$ and it acts continuously on $\EE^{*}(U)$ and $\DD^{*}(U)$ and the corresponding spaces of ultradistributions $\EE'^{*}(U)$ and $\DD'^{*}(U)$.

We denote as $\mathfrak{R}$ the set of all positive sequences which monotonically increase to infinity. For $(r_j)\in\mathfrak{R}$, we write $R_k$ for the product $\prod_{j=1}^k r_j$ and $R_0=1$. For $(r_p)\in\mathfrak{R}$, consider the sequence $N_0=1$, $N_p=M_pR_p$, $p\in\ZZ_+$. Its associated function will be denoted by $N_{r_p}(\rho)$, that is, $\ds N_{r_{p}}(\rho )=\sup_{p\in\NN} \ln_+ \frac{\rho^{p }}{M_pR_p}$, $\rho > 0$. As proved in \cite[Prop. 3.5]{Komatsu3}, the seminorms $\|\varphi\|_{K,(r_j)}=\ds \sup_{\alpha\in\NN^d}\sup_{x\in K} \frac{\left|D^{\alpha}\varphi(x)\right|}{R_{\alpha}M_{\alpha}}$, when $K$ ranges over compact subsets of $U$ and $(r_j)$ in $\mathfrak{R}$, give the topology of $\EE^{\{M_p\}}(U)$. Also, for $K\Subset\RR^d$, the topology of $\DD^{\{M_p\}}_K$ is given by the seminorms
$\|\cdot\|_{K,(r_j)}$, with $(r_j)$ ranging over $\mathfrak{R}$. From this it follows that $\ds \DD^{\{M_p\}}_K=\lim_{\substack{\longleftarrow\\ (r_j)\in\mathfrak{R}}} \DD^{\{M_p\}}_{K,(r_j)}$,
where $\DD^{\{M_p\}}_{K,(r_j)}$ is the $(B)$-space of all $C^{\infty}$-functions supported by $K$ for which the norm $\|\cdot\|_{K,(r_j)}$ is finite. Furthermore, for $U$ open and $r>0$ (for $(r_j)\in\mathfrak{R}$) we denote $\ds \DD^{(M_p)}_{U,r}=\lim_{\substack{\longrightarrow\\ K\Subset U}}\DD^{\{M_p\},r}_{K}$ ($\ds \DD^{\{M_p\}}_{U,(r_j)}=\lim_{\substack{\longrightarrow\\ K\Subset U}}\DD^{\{M_p\}}_{K,(r_j)}$). Both spaces carry natural $(LB)$ topologies (but we shall not need this fact).

We will often make use of the following lemma by Komatsu (see \cite[p. 195]{Komatsu}). In the future we refer to it as the parametrix of Komatsu.

\begin{lemma}[\cite{Komatsu}]
Let $K$ be a compact neighborhood of zero, $r > 0$, and $(r_p)\in\mathfrak{R}$.
\begin{itemize}
\item[$(i)$] There are $u\in\DD^{\{M_p\}}_{K,r}$ and $\psi\in\DD^{(M_p)}_K$ such that $P(D)u=\delta+\psi$ where $P(D)$ is ultradifferential operator of  class $(M_p)$.
\item[$(ii)$] There are $u\in \DD^{\{M_p\}}_{K,(r_{j})}$ and $\psi\in\DD^{\{M_p\}}_K$ such that
    \beqs
    \frac{\left\|D^{\alpha} u\right\|_{L^{\infty}(K)}}{M_{\alpha} \prod_{j=1}^{|\alpha|}r_j}\rightarrow 0 \mbox{ as } |\alpha|\rightarrow\infty \mbox{ and } P(D)u=\delta+\psi,
    \eeqs
    where $P(D)$ is an ultradifferential operator of class $\{M_p\}$.
\end{itemize}
\end{lemma}

We denote as $\SSS^{\{M_{p}\},m}_{\infty} \left(\RR^d\right)$, $m > 0$, the $(B)$-space of all $\varphi\in C^{\infty}\left(\RR^d\right)$ which satisfy
\beq\label{pcnnk}
\sigma_{m}(\varphi): =\sup_{\alpha\in\NN^d}\frac{m^{|\alpha|}\left\|e^{M(m|\cdot|)}D^{\alpha}\varphi\right\|_{L^{\infty}}}{M_{\alpha}}<\infty,
\eeq
supplied with the norm $\sigma _{m}$. The spaces $\SSS'^{(M_{p})}(\RR^d)$ and $\SSS'^{\{M_{p}\}}(\RR^d)$ of tempered ultradistributions of Beurling and Roumieu type are defined as the strong duals of $\ds\SSS^{(M_{p})}(\RR^d)=\lim_{\substack{\longleftarrow\\ m\rightarrow\infty}}\SSS^{\{M_{p}\},m}_{\infty}\left(\RR^d\right)$ and $\ds\SSS^{\{M_{p}\}}(\RR^d)=\lim_{\substack{\longrightarrow\\ m\rightarrow 0}}\SSS^{\{M_{p}\},m}_{\infty}\left(\RR^d\right)$ respectively. For the properties of these spaces, we refer to \cite{PilipovicK,PilipovicU,Pilipovic}. It is proved in \cite[p. 34]{PilipovicK} and \cite[Lem. 4]{Pilipovic} that $\ds\SSS^{\{M_{p}\}}(\RR^d) = \lim_{\substack{\longleftarrow\\ (r_{i}), (s_{j}) \in \mathfrak{R}}}\SSS^{M_{p}}_{(r_{p}),(s_{q})}(\RR^d)$, where $\ds\SSS^{M_{p}}_{(r_{p}),(s_{q})}(\RR^d)=\left\{\varphi \in C^{\infty} \left(\RR^d\right)|\|\varphi\|_{(r_{p}),(s_{q})}<\infty\right\}$ and $\ds\|\varphi\|_{(r_{p}),(s_{q})} =\sup_{\alpha\in \NN^d}\frac{\left\|e^{N_{s_p}(|\cdot|)}D^{\alpha}\varphi\right\|_{L^{\infty}}} {M_{\alpha}\prod^{|\alpha|}_{p=1}r_{p}}$.

We denote as $\mathcal{O}'^*_C(\RR^d)$ the space of convolutors of $\SSS^*(\RR^d)$, that is, the subspace of all $f\in\SSS'^*(\RR^d)$ such that $f*\varphi\in\SSS^*(\RR^d)$ for all $\varphi\in\SSS^*(\RR^d)$ and the mapping $\varphi\mapsto f*\varphi$, $\SSS^*(\RR^d)\rightarrow\SSS^*(\RR^d)$ is continuous. We refer to \cite{PBD} for its properties.

Finally, we need the following technical result \cite[Lem. 2.4]{BojanL}. See \cite[p. 53]{Komatsu1} for the definition of subordinate function.

\begin{lemma}[\cite{BojanL}]\label{ppp}
Let $g:[0,\infty)\rightarrow[0,\infty)$ be an increasing function that satisfies the following estimate: for every $L>0$ there exists $C>0$ such that $g(\rho)\leq M(L\rho)+ \ln C$. Then, there exists subordinate function $\epsilon(\rho)$ such that $g(\rho)\leq M(\epsilon(\rho))+ \ln C'$, for some constant $C'>1$.
\end{lemma}

\section{On the space of ultratempered convolutors}\label{sekcijakonvolutori}

Our goal in this section is to construct a test function space whose dual is algebraically isomorphic to $\mathcal{O}'^{\ast}_{C}(\RR^d)$. (We refer to \cite{PBD} for properties of the latter space.) We start with an important characterization of tempered ultradistributions in terms of growth properties of convolution averages, an analogue to this result for $\mathcal{S}'(\RR^d)$ was obtained long ago by Schwartz (cf. \cite[Thm. VI, p. 239]{S})

\begin{pro}\label{S'}
Let $f\in \mathcal{D}'^*(\mathbb{R}^d)$. Then, $f$ belongs to $\mathcal{S}'^*(\mathbb{R}^d)$ if and only if there exists $\lambda>0$ (there exists $(l_p)\in\mathfrak{R}$) such that for every $\varphi \in \mathcal{D}^*(\mathbb{R}^d)$
\begin{equation}
\label{eqgrowthMp}
\sup_{x\in\mathbb{R}^{d}} e^{-M(\lambda |x|)} |(f*\varphi)(x)|<\infty\,\,\,\,\, \left(\sup_{x\in\mathbb{R}^{d}} e^{-N_{l_p}(|x|)} |(f*\varphi)(x)|<\infty\right).
\end{equation}
\end{pro}

\begin{proof} Observe that if $f\in\SSS'^*\left(\RR^d\right)$ then (\ref{eqgrowthMp}) obviously holds (one just needs to apply the representation theorem for the elements of $\SSS'^*\left(\RR^d\right)$, see \cite[Thm. 2.6.1, p. 38]{PilipovicK}). We prove the converse part only in the ${\{M_p\}}$ case; the $(M_p)$ case is similar. Let $\Omega$ be an open bounded subset of $\mathbb{R}^d$ which contains $0$ and is symmetric (i.e., $-\Omega=\Omega$) and denote $\overline{\Omega}=K$. Let $B_1$ be the unit ball in the weighted $(B)$-space $L^1_{\exp(N_{l_p}(|\: \cdot \:|))}$. Fix $\varphi\in \DD^{\{M_p\}}_K$. For every $\phi \in B_1\cap \mathcal{D}^{\{M_p\}}(\RR^d)$, (\ref{eqgrowthMp}) implies $|\langle f*\phi,\varphi\rangle|= |\langle f*\check{\varphi},\check{\phi}\rangle|\leq \left\|e^{-N_{l_p}(|\:\cdot\:|)}(f*\check{\varphi})\right\|_{L^{\infty}} \|\phi\|_{L^1_{\exp(N_{l_p}(|\cdot |))}}\leq C_{\varphi}$. We obtain $\left\{f*\phi|\, \phi\in B_1\cap \mathcal{D}^{\{M_p\}}(\RR^d)\right\}$ is weakly bounded and, hence, equicontinuous in $\DD'^{\{M_p\}}_K$ ($\DD^{\{M_p\}}_K$ is barreled). Hence, there exist $(k_p)\in \mathcal{R}$ and $\varepsilon>0$ such that $|\langle f*\psi,\check{\phi}\rangle|\leq 1$ for all $\psi\in V_{k_p}(\varepsilon)=\{\eta\in \mathcal{D}^{\{M_p\}}_K|\,\|\eta\|_{K,k_p}\leq \varepsilon\}$ and $\phi\in B_1\cap \mathcal{D}^{\{M_p\}}(\mathbb{R}^d)$.

 Let $r_p=k_{p-1}/H$, for $p\in\NN$, $p\geq 2$ and set $r_1=\min\{1,r_2\}$. Then $(r_p)\in \mathfrak{R}$. Let $\psi\in \mathcal{D}^{\{M_p\}}_{\Omega,(r_p)}$ and choose $C_{\psi}$ such that $\|\psi/C_{\psi}\|_{K,(r_p)}\leq \varepsilon/2$. Let $\delta_1\in\DD^{\{M_p\}}\left(\RR^d\right)$ such that $\delta_1\geq 0$, $\mathrm{supp}\, \delta_1\subseteq\{x\in\RR^d|\, |x|\leq 1\}$ and $\int_{\RR^d} \delta_1(x)dx=1$. Set $\delta_j(x)=j^d \delta_1(jx)$, for $j\in\NN$, $j\geq 2$. Observe that for $j$ large enough $\psi*\delta_j\in\DD^{\{M_p\}}_K$. Also
\beqs
|\partial^\alpha((\psi*\delta_j)(x)-\psi(x))|\leq \int_{\mathbb{R}^d}|\partial^\alpha(\psi(x-t)-\psi(x))|\delta_j(t)dt.
\eeqs
Using the Taylor expansion of the function $\partial^\alpha\psi$ at the point $x-t$, we obtain
\beqs
\left|\partial^\alpha\left(\psi(x)-\psi(x-t)\right)\right|\leq\sum_{|\beta|=1}\left|t^{\beta}\right| \int_0^1\left|\partial^{\alpha+\beta}\psi(sx+(1-s)(x-t))\right|ds\leq C |t| M_{|\alpha|+1}\prod_{i=1}^{|\alpha|+1}{r_i}.
\eeqs
So, for $j$ large enough, keeping in mind the definition of $(r_p)$ and by using $(M.2)$ for $M_p$ we have
\beqs
|\partial^\alpha((\psi*\delta_j)(x)-\psi(x))|&\leq& \frac{C'_1}{j} M_{\alpha+1}\prod_{i=2}^{|\alpha|+1}{r_i} \int_{\mathrm{supp}\,\delta_j}\delta_j(t) dt\\
&\leq& \frac{C''_1}{j} H^{|\alpha|+1}M_{\alpha}\prod_{i=1}^{|\alpha|}(k_i/H)\leq \frac{C_1}{j}M_{\alpha}\prod_{i=1}^{|\alpha|}k_i.
\eeqs
Hence $C_{\psi}^{-1}\psi*\delta_j\in V_{(k_p)}(\varepsilon)$ for all large enough $j$. We obtain $\left|\langle f*(\psi*\delta_j),\phi\rangle\right|\leq C_{\psi}$ and after passing to the limit $\left|\langle f*\psi,\phi\rangle\right|\leq C_{\psi}$. From the arbitrariness of $\psi$ we have that for every $\psi\in \mathcal{D}^{\{M_p\}}_{\Omega,(r_p)}$ there exists $C_{\psi}>0$ such that $\left|\langle f*\psi,\phi\rangle\right|\leq C_{\psi} \|\phi\|_{L^1_{\exp(N_{l_p}(|x|))}}$, for all $\phi\in \DD^{\{M_p\}}\left(\RR^d\right)$. The density of $\mathcal{D}^{\{M_p\}}(\mathbb{R}^d)$ in $L^1_{\exp(N_{l_p}(|\cdot|))}$ implies that for every fixed $\psi\in \mathcal{D}^{\{M_p\}}_{\Omega,(r_p)}$, $f*\psi$ is a continuous functional on $L^1_{\exp(N_{l_p}(|\cdot|))}$; hence, $\|\exp(-N_{l_p}(|\cdot|))(f*\psi)\|_{L^{\infty}}\leq C_{2,\psi}$. From the parametrix of Komatsu, for the sequence $(r_p)$ there are $u\in \mathcal{D}^{\{M_p\}}_{\Omega,(r_p)}$, $\chi \in \mathcal{D}^{\{M_p\}}(\Omega)$ and an ultradifferential operator of $\{M_p\}$ type such that $f=P(D)(u*f)+\chi*f$. Thus $f\in\mathcal{S}'^*(\mathbb{R}^d)$.
\end{proof}

Our next concern is to define the test function spaces $\mathcal{O}_C^*(\RR^d)$ corresponding to the spaces $\mathcal{O}_C'^*(\RR^d)$. We first define for every $m,h>0$ the $(B)$-spaces $$\mathcal{O}^{M_p}_{C,m,h}(\RR^d)=\left\{\varphi\in C^\infty\left(\RR^d\right)\bigg|\,\|\varphi\|_{m,h}=\left(\sum_{\alpha\in \mathbb{N}^d}\frac{m^{2|\alpha|}}{M_{\alpha}^2}\left\|D^{\alpha}\varphi e^{-M(h|\cdot|)}\right\|^2_{L^2}\right)^{1/2}<\infty\right\}.$$ Observe that for $m_1\leq m_2$ we have the continuous inclusion $\mathcal{O}^{M_p}_{C,m_2,h}(\RR^d)\rightarrow\mathcal{O}^{M_p}_{C,m_1,h}(\RR^d)$ and for $h_1\leq h_2$ the inclusion $\mathcal{O}^{M_p}_{C,m,h_1}(\RR^d)\rightarrow\mathcal{O}^{M_p}_{C,m,h_2}(\RR^d)$ is also continuous. As locally convex spaces (l.c.s.) we define
\beqs
\mathcal{O}_{C,h}^{(M_p)}(\RR^d)=\lim_{\substack{\longleftarrow\\ m\rightarrow \infty}}\mathcal{O}^{M_p}_{C,m,h}(\RR^d)&,&\,\, \mathcal{O}_C^{(M_p)}(\RR^d)=\lim_{\substack{\longrightarrow\\ h\rightarrow \infty}}\mathcal{O}^{(M_p)}_{C,h}(\RR^d);\\
\mathcal{O}_{C,h}^{\{M_p\}}(\RR^d)=\lim_{\substack{\longrightarrow\\ m\rightarrow 0}}\mathcal{O}^{M_p}_{C,m,h}(\RR^d)&,& \,\,\mathcal{O}_C^{\{M_p\}}(\RR^d)=\lim_{\substack{\longleftarrow\\ h\rightarrow 0}}\mathcal{O}^{\{M_p\}}_{C,h}(\RR^d).
\eeqs
Note that $\mathcal{O}_{C,h}^{(M_p)}(\RR^d)$ is an $(F)$-space and since all inclusions $\mathcal{O}_{C,h}^{(M_p)}(\RR^d)\rightarrow\EE^{(M_p)}(\RR^d)$ are continuous (by the Sobolev imbedding theorem), $\mathcal{O}_C^{(M_p)}(\RR^d)$ is indeed a (Hausdorff) l.c.s.. Moreover, as an inductive limit of barreled and bornological spaces, $\mathcal{O}_C^{(M_p)}(\RR^d)$ is barreled and bornological as well. Also $\mathcal{O}_{C,h}^{\{M_p\}}(\RR^d)$ is a (Hausdorff) l.c.s., since all inclusions $\mathcal{O}^{M_p}_{C,m,h}(\RR^d)\rightarrow\EE^{\{M_p\}}(\RR^d)$ are continuous (again by the Sobolev embedding theorem). Hence $\mathcal{O}_C^{\{M_p\}}(\RR^d)$ is indeed a (Hausdorff) l.c.s.. Moreover, $\mathcal{O}_{C,h}^{\{M_p\}}(\RR^d)$ is barreled and bornological $(DF)$-space, as the inductive limit of $(B)$-spaces. By these considerations it also follows that $\mathcal{O}_C^*(\RR^d)$ is continuously injected into $\EE^*(\RR^d)$. One easily verifies that for each $h>0$, $\SSS^{(M_p)}(\RR^d)$ is continuously injected into $\mathcal{O}_{C,h}^{(M_p)}(\RR^d)$ ($\SSS^{\{M_p\}}(\RR^d)$ is continuously injected into $\mathcal{O}_{C,h}^{\{M_p\}}(\RR^d)$). Moreover, one can also prove (by using cutoff functions) that for each $h>0$, $\DD^{(M_p)}(\RR^d)$ is sequentially dense in $\mathcal{O}_{C,h}^{(M_p)}(\RR^d)$, ($\DD^{\{M_p\}}(\RR^d)$ is sequentially dense in $\mathcal{O}_{C,h}^{\{M_p\}}(\RR^d)$). Hence, $\SSS^*(\RR^d)$ is continuously and densely injected into $\mathcal{O}_C^*(\RR^d)$. Consequently, the dual $\left(\mathcal{O}_C^*(\RR^d)\right)'$ can be regarded as vector subspace of $\SSS'^*(\RR^d)$.

 We will prove that the dual of $\mathcal{O}^*_C(\RR^d)$ is equal, as a set, to $\mathcal{O}'^*_C(\RR^d)$ (the general idea is similar to the one used by Komatsu \cite{Komatsu1}, p. 79). To do this, we need several additional spaces.

 For $m,h>0$ define

\beqs
Y_{m,h}&=&\Bigg\{(\psi_{\alpha})_{\alpha\in\NN^d}\bigg|\, e^{-M(h|\cdot|)}\psi_{\alpha}\in L^2(\RR^d),\\
&{}&\quad\left\|(\psi_{\alpha})_{\alpha}\right\|_{Y_{m,h}} =\left(\sum_{\alpha\in\NN^d}\frac{m^{2|\alpha|}\left\|e^{-M(h|\cdot|)}\psi_{\alpha}\right\|_{L^2}^2} {M_{\alpha}^2}\right)^{1/2}<\infty \Bigg\}.
\eeqs
One easily verifies that $Y_{m,h}$ is a $(B)$-space, with the norm $\|\cdot\|_{Y_{m,h}}$.\\

 Let $\tilde{U}$ be the disjoint union of countable number of copies of $\RR^d$, one for each $\alpha\in\NN^d$, i.e., $\tilde{U}=\bigsqcup_{\alpha\in\NN^d}\RR^d_{\alpha}$. Equip $\tilde{U}$ with the disjoint union topology. Then $\tilde{U}$ is a Hausdorff locally compact space. Moreover every open set in $\tilde{U}$ is $\sigma$-compact. On each $\RR^d_{\alpha}$ we define the Radon measure $\nu_{\alpha}$ by $d\nu_{\alpha}=e^{-2M(h|x|)}dx$. One can define a Borel measure $\mu_m$ on $\tilde{U}$ by $\ds\mu_m(E)=\sum_{\alpha}\frac{m^{2|\alpha|}}{M_{\alpha}^2}\nu_{\alpha}\left(E\cap \RR^d_{\alpha}\right)$, for $E$ a Borel subset of $\tilde{U}$. It is obviously locally finite, $\sigma$-finite and $\mu_m(\tilde{K})<\infty$ for every compact subset $\tilde{K}$ of $\tilde{U}$. By the properties of $\tilde{U}$ described above, $\mu_m$ is regular (both inner and outer regular). We obtain that $\mu_m$ is a Radon measure. To every $(\psi_{\alpha})_{\alpha}\in Y_{m,h}$ there corresponds an element $\chi\in L^2(\tilde{U},\mu_m)$ defined by $\chi_{|\RR^d_{\alpha}}=\psi_{\alpha}$. One easily verifies that the mapping $(\psi_{\alpha})_{\alpha}\mapsto \chi$, $Y_{m,h}\rightarrow L^2(\tilde{U},\mu_m)$ is an isometry, that is, $Y_{m,h}$ can be identified with $L^2(\tilde{U},\mu_m)$. Also, observe that $\mathcal{O}^{M_p}_{C,m,h}(\RR^d)$ can be identified with a closed subspace of $Y_{m,h}$ via the mapping $\varphi\mapsto ((-D)^{\alpha}\varphi)_{\alpha}$; hence, it is a reflexive space as a closed subspace of a reflexive $(B)$-space. We obtain that the linking mappings in $\ds \mathcal{O}_{C,h}^{(M_p)}(\RR^d)=\lim_{\substack{\longleftarrow\\ m\rightarrow \infty}}\mathcal{O}^{M_p}_{C,m,h}(\RR^d)$ and $\ds\mathcal{O}_{C,h}^{\{M_p\}}(\RR^d)=\lim_{\substack{\longrightarrow\\ m\rightarrow 0}}\mathcal{O}^{M_p}_{C,m,h}(\RR^d)$ are weakly compact, whence $\mathcal{O}_{C,h}^{(M_p)}(\mathbb{R}^{d})$ is an $(FS^*)$-space and $\mathcal{O}_{C,h}^{\{M_p\}}(\RR^d)$ is a $(DFS^*)$-space. In particular they are both reflexive and the inductive limit $\ds\mathcal{O}_{C,h}^{\{M_p\}}(\RR^d)=\lim_{\substack{\longrightarrow\\ m\rightarrow 0}}\mathcal{O}^{M_p}_{C,m,h}(\RR^d)$ is regular.

\begin{te}\label{1113} We have that 
$T\in\DD'^*(\RR^d)$ belongs to $\left(\mathcal{O}^*_C(\RR^d)\right)'$ if and only if
\begin{itemize}
\item[$(i)$] in the $(M_p)$ case, for every $h>0$ there exist $F_{\alpha,h}$, $\alpha\in\NN^d$, with $F_{\alpha,h}e^{M(h|\cdot|)}\in L^2(\RR^d)$, and $m>0$ such that

\beq\label{1115}
\sum_{\alpha}\frac{M_{\alpha}^2\left\|F_{\alpha,h}e^{M(h|\cdot|)}\right\|_{L^2}^2}{m^{2|\alpha|}}<\infty
\eeq
and the restriction of $T$ to $\mathcal{O}_{C,h}^{(M_p)}(\RR^d)$ is equal to $\sum_{\alpha} D^{\alpha}F_{\alpha,h}$, where the series is absolutely convergent in the strong dual of $\mathcal{O}_{C,h}^{(M_p)}(\RR^d)$;

\item[$(ii)$] in the $\{M_p\}$ case, there exist $h>0$ and $F_{\alpha,h}$, $\alpha\in\NN^d$, with $F_{\alpha,h}e^{M(h|\cdot|)}\in L^2(\RR^d)$, such that for every $m>0$ (\ref{1115}) holds and $T$ is equal to $\sum_{\alpha} D^{\alpha}F_{\alpha,h}$, where the series is absolutely convergent in the strong dual of $\mathcal{O}_{C}^{\{M_p\}}(\RR^d)$.
\end{itemize}
\end{te}

\begin{proof} We will consider first the Beurling case. Let $T\in \left(\mathcal{O}^{(M_p)}_C(\RR^d)\right)'$ and let $h>0$ be arbitrary but fixed. Denote by $T_h$ the restriction of $T$ on $\mathcal{O}_{C,h}^{(M_p)}(\RR^d)$. By the definition of the projective limit topology, it follows that there exists $m>0$ such that $T_h$ can be extended to a continuous linear functional on $\mathcal{O}^{M_p}_{C,m,h}(\RR^d)$. Denote this extension by $T_{h,1}$. Extend $T_{h,1}$, by the Hahn-Banach theorem, to a continuous linear functional $T_{h,2}$ on $Y_{m,h}$. Since $Y_{m,h}$ is isometric to $L^2(\tilde{U},\mu_m)$, there exists $g\in L^2(\tilde{U},\mu_m)$ such that $\ds T_{2,h}\left((\psi_{\alpha})_{\alpha}\right)=\int_{\tilde{U}}(\psi_{\alpha})_{\alpha} g d\mu_m$. Let $\ds F_{\alpha,h}=\frac{m^{2|\alpha|}}{M_{\alpha}^2}g_{|\RR^d_{\alpha}}e^{-2M(h|\cdot|)}$, $\alpha\in\NN^d$. Then, obviously $e^{M(h|\cdot|)}F_{\alpha,h}\in L^2(\RR^d)$ and $\ds \sum_{\alpha}\frac{M_{\alpha}^2\left\|F_{\alpha,h}e^{M(h|\cdot|)}\right\|_{L^2}^2}{m^{2|\alpha|}}= \|g\|^2_{L^2(\tilde{U},\mu_m)} <\infty$. For $\varphi\in\mathcal{O}^{(M_p)}_{C,h}(\RR^d)$,
\beqs
\langle T,\varphi\rangle=T_{h,2}\left(((-D)^{\alpha}\varphi)_{\alpha}\right)= \sum_{\alpha}\int_{\RR^d}F_{\alpha,h}(x)(-D)^{\alpha}\varphi(x) dx=\sum_{\alpha}\langle D^{\alpha}F_{\alpha,h},\varphi\rangle.
\eeqs
Moreover, one easily verifies that the series $\sum_{\alpha} D^{\alpha}F_{\alpha,h}$ is absolutely convergent in the strong dual of $\mathcal{O}_{C,h}^{(M_p)}(\RR^d)$.

Conversely, let $T\in \DD'^{(M_p)}(\RR^d)$ be as in $(i)$. Let $h>0$ be arbitrary but fixed. One easily verifies that $T$ is continuous functional on $\DD^{(M_p)}(\RR^d)$ supplied with the topology induced by $\mathcal{O}_{C,h}^{(M_p)}(\RR^d)$. Since $\DD^{(M_p)}(\RR^d)$ is dense in $\mathcal{O}_{C,h}^{(M_p)}(\RR^d)$ we obtain the conclusion in $(i)$.

 Next, we consider the Roumieu case. Let $T\in \left(\mathcal{O}^{\{M_p\}}_C(\RR^d)\right)'$. By the definition of the projective limit topology it follows that there exists $h>0$ such that $T$ can be extended to a continuous linear functional $T_1$ on $\mathcal{O}_{C,h}^{\{M_p\}}(\RR^d)$. For brevity in notation, set $X_{m,h}=\mathcal{O}^{M_p}_{C,m,h}(\RR^d)$ and $Z_{m,h}=Y_{m,h}/X_{m,h}$. Since the spaces $Y_{m,h}$ are reflexive, so are $X_{m,h}$ and $Z_{m,h}$ as closed subspaces and quotient spaces of reflexive $(B)$-spaces respectively. Moreover, observe that for $m_1<m_2$ we have $X_{m_1,h}\cap Y_{m_2,h}=X_{m_2,h}$. Hence we have the following injective inductive sequence of short topologically exact sequences of $(B)$-spaces:
\begin{center}
\begin{tikzpicture}[normal line/.style={->}]
\matrix (m) [matrix of math nodes, row sep=3em,
column sep=2.5em, text height=1.5ex, text depth=0.25ex]
{0 & X_{1,h} & Y_{1,h} & Z_{1,h} & 0 \\
0 & X_{1/2,h} & Y_{1/2,h} & Z_{1/2,h} & 0 \\
0 & X_{1/3,h} & Y_{1/3,h} & Z_{1/3,h} & 0\\
{} & \vdots & \vdots & \vdots & {} \\ };
\path[normal line]
(m-1-1) edge (m-1-2)
(m-1-2) edge (m-1-3)
        edge (m-2-2)
(m-1-3) edge (m-1-4)
        edge node [auto] {$\iota_{1,1/2}$} (m-2-3)
(m-1-4) edge (m-1-5)
        edge (m-2-4)

(m-2-1) edge (m-2-2)
(m-2-2) edge (m-2-3)
        edge (m-3-2)
(m-2-3) edge (m-2-4)
        edge node [auto] {$\iota_{1/2,1/3}$} (m-3-3)
(m-2-4) edge (m-2-5)
        edge (m-3-4)

(m-3-1) edge (m-3-2)
(m-3-2) edge (m-3-3)
        edge (m-4-2)
(m-3-3) edge (m-3-4)
        edge node [auto] {$\iota_{1/3,1/4}$} (m-4-3)
(m-3-4) edge (m-3-5)
        edge (m-4-4);
\end{tikzpicture}
\end{center}
where every vertical line is a weakly compact injective inductive sequence of $(B)$-spaces (since $X_{m,h}$, $Y_{m,h}$, $Z_{m,h}$ are reflexive $(B)$-spaces). The dual Mittag-Leffler lemma (see \cite[Lem. 1.4]{Komatsu1}) yields the short topologically exact sequence:
\beqs
0\longleftarrow\Big(\lim_{\substack{\longrightarrow\\ m\rightarrow 0}}X_{m,h}\Big)' \longleftarrow\Big(\lim_{\substack{\longrightarrow\\ m\rightarrow 0}}Y_{m,h}\Big)' \longleftarrow\Big(\lim_{\substack{\longrightarrow\\ m\rightarrow 0}}Z_{m,h}\Big)'\longleftarrow 0.
\eeqs
Since $(X_{m,h})_m$, $(Y_{m,h})_m$ and $(Z_{m,h})_m$ are weakly compact injective inductive sequences, hence regular, we have the following isomorphisms of l.c.s. $\ds\Big(\lim_{\substack{\longrightarrow\\ m\rightarrow 0}}X_{m,h}\Big)'=\lim_{\substack{\longleftarrow\\ m\rightarrow 0}}X'_{m,h}$, $\ds \Big(\lim_{\substack{\longrightarrow\\ m\rightarrow 0}}Y_{m,h}\Big)'=\lim_{\substack{\longleftarrow\\ m\rightarrow 0}}Y'_{m,h}$, and $\ds \Big(\lim_{\substack{\longrightarrow\\ m\rightarrow 0}}Z_{m,h}\Big)'=\lim_{\substack{\longleftarrow\\ m\rightarrow 0}}Z'_{m,h}$, from which we obtain the short topologically exact sequence
\beqs
0 \longleftarrow \lim_{\substack{\longleftarrow\\ m\rightarrow 0}}X'_{m,h} \longleftarrow\lim_{\substack{\longleftarrow\\ m\rightarrow 0}}Y'_{m,h} \longleftarrow\lim_{\substack{\longleftarrow\\ m\rightarrow 0}}Z'_{m,h} \longleftarrow 0.
\eeqs
Hence, there exists $\ds T_2\in\lim_{\substack{\longleftarrow\\ m\rightarrow 0}}Y'_{m,h}$ whose restriction to $\ds\mathcal{O}_{C,h}^{\{M_p\}}=\lim_{\substack{\longrightarrow\\ m\rightarrow 0}}X_{m,h}$ is $T_1$. Now observe the projective sequence
\beqs
Y'_{1,h}\xleftarrow{{}^t\iota_{1,1/2}} Y'_{1/2,h}\xleftarrow{{}^t\iota_{1/2,1/3}} Y'_{1/3,h}\xleftarrow{{}^t\iota_{1/3,1/4}}\ldots
\eeqs
where ${}^t\iota_{1/n,1/(n+1)}$ is the transposed mapping of the inclusion $\iota_{1/n,1/(n+1)}$. One easily verifies that ${}^t\iota_{1/n,1/(n+1)}: Y'_{1/(n+1),h}\rightarrow Y'_{1/n,h}$ is given by $\ds (\psi_{\alpha})_{\alpha}\mapsto \left(\frac{n^{2|\alpha|}}{(n+1)^{2|\alpha|}}\psi_{\alpha}\right)_{\alpha}$. By definition, the projective limit $\ds \lim_{\substack{\longleftarrow\\ m\rightarrow 0}}Y'_{m,h}$ is the subspace of $\prod_n Y'_{1/n,h}$ consisting of all elements $\left((\psi^{(k)}_{\alpha})_{\alpha}\right)_k\in \prod_n Y'_{1/n,h}$ such that for all $t,j\in\ZZ_+$, $t<j$, ${}^t\iota_{1/t,1/j}\left((\psi^{(j)}_{\alpha})_{\alpha}\right)=(\psi^{(t)}_{\alpha})_{\alpha}$ (where ${}^t\iota_{1/t,1/j}={}^t\iota_{1/t,1/(t+1)}\circ...\circ{}^t\iota_{1/(j-1),1/j}$). Hence, if we set $(\psi_{\alpha})_{\alpha}=(\psi^{(1)}_{\alpha})_{\alpha}$, then $\ds L^2(\tilde{U},\mu_{1/k})\ni(\psi^{(k)}_{\alpha})_{\alpha}=\left(k^{2|\alpha|}\psi_{\alpha}\right)_{\alpha}$ for all $k\in\ZZ_+$. In other words, we can identify $\ds \lim_{\substack{\longleftarrow\\ m\rightarrow 0}}Y'_{m,h}$ with the space of all $(\psi_{\alpha})_{\alpha}$ such that for every $s>0$, $\ds \left(\sum_{\alpha}\frac{s^{2|\alpha|}}{M_{\alpha}^2}\left\|\psi_{\alpha} e^{-M(h|\cdot|)}\right\|_{L^2(\RR^d)}^2\right)^{1/2}<\infty$. Since $\ds T_2\in \lim_{\substack{\longleftarrow\\ m\rightarrow \infty}} Y'_{1/m,h}$, there exists such $(\psi_{\alpha})_{\alpha}$ such that, for $m\in\ZZ_+$ and $(\chi_{\alpha})_{\alpha}\in Y_{1/m,h}$, we have $\ds T_2\left((\chi_{\alpha})_{\alpha}\right)=\sum_{\alpha}\int_{\RR^d_{\alpha}} m^{2|\alpha|}\psi_{\alpha}\chi_{\alpha} d\mu_{1/m}$. Set $\ds F_{\alpha,h}=\frac{\psi_{\alpha}e^{-2M(h|\cdot|)}}{M_{\alpha}^2}$. Hence, for every $s>0$, $\ds \left(\sum_{\alpha}s^{2|\alpha|}M_{\alpha}^2\left\|F_{\alpha,h} e^{M(h|\cdot|)}\right\|_{L^2(\RR^d)}^2\right)^{1/2}<\infty$. Moreover, for $\varphi\in \mathcal{O}_{C}^{\{M_p\}}(\RR^d)$, there exists $m\in \ZZ_+$ such that $\varphi\in \mathcal{O}_{C,1/m,h}^{M_p}(\RR^d)$. We have
\beqs
\langle T,\varphi \rangle=\sum_{\alpha}\int_{\RR^d} F_{\alpha,h}(x)(-D)^{\alpha}\varphi(x) dx=\sum_{\alpha}\langle D^{\alpha}F_{\alpha,h},\varphi\rangle.
\eeqs
Since $\mathcal{O}_{C,h}^{\{M_p\}}(\RR^d)$ is a $(DFS^*)$-space its strong dual $\left(\mathcal{O}_{C,h}^{\{M_p\}}(\RR^d)\right)'_b$ is complete. If $B$ is a bounded subset of $\mathcal{O}_{C,h}^{\{M_p\}}(\RR^d)$ then it must belong to some $\mathcal{O}^{M_p}_{C,m,h}(\RR^d)$ and be bounded there (the inductive limit $\ds\mathcal{O}_{C,h}^{\{M_p\}}(\RR^d)=\lim_{\substack{\longrightarrow\\ m\rightarrow 0}}\mathcal{O}^{M_p}_{C,m,h}(\RR^d)$ is regular). One easily verifies that $\sum_{\alpha}\sup_{\varphi\in B}\left|\langle D^{\alpha} F_{\alpha, h},\varphi\rangle\right|<\infty$, hence $\sum_{\alpha} D^{\alpha} F_{\alpha,h}$ converges absolutely in $\left(\mathcal{O}_{C,h}^{\{M_p\}}(\RR^d)\right)'_b$. Since $\mathcal{O}_{C}^{\{M_p\}}(\RR^d)$ is continuously and densely injected into $\mathcal{O}_{C,h}^{\{M_p\}}(\RR^d)$ ($\DD^{\{M_p\}}\left(\RR^d\right)$ is dense in these spaces) it follows that the series $\sum_{\alpha} D^{\alpha} F_{\alpha,h}$ converges absolutely in the strong dual of $\mathcal{O}_C^{\{M_p\}}(\RR^d)$.\\

 Conversely, let $T\in \DD'^{\{M_p\}}(\RR^d)$ be as in $(ii)$. Then it is easy to verify that $T$ is a continuous functional on $\DD^{\{M_p\}}(\RR^d)$ when we regard it as subspace of $\mathcal{O}_{C,h}^{\{M_p\}}(\RR^d)$, where $h$ is the one from the condition in $(ii)$. Since $\DD^{\{M_p\}}(\RR^d)$ is dense in $\mathcal{O}_{C,h}^{\{M_p\}}(\RR^d)$, $T$ is continuous functional on $\mathcal{O}_{C,h}^{\{M_p\}}(\RR^d)$ and hence on $\mathcal{O}_C^{\{M_p\}}(\RR^d)$.
\end{proof}

The next theorem realizes our first goal in the article: we may identify $\mathcal{O}'^*_C(\RR^d)$ with the topological dual of $\mathcal{O}^*_C(\RR^d)$. We make use below of the following elementary inequality
\beq\label{inqfor_ass_fun}
e^{M(\rho+\lambda)}\leq 2 e^{M(2\rho)}e^{M(2\lambda)},\,\, \rho,\lambda>0,
\eeq
which is a consequence of the following observation
\beqs
\frac{(\lambda+\rho)^p}{M_p}\leq \frac{2^p\rho^p}{M_p}+\frac{2^p\lambda^p}{M_p}\leq e^{M(2\rho)}+e^{M(2\lambda)}\leq 2e^{M(2\rho)}e^{M(2\lambda)},
\eeqs
where the last inequality holds because the associated function is nonnegative.

\begin{te}
The dual of $\mathcal{O}_C^*(\RR^d)$ is algebraically isomorphic to $\mathcal{O}'^*_C(\RR^d)$.
\end{te}

\begin{proof} Let $T\in \left(\mathcal{O}_C^*(\RR^d)\right)'\subseteq \SSS'^*(\RR^d)$. To prove that $T\in \mathcal{O}'^*_C(\RR^d)$, by \cite[Prop. 2]{PBD}, it is enough to prove that $T*\varphi\in\SSS^*(\RR^d)$ for each $\varphi\in\DD^*(\RR^d)$. We consider first the $(M_p)$ case. Let $\varphi\in\DD^{(M_p)}(\RR^d)$ and let $m>0$ be arbitrary but fixed. By Theorem \ref{1113}, for $h\geq 2m$, there exist $m_1>0$ and $F_{\alpha,h}$, $\alpha\in\NN^d$, such that (\ref{1115}) holds. Take $m_2>0$ such that $m_2\geq Hm$ and $H/m_2\leq 1/(2m_1)$. For this $m_2$ there exists $C'>0$ such that $\left|D^{\beta}\varphi(x)\right|\leq C' M_{\beta}/ m_2^{|\beta|}$. Using the inequality (\ref{inqfor_ass_fun}), for $x,t\in\RR^d$ one obtains $e^{M(m|x|)}\leq 2 e^{M(h|x-t|)}e^{M(h|t|)}$. Then, we have\\
\\
$\ds \frac{m^{|\beta|}\left|D^{\beta}(T*\varphi)(x)\right|e^{M(m|x|)}}{M_{\beta}}$
\beqs
&\leq& \frac{m^{|\beta|}e^{M(m|x|)}}{M_{\beta}}\sum_{\alpha}\left\|F_{\alpha,h}e^{M(h|\cdot|)}\right\|_{L^2} \left(\int_{\RR^d}\left|D^{\alpha+\beta}\varphi(x-t)\right|^2 e^{-2M(h|t|)}dt\right)^{1/2}\\
&=&\frac{m^{|\beta|}}{M_{\beta}}\sum_{\alpha}\left\|F_{\alpha,h}e^{M(h|\cdot|)}\right\|_{L^2} \left(\int_{\RR^d}\left|D^{\alpha+\beta}\varphi(x-t)\right|^2 e^{2M(m|x|)}e^{-2M(h|t|)}dt\right)^{1/2}\\
&\leq& 2\frac{m^{|\beta|}}{M_{\beta}}\sum_{\alpha}\left\|F_{\alpha,h}e^{M(h|\cdot|)}\right\|_{L^2} \left(\int_{\RR^d}\left|D^{\alpha+\beta}\varphi(x-t)\right|^2 e^{2M(h|x-t|)}dt\right)^{1/2}\\
&\leq& C_1\sum_{\alpha}\frac{m^{|\beta|}M_{\alpha+\beta}}{M_{\beta}m_2^{|\alpha|+|\beta|}} \left\|F_{\alpha,h}e^{M(h|\cdot|)}\right\|_{L^2}\leq C_2\left(\frac{Hm}{m_2}\right)^{|\beta|} \sum_{\alpha}\frac{1}{2^{|\alpha|}}\leq C.
\eeqs
Since $m>0$ is arbitrary, $T*\varphi\in \SSS^{(M_p)}(\RR^d)$ and we obtain $T\in \mathcal{O}'^{(M_p)}_C(\RR^d)$. In the $\{M_p\}$ case, there exist $m_2,C'>0$ such that $\left|D^{\beta}\varphi(x)\right|\leq C' M_{\beta}/ m_2^{|\beta|}$. Also, for $T$ there exist $h>0$ and $F_{\alpha,h}$, $\alpha\in\NN^d$ such that (\ref{1115}) holds for every $m_1>0$. Take $m>0$ such that $m\leq h/2$ and $m\leq m_2/H$ and take $m_1>0$ such that $1/(2m_1)\geq H/m_2$. Then the same calculations as above give $\ds \frac{m^{|\beta|}\left|D^{\beta}(T*\varphi)(x)\right|e^{M(m|x|)}}{M_{\beta}}\leq C$, that is, $T*\varphi\in \SSS^{\{M_p\}}(\RR^d)$. We obtain $T\in\mathcal{O}'^{\{M_p\}}_C(\RR^d)$.

Conversely, let $T\in \mathcal{O}'^*_C(\RR^d)$. In the $(M_p)$ case, by \cite[Prop. 2]{PBD} for every $r>0$ there exist an ultradifferential operator $P(D)$ of class $(M_p)$ and $F_1,F_2\in L^{\infty}\left(\RR^d\right)$ such that $T=P(D)F_1+ F_2$ and $\left\|e^{M(r|\cdot|)}(F_1+F_2)\right\|_{L^{\infty}(\RR^d)}\leq C$. Let $h>0$ be arbitrary but fixed. Choose such a representation of $T$ for $r\geq H^2 h$. For simplicity, we assume that $F_2=0$ and set $F=F_1$. The general case is proved analogously. Let $P(D)=\sum_{\alpha}c_{\alpha} D^{\alpha}$. Then, there exist $c,L\geq 1$ such that $|c_{\alpha}|\leq cL^{|\alpha|}/M_{\alpha}$. Let $F_{\alpha}=c_{\alpha}F$. By \cite[Prop. 3.6]{Komatsu1} we have $e^{4M(h|x|)}\leq C_1 e^{M(H^2h|x|)}\leq C_1 e^{M(r|x|)}$. We obtain
\beqs
\sum_{\alpha} \frac{M_{\alpha}^2}{(2L)^{2|\alpha|}}\left\|e^{M(h|\cdot|)}F_{\alpha}\right\|_{L^2}^2\leq C_1\sum_{\alpha} \frac{M_{\alpha}^2}{(2L)^{2|\alpha|}}|c_{\alpha}|^2\left\|e^{M(r|\cdot|)}F\right\|_{L^{\infty}}^2 \left\| e^{-M(h|\cdot|)}\right\|_{L^2}^2<\infty.
\eeqs
So, for the chosen $h>0$, (\ref{1115}) holds with $m=2L$. Since $T=\sum_{\alpha}D^{\alpha} F_{\alpha}$, by Theorem \ref{1113} we have $T\in\left(\mathcal{O}^{(M_p)}_C(\RR^d)\right)'$. In the $\{M_p\}$ case there exist $r>0$, an ultradifferential operator $P(D)$ of class $\{M_p\}$ and $L^{\infty}$-functions $F_1$ and $F_2$ such that $T=P(D)F_1+F_2$ and $\left\|e^{M(r|\cdot|)}(F_1+F_2)\right\|_{L^{\infty}(\RR^d)}\leq C$. For simplicity, we assume that $F_2=0$ and set $F=F_1$. The general case is proved analogously. Since $P(D)=\sum_{\alpha}c_{\alpha} D^{\alpha}$ is of class $\{M_p\}$ for every $L>0$ there exists $c>0$ such that $|c_{\alpha}|\leq cL^{|\alpha|}/M_{\alpha}$. Set $F_{\alpha}=c_{\alpha}F$. Take $h\leq r/H^2$. Let $m>0$ be arbitrary but fixed. Then there exists $c>0$ such that $|c_{\alpha}|\leq cm^{|\alpha|}/(2^{|\alpha|}M_{\alpha})$. Similarly as above $\sum_{\alpha} M_{\alpha}^2 m^{-2|\alpha|}\left\|e^{M(h|\cdot|)}F_{\alpha}\right\|_{L^2}^2<\infty$. Since $T=\sum_{\alpha} D^{\alpha} F_{\alpha}$, by Theorem \ref{1113}, we have $T\in\left(\mathcal{O}^{\{M_p\}}_C(\RR^d)\right)'$.
\end{proof}

It would be also interesting to study the relation between the strong dual topology on $\mathcal{O}'^*_C(\mathbb{R}^{d})$ provided by the duality $\left\langle \mathcal{O}^*_C(\RR^d),\mathcal{O}'^*_C(\RR^d)\right\rangle$  and the induced one on  $\mathcal{O}'^*_C(\mathbb{R}^{d})$  as a (closed) subspace of $\mathcal{L}_{b}(\mathcal{S}^*(\mathbb{R}^{d}),\mathcal{S}^{*}(\mathbb{R}^{d}))$ (the latter topology was considered in \cite{PBD}).


\section{Translation-invariant Banach spaces of tempered ultradistributions}\label{UTIB}

We employ the notation $T_h$ for the translation operator $T_hg=g(\:\cdot\:+h),$ $h\in\mathbb R^d$. The symbol ``$\hookrightarrow $'' stands for a continuous and dense inclusion. In the rest of the article we are interested in translation-invariant $(B)$-spaces of ultradistributions satisfying the properties from the following definition.

\begin{de}\label{def E}
A $(B)$-space $E$ is said to be a translation-invariant $(B)$-space of tempered ultradistributions of class $\ast$ if it satisfies the following three axioms:

\begin{itemize}
    \item[(I)] $\mathcal{D}^*(\mathbb{R}^d)\hookrightarrow E\hookrightarrow \mathcal{D}'^*(\mathbb{R}^d)$.
    \item[(II)] $T_h:E\rightarrow E$ for every $h\in\RR^d$ (i.e., $E$ is translation-invariant).
\item [(III)] For any $g\in E$ there exist $C=C_g>0$ and $\tau=\tau_g>0$, (for every $\tau>0$ there exists $C=C_{g,\tau}>0$) such that $\|T_hg\|_E\leq C e^{M(\tau|h|)}$, $\forall h\in\RR^d$.
\end{itemize}
The weight function of $E$ is the function $\omega:\mathbb{R}^{d}\to (0,\infty)$ given by\footnote{By applying the closed graph theorem, the axioms (I) and (II) yield $T_h\in\mathcal{L}(E)$ for every $h\in\RR^d$.} $\omega(h):=\|T_{-h}\|_{\mathcal{L}(E)}$.
\end{de}

\smallskip

Throughout the rest of the article we assume that $E$ is a translation-invariant $(B)$-space of tempered ultradistributions. It is clear that $\omega(0)=1$ and that $\ln \omega$ is a subadditive function. We will prove that $\omega$ is measurable and locally bounded; this allows us to associate to $E$ the Beurling algebra $L^{1}_{\omega}$ \cite{beurling}, namely, the Banach algebra of measurable functions $u$ such that $\|u\|_{1,\omega}:=\int_{\mathbb{R}^{d}}|u(x)| \:\omega(x) dx<\infty$. The next theorem collects a number of important properties of $E$.

\begin{te}\label{ttt11}
The following property hold for $E$ and $\omega$:
\begin{itemize}
\item [$(a)$] $\mathcal{S}^*(\mathbb{R}^d)\hookrightarrow E\hookrightarrow \mathcal{S}'^*(\mathbb{R}^d)$.
\item [$(b)$] For each $g\in E$, $\ds\lim_{h\to0}\|T_{h}g-g\|_{E}=0$. $($Hence the mapping $h\mapsto T_h g$ is continuous.$)$
\item [$(c)$] There are $\tau,C>0$ (for every $\tau>0$ there is $C=C_{\tau}>0$) such that
		$$\omega(h)\leq C e^{M(\tau |h|)}, \ \ \ \forall h\in\mathbb{R}^{d}.$$
\item [$(d)$] $E$ is separable and $\omega$ is measurable.
\item [$(e)$] The convolution mapping $\ast: \mathcal{S}^*(\mathbb{R}^d)\times \mathcal{S}^*(\mathbb{R}^d)\rightarrow \mathcal{S}^*(\mathbb{R}^d) $ extends to  $\ast: L^{1}_{\omega}\times E\rightarrow E$ and $E$ becomes a Banach module over the Beurling
algebra $L^{1}_{\omega}$, that is,

\begin{equation}\label{eqwconvolution1}
\|u\ast g\|_{E}\leq \|u\|_{1,\omega}\|g\|_{E}.
\end{equation}
 Furthermore, the bilinear mapping $\ast:\mathcal{S}^*(\mathbb{R}^d)\times
E\rightarrow E$ is continuous.
\item [$(f)$] Let $g\in E$ and let $\varphi\in \mathcal{S}^*(\mathbb{R}^d)$. Set $\ds \varphi_{\varepsilon}(x)=\varepsilon^{-d}\varphi
\left(x/ \varepsilon \right)$ and $c=\int_{\mathbb{R}^d}\varphi(x)dx$. Then, $\ds\lim_{\varepsilon\to 0^+} \|cg-\varphi_\varepsilon*g\|_E=0$.
\end{itemize}
Alternatively, in the $\{M_p\}$ case, the property $(c)$ is equivalent to:
\begin{itemize}
\item[$(\tilde{c})$] there exist $(l_p)\in\mathfrak{R}$ and $C>0$ such that $\omega(h)\leq C e^{N_{l_p}(|h|)}$, $\forall h\in\RR^d$.
\end{itemize}
\end{te}

\begin{proof}  The property $(b)$ follows directly from the axioms (I)--(III). For $(d)$, notice that (I) yields at once the separability of $E$. On the other hand, if $D$ is a countable and dense subset of the unit ball of $E$, we have $\omega(h)=\sup_{g\in D}\|T_{-h}g\|_{E}$, and so $(b)$ yields the measurability of $\omega$.

We now show $(c)$. In the $(M_p)$ case, consider the sets $E_{j,\nu}=\{g\in E\,|\, \|T_{h}g\|_E\leq j e^{M(\nu|h|)},\forall h \in \mathbb{R}^d\},$ $j,\nu\in \ZZ_+.$ Because of (III), $E=\bigcup_{j,\nu\in \ZZ_+}E_{j,\nu}$. Since $E_{j,\nu}=\bigcap_{h\in \RR^d}E_{j,\nu,h}$, where $E_{j,\nu,h}=\left\{g\in E\,|\, \|T_{h}g\|_E\leq j e^{M(\nu|h|)}\right\}$,  each of these sets is closed in $E$ by the continuity of $T_h$, and so are $E_{j,\nu}$. Now, a classical category argument gives the claim.
In the $\{M_p\}$ case, for fixed $\tau>0$, we consider the sets $E_{j}=\left\{g\in E\,|\, \|T_{h}g\|_E\leq j e^{M(\tau|h|)}\mbox \,\,\mbox{for all }\mbox{   } h \in \mathbb{R}^d\right\}, j \in\ZZ_+.$
Obviously $E=\bigcup_{j\in\ZZ_+} E_j$. Again the Baire category theorem yields the claim.

Let us prove that $(c)$ is equivalent to $(\tilde{c})$. Obviously $(\tilde{c})\Rightarrow (c)$. Conversely, define $F:[0,\infty)\rightarrow [0,\infty)$ as
\beqs
F (\rho)=\sup_{|h|\leq \rho}\sup_{\|g\|_E\leq 1}\ln_+\|T_h g\|_E.
\eeqs
One easily verifies that $F(\rho)$ is increasing and satisfies the conditions of Lemma \ref{ppp}. Hence there exists a subordinate function $\epsilon(\rho)$ and $C'>1$ such that $F(\rho)\leq M(\epsilon(\rho))+\ln C'$. Hence, we obtain $\ds \sup_{\|g\|_E\leq 1}\|T_h g\|_E\leq C' e^{M(\epsilon(|h|))}$. Now, \cite[Lem. 3.12]{Komatsu1} implies that there exists a sequence $\tilde{N}_p$ which satisfies $(M.1)$ such that $M(\epsilon(\rho))\leq \tilde{N}(\rho)$ as $\ds\frac{\tilde{N}_p M_{p-1}}{\tilde{N}_{p-1}M_p}\rightarrow \infty$ as $p\rightarrow \infty$. Set $\ds l'_p=\frac{\tilde{N}_p M_{p-1}}{\tilde{N}_{p-1}M_p}$. Take $(l_p)\in \mathfrak{R}$ such that $l_p\leq l'_p$, for all $p\in\ZZ_+$. Then
\beqs
\sup_{\|g\|_E\leq 1}\|T_h g\|_E\leq C' e^{\tilde{N}(|h|)}= C' \sup_{p\in\NN} \frac{|h|^p}{M_p\prod_{j=1}^p l'_j}\leq C' \sup_{p\in\NN} \frac{|h|^p}{M_p\prod_{j=1}^p l_j}= C' e^{N_{l_p}(|h|)},
\eeqs
whence $(\tilde{c})$ follows.

We now address the property $(a)$. We first prove the embedding $\mathcal{S}^*(\mathbb{R}^{d})\hookrightarrow E$. Since $\mathcal
D^*(\mathbb R^d)\hookrightarrow \mathcal S^*(\mathbb R^d),$ it is enough to prove that
$\mathcal{S}^*(\mathbb{R}^{d})$ is continuously injected into $E$. Let $\varphi \in \SSS^*(\RR^d)$. We use a special partition of unity:
\beqs
1=\sum_{m\in\mathbb{Z}^{d}} \psi(x-m), \ \ \ \psi\in\mathcal{D}^*_{[-1,1]^{d}}
\eeqs
and we get the representation $\varphi(x)=\sum_{m\in\mathbb{Z}^{d}}\psi(x-m)\varphi(x)$. We estimate each term in this
sum. Because of $(c)$, there exist constants $C>0$ and $\tau>0$ (for every $\tau>0$ there exists $C>0$) such that
\begin{equation}
\label{eq:1.1}
\|\varphi\: T_{-m}\psi\|_E\leq\frac{C}{e^{M(\tau|m|)}}\left\|e^{2M(\tau|m|)}\psi T_{m}\varphi\right\|_E.
\end{equation}
We need to prove that the multi-sequence of operators
$\left\{\rho_{m,\tau}\right\}_{m\in\mathbb{Z}^{d}}:\mathcal{S}^*(\RR^d)\rightarrow\mathcal{D}^*_{[-1,1]^d},$
defined as
\begin{equation}
\label{eq:1.2}
\rho_{m,\tau}(\varphi):=e^{2M(\tau|m|)}\psi T_{m}\varphi,
\end{equation}
is uniformly bounded on a fixed bounded subset of $\SSS^*(\RR^d)$, where $\tau>0$ in the $\{M_p\}$ case will be chosen later. Let $B$ be bounded set in $\mathcal{S}^*(\RR^d)$. Then for each $h>0$ (for some $h>0$)
\beq\label{11777}
\sup_{\varphi\in B}\sup_{\alpha\in\NN^d}\frac{h^{|\alpha|} \left\|e^{M(h|\cdot|)}D^{\alpha}\varphi\right\|_{L^{\infty}(\RR^d)}} {M_{\alpha}}<\infty.
\eeq
By \cite[Lem. 3.6]{Komatsu1} we have $e^{2M(\tau|m|)}\leq c_0 e^{M(H\tau|m|)}$ and hence
\beq\label{11177}
e^{2M(\tau|m|)}\leq 2c_0 e^{M(2H\tau|m+x|)}e^{M(2H\tau|x|)}\leq C_1e^{M(2H\tau|m+x|)},\,\,
\eeq
$ \forall x\in[-1,1]^d,\, \forall m\in\ZZ^d.$
In the $(M_p)$ case let $h_1>0$ be arbitrary but fixed. Choose $h>0$ such that $h\geq 2h_1$ and $h\geq 2H\tau$. For this $h$, (\ref{11777}) holds and by (\ref{11177}) and the fact that $\psi\in \DD^{(M_p)}_{[-1,1]^d}$, one readily verifies that
\beq\label{nr111}
\frac{h_1^{|\alpha|} \left|D^{\alpha}\left(\psi(x) T_m\varphi(x)\right)\right|}{ M_\alpha}\leq \frac{C'}{e^{2M(\tau|m|)}},\,\, \mbox{for all } \varphi\in B,\, m\in\ZZ^d.
\eeq
Hence $\{\rho_{m,\tau}|\, m\in\ZZ^d\}$ is uniformly bounded on $B$. In the $\{M_p\}$ case, there exist $\tilde{h},\tilde{C}>0$ such that $\left|D^{\alpha}\psi(x)\right|\leq \tilde{C}M_{\alpha}/ \tilde{h}^{|\alpha|}$. For the $h$ for which (\ref{11777}) holds choose $h_1>0$ such that $h_1\leq \min\{h/2,\tilde{h}/2\}$ and choose $\tau\leq h/(2H)$. Then, by using (\ref{11177}), similarly as in the $(M_p)$ case, we obtain (\ref{nr111}), namely, $\{\rho_{m,\tau}|\, m\in\ZZ^d\}$ is uniformly bounded on $B$. By (I), the mapping $\mathcal{D}^*_{[-1,1]^d}\rightarrow E$ is continuous; hence, $\|\rho_m(\varphi)\|_E\leq C_2$, for all $\varphi\in B$, $m\in\mathbb{Z}^{d}$.

In view of (\ref{eq:1.1}) and the later fact, we have that
$\left\{\sum_{|m|\leq N} \varphi T_{-m}\psi\right\}_{N=0}^{\infty}$ is a Cauchy sequence in $E$ whose limit is
$\varphi\in E$; one also obtains
$\|\varphi\|_E \leq C$ for all $\varphi\in B$. We have just proved that the inclusion $\mathcal{S}^*(\mathbb{R}^{d})\rightarrow E$ maps bounded sets into bounded sets and,
since $\mathcal{S}^*(\mathbb{R}^{d})$ is bornological, it is continuous.

We now address $E\subseteq\mathcal{S}'^*(\mathbb{R}^{d})$ and the continuity of the inclusion mapping. Let $g\in E$. We employ Proposition \ref{S'}. Let $B$ be a bounded set in $\mathcal{D}^*(\mathbb{R}^{d})$. The inclusion $E\hookrightarrow\mathcal{D}'^*(\mathbb{R}^{d})$ yields the existence of a constant $D=D(B)$ such that $\left|\left\langle
g,\check{\phi}\right\rangle\right|\leq D\|g\|_{E}$ for all $g\in E$ and $\phi\in B$. Therefore, by $(c)$, there exist $\tau,C>0$ (for every $\tau>0$ there exists $C>0$) such that
\beqs
|(g\ast \phi)(h)|\leq D\|T_{h}g\|_{E}\leq CD\|g\|_{E}e^{M(\tau|h|)},
\eeqs
for all  $g\in E$, $\phi\in B$, $h\in\mathbb{R}^{d}$. In the $(M_p)$ case, Proposition \ref{S'} implies that $E\subseteq \mathcal{S}'^{(M_p)}(\mathbb{R}^{d})$. In the $\{M_p\}$ case, the property $(\tilde{c})$, together with Proposition \ref{S'}, implies $E\subseteq \mathcal{S}'^{\{M_p\}}(\mathbb{R}^{d})$. Since $E\rightarrow \DD'^*\left(\RR^d\right)$ is continuous it has a closed graph, hence so does the inclusion $E\rightarrow \SSS'^*(\RR^d)$ ($\SSS'^*(\RR^d)$ is continuously injected into $\DD'^*\left(\RR^d\right)$). Since $\SSS'^{(M_p)}(\RR^d)$ is a $(DFS)$-space ($\SSS'^{\{M_p\}}(\RR^d)$ is an $(FS)$-space) it is a Pt\'{a}k space (cf. \cite[Sect. IV. 8, p. 162]{Sch}). Thus, the continuity of $E\rightarrow \SSS'^*(\RR^d)$ follows from the Pt\'{a}k closed graph theorem (cf. \cite[Thm. 8.5, p. 166]{Sch}). The proof of $(a)$ is complete.

We now show that $E$ is a Banach module over $L^{1}_{\omega}$.
Let $\varphi,\psi\in\DD^*(\RR^d)$ and denote $K=\mathrm{supp}\, \varphi$. We prove that \beq\label{1}
\|\varphi\ast \psi\|_E\leq \|\psi\|_E\int_{\mathbb{R}^d}|\varphi(x)|\:\omega(x)dx.
\eeq
The Riemann sums
\beqs
L_{\varepsilon}(\cdot)=\varepsilon^d\sum_{n\in\ZZ^d,\, \varepsilon n\in K}\varphi(\varepsilon n)\psi(\cdot-\varepsilon n)= \varepsilon^d\sum_{n\in\ZZ^d,\, \varepsilon n\in K}\varphi(\varepsilon n)T_{-\varepsilon n}\psi
\eeqs
converge to $\varphi*\psi$ in $\SSS^*(\RR^d)$ as $\varepsilon \rightarrow0^+$. By $(a)$ they also converge in $E$ to the same element, that is, $L_{\varepsilon}\rightarrow \varphi*\psi$ as $\varepsilon\rightarrow 0^+$ in $E$. Set $\omega_{\psi}(t)=\|T_{-t}\psi\|_E$. Then $\omega_{\psi}$ is continuous by $(b)$. Observe that
\beq\label{11}
\|L_{\varepsilon}\|_E\leq \sum_{y\in\ZZ^d,\, \varepsilon y\in K}|\varphi(\varepsilon y)|\|T_{-\varepsilon y}\psi\|_E \varepsilon^d=\sum_{y\in\ZZ^d,\, \varepsilon y\in K}|\varphi(\varepsilon y)|\omega_{\psi}(\varepsilon y) \varepsilon^d
\eeq
and the last term converges to $\ds \int_K |\varphi(y)|\omega_{\psi}(y)dy$. Since $\omega_{\psi}(t)=\|T_{-t}\psi\|_E\leq \|\psi\|_E \omega(t)$, if we let $\varepsilon\rightarrow 0^+$ in (\ref{11}) we obtain (\ref{1}). By using (I) and a standard density argument, the convolution can be extended to $\ast: L^{1}_{\omega}\times E\to E$ and (\ref{1}) leads to (\ref{eqwconvolution1}). The continuity of the convolution as a bilinear mapping $\mathcal{S}^*(\mathbb{R}^d)\times E\rightarrow E$ in the $(M_p)$ case is an easy consequence of (\ref{eqwconvolution1}). In the $\{M_p\}$ case, we can conclude separate continuity from (\ref{eqwconvolution1}), but then, \cite[Thm. 41.1, p. 421]{treves} implies the desired continuity. This shows (e).

Finally, if $g\in \mathcal{S}^*(\mathbb{R}^d)$ and $\varphi\in \mathcal{S}^*(\mathbb{R}^d)$, then, by property $(a)$ and (\ref{eqwconvolution1}), $\ds\lim_{\varepsilon\to 0^+} \|cg-\varphi_\varepsilon*g\|_E=0$. The general case of $(f)$, namely the case $g\in E$, can be established via a density argument.

\end{proof}

As done in $(e)$, one can also extend the convolution as a mapping $\ast: E\times L^{1}_{\omega}\to E$ and obviously $u\ast g=g\ast u$.

We now discuss some properties that automatically transfer to the dual space $E'$ by duality. Note that the property $(a)$ from Theorem \ref{ttt11}  implies the continuous injections
$\mathcal{S}^*(\mathbb{R}^d)\to E'\to \mathcal{S}'^*(\mathbb{R}^d)$. The condition (II) from Definition \ref{def E} remains valid for $E'$. We define the weight function of $E'$ as
$$\check{\omega}(h):=\|T_{-h}\|_{\mathcal{L}(E')}=\|T^{\top}_{h}\|_{\mathcal{L}(E')}=\omega(-h),$$
where one of the equalities follows from the well known bipolar theorem (cf. \cite[p. 160]{Sch}). Thus $(c)$ and $(\tilde{c})$ from Theorem \ref{ttt11} hold for the weight function $\check\omega$ of $E'$. In particular, the axiom (III) holds for $E'$. In general, however, $E'$ may fail to be a translation-invariant $(B)$-spaces of tempered ultradistributions because (I) may not be any longer true for it. Note also that $E'$ can be non-separable. In addition, the property $(b)$ from Theorem \ref{ttt11} may also fail for $E'$, but on the other hand it follows by duality that, given $f\in E'$,

 \begin{itemize}
        \item [$(b'')$] The mappings $\mathbb{R}^d\to E'$ given by $h\mapsto T_hf$ are continuous for the weak$^{\ast}$ topology.
    \end{itemize}
The associated Beurling algebra to $E'$ is $L^{1}_{\check{\omega}}$. We define the convolution $u\ast f=f\ast u$ of $f\in E'$ and $u\in L^{1}_{\check{\omega}}$ via transposition: $\left\langle u\ast f,g \right\rangle:= \left\langle f,\check {u}\ast g\right\rangle$, $g\in E$. In view of $(e)$ from Theorem \ref{ttt11}, this convolution is well defined because $\check {u}\in L^{1}_{\omega}$. It readily follows that $(e)$ holds when $E$ and $\omega$ are replaced by $E'$ and $\check{\omega}$; so $E'$ is a Banach module
over the Beurling algebra $L^{1}_{\check{\omega}}$, i.e., $\|u\ast f\|_{E'}\leq \|u\|_{1,\check{\omega}}\|f\|_{E'}$. Concerning the property $(f)$ from Theorem \ref{ttt11}, it may be no longer satisfied by $E'$.

In summary, $E'$ might not be as rich as $E$. We introduce the following space that enjoys better properties than $E'$ with respect to the translation group.

\begin{de}\label{goodspace}
The $(B)$-space $E'_{\ast}$ stands for $E'=L^{1}_{\check{\omega}}\ast E'$.
\end{de}

Note that $E'_{\ast}$ is a closed linear subspace of $E'$, due to the Cohen-Hewitt factorization theorem \cite[p. 178]{kisynski} and the fact that $L^{1}_{\check{\omega}}$ possesses bounded approximation unities. The ensuing theorem shows that $E'_{\ast}$ possesses many of the properties that $E'$ lacks. It also gives a characterization of $E'_{\ast}$ and tells us that the property (I) holds for $E'$ when $E$ is reflexive.

\begin{te}
\label{thgoodspace}
The $(B)$-space $E'_{\ast}$ satisfies:
\begin{itemize}
\item [$(i)$] $\mathcal{S}^*(\mathbb{R}^d)\to E'_{\ast}\to \mathcal{S}'^*(\mathbb{R}^d)$ and $E'_{\ast}$ a is Banach module over $L^{1}_{\check\omega}$.
\item [$(ii)$] The properties (II) from Definition \ref{def E} and $(b)$ and $(f)$ from Theorem \ref{ttt11} are valid when $E$ is replaced by $E'_{\ast}$.
\item [$(iii)$] $\ds E'_{\ast}=\left\{f\in E'|\, \lim_{h\to 0}\|T_{h}f-f\|_{E'}=0\right\}$.

\item [$(iv)$] If $E$ is reflexive, then $E'_{\ast}=E'$ and $E'$ is also a translation-invariant $(B)$-space of tempered ultradistributions of class $\ast$.

\end{itemize}

\end{te}

\begin{proof} Except for the inclusion $\mathcal{S}^{\ast}(\RR^d)\subseteq E'_{\ast}$, the rest of the assertions can be proved in exactly the same way as for the distribution case; we therefore omit details and refer to \cite[Sect. 3]{DPV}. To show the inclusion $\mathcal{S}^{\ast}(\RR^d)\subseteq E'_{\ast}$, note that $\SSS^*(\RR^d)=\overline{\operatorname*{span} (\SSS^*(\RR^d) * \SSS^*(\RR^d))}$ where the closure is taken in $\mathcal{S}^{\ast}(\RR^d)$ (this follows because $\varphi*\delta_j\to \varphi$ in $\SSS^*(\RR^d)$, where the sequence $\{\delta_j\}_{j=1}^{\infty}$ can be taken as in the proof of Proposition \ref{S'}). Hence, $\SSS^*(\RR^d)$ is a subset of the closure of $\operatorname*{span} (\SSS^*(\RR^d) * \SSS^*(\RR^d))$ in $E'$, and so the inclusion  $\mathcal{S}^{\ast}(\RR^d)\subseteq E'_{\ast}$ must hold.
\end{proof}

It is worth noticing that $E'$ carries another useful convolution structure. In fact, we can define the convolution mapping $\ast:E'\times \check{E} \to L^{\infty}_{\omega}$ by $$(f*g)(x)=\langle f(t),g(x-t)\rangle=\langle f(t),T_{-x}\check{g}(t)\rangle,$$
where $\check{E}=\left\{g\in \mathcal{S}'^*(\mathbb{R}^{d})|\, \check{g}\in E\right\}$ with norm $\|g\|_{\check{E}}:=\|\check{g}\|_{E}$ and $L^{\infty}_{\omega}$ is the dual of the Beurling algebra $L^{1}_{\omega}$, namely, the $(B)$-space of all measurable functions satisfying $\|u\|_{\infty,\omega}=\operatorname*{ess}\sup_{x\in\mathbb{R}^{d}} |g(x)|/\omega(x)<\infty$. We consider the following two closed subspaces of $L^{\infty}_{\omega}$:
\begin{equation}
\label{UC}
UC_{\omega}= \left\{u\in L^{\infty}_{\omega}\big|\, \lim_{h\to0}\|T_{h}u-u\|_{\infty,\omega}=0 \right\}
\ \ \mbox{ and} \ \
C_{\omega}= \left\{u\in C(\mathbb{R}^{d})\Big|\, \lim_{|x|\to\infty}\frac{u(x)}{\omega(x)}=0 \right\}.
\end{equation}
The first part of the next proposition is a direct consequence of $(b)$ from Theorem \ref{ttt11}. The range refinement in the reflexive case follows from the density of $\SSS^{\ast}(\mathbb{R}^{d})$ in $E'$ (part $(iv)$ of Theorem \ref{thgoodspace}).

\begin{pro}\label{convolution E E'} We have that
$E'\ast \check{E} \subseteq UC_{\omega}$ and $\ast:E'\times \check{E} \to UC_{\omega}$ is continuous. If $E$ is reflexive, then
$E'\ast \check{E} \subseteq C_{\omega}$.
\end{pro}

%
\section{The test function space $\mathcal{D}^*_{E}$}\label{tspace}

In this section we define and study the test function space $\mathcal{D}^{\ast}_{E}$, whose construction is based on the $(B)$-space $E$.
Let $$\mathcal{D}_E^{\{M_p\},m}=\left\{\varphi\in E\Big|\, D^{\alpha}\varphi\in E, \forall \alpha\in\NN^d, \|\varphi\|_{E,m}=\sup_{\alpha\in\NN^d}\frac{m^\alpha \|D^\alpha\varphi\|_E}{M_\alpha}<\infty\right\}.$$ It is a $(B)$-space
 with the norm $\|\cdot\|_{E,m}$.
One easily verifies that none of these spaces is trivial; indeed, they contain $\mathcal{D}^{\ast}(\mathbb{R}^{d})$. Also, $\mathcal{D}_E^{\{M_p\},m_1}\subseteq\mathcal{D}_E^{\{M_p\},m_2}$ for $m_2<m_1$ with continuous inclusion mapping. As l.c.s. we define
$$\mathcal{D}_E^{(M_p)}=\lim_{\substack{\longleftarrow\\ m\rightarrow\infty}} \mathcal{D}_E^{\{M_p\},m},\,\,\,\mathcal{D}_E^{\{M_p\}}=\lim_{\substack{\longrightarrow\\ m\rightarrow 0}} \mathcal{D}_E^{\{M_p\},m}.$$
Since $\mathcal{D}_E^{\{M_p\},m}$ is continuously injected in $E$ for each $m>0$, $\mathcal{D}_E^{\{M_p\}}$ is indeed a (Hausdorff) l.c.s.. Moreover $\mathcal{D}_E^{\{M_p\}}$ is barreled, bornological $(DF)$-space as an inductive limit of $(B)$-spaces. Obviously, $\mathcal{D}_E^{(M_p)}$ is an $(F)$-space. Of course $\mathcal{D}_E^*$ is continuously injected into $E$.

Additionally, in the $\{M_p\}$ case, for each fixed $(r_p)\in\mathfrak{R}$ we define the $(B)$-space
\beqs
\mathcal{D}_E^{\{M_p\},(r_p)}=\left\{\varphi \in E|\, D^{\alpha}\varphi\in E, \forall\alpha\in\NN^d, \|\varphi\|_{E,(r_p)}=\sup_{\alpha}\frac{\|D^\alpha\varphi\|_E}{M_\alpha \prod_{j=1}^{|\alpha|}r_j }<\infty\right\},
\eeqs
with norm $\|\cdot\|_{E,(r_p)}$. Since for $k>0$ and $(r_p)\in\mathfrak{R}$, there exists $C>0$ such that $k^{|\alpha|}\geq C/\left(\prod_{j=1}^{|\alpha|}r_j\right)$, $\DD^{\{M_p\},k}_E$ is continuously injected into $\DD^{\{M_p\},(r_p)}_E$. Define as l.c.s. $\ds\tilde{\mathcal{D}}^{\{M_p\}}_E=\lim_{\substack{\longleftarrow\\ (r_p)\in \mathfrak{R}}}\mathcal{D}_E^{\{M_p\},(r_p)}$. Then $\tilde{\DD}^{\{M_p\}}_E$ is a complete l.c.s. and $\DD^{\{M_p\}}_E$ is continuously injected into it.

\begin{pro}\label{regular}
The space $\mathcal{D}_E^{\{M_p\}}$ is regular, namely, every bounded set $B$ in $\mathcal{D}_E^{\{M_p\}}$ is bounded in some $\mathcal{D}_E^{\{M_p\},m}$. In addition $\DD^{\{M_p\}}_E$ is complete.
\end{pro}

\begin{proof} For $(r_p)\in\mathfrak{R}$ denote by $R_{\alpha}$ the product $\prod_{j=1}^{|\alpha|}r_j$. Let $B$ be a bounded set in $\mathcal{D}_E^{\{M_p\}}$. Then $B$ is bounded in $\tilde{\mathcal{D}}^{\{M_p\}}_E$; hence, for each $(r_p)\in \mathfrak{R}$ there exists $C_{(r_p)}>0$ such that $\ds\sup_{\alpha}\frac{\|D^\alpha\varphi\|_E}{R_\alpha M_\alpha}\leq C_{(r_p)}$, for all $\varphi\in B$. By \cite[Lem. 3.4]{Komatsu3} we obtain that there exist $m,C_2>0$ such that $\ds\sup_{\alpha}\frac{m^{|\alpha|}\|D^\alpha\varphi\|_E}{ M_\alpha}\leq C_2$, $\forall\varphi\in B$, which proves the regularity of $\DD^{\{M_p\}}_E$.\\
\indent It remains to prove the completeness. Since $\DD^{\{M_p\}}_E$ is a $(DF)$-space it is enough to prove that it is quasi-complete (see \cite[Thm. 3, p. 402]{kothe1}). Let $\varphi_{\nu}$ be a bounded Cauchy net in $\DD^{\{M_p\}}_E$. Hence there exist $m,C>0$ such that $\|\varphi_{\nu}\|_{E,m}\leq C$ and since the inclusions $\DD^{\{M_p\}}_E\rightarrow \DD^{\{M_p\},(r_p)}_E$ are continuous it follows that $\varphi_{\nu}$ is a Cauchy net in $\DD^{\{M_p\},(r_p)}_E$ for each $(r_p)\in\mathfrak{R}$. It is obvious that without losing generality we can assume that $m\leq 1$. Fix $m_1<m$. Let $\varepsilon>0$. There exists $p_0\in\ZZ_+$ such that $(m_1/m)^p\leq \varepsilon/(2C)$ for all $p\geq p_0$, $p\in\NN$. Let $r_p=p$. Obviously $(r_p)\in\mathfrak{R}$. Since $\varphi_{\nu}$ is a Cauchy net in $\DD^{\{M_p\},(r_p)}_E$, there exists $\nu_0$ such that for all $\nu,\lambda\geq \nu_0$ we have $\|\varphi_{\nu}-\varphi_{\lambda}\|_{E,(r_p)}\leq \varepsilon/(p_0!)$. Hence, for $|\alpha|<p_0$
\beqs
\frac{m_1^{|\alpha|}\|D^{\alpha}\varphi_{\nu}-D^{\alpha}\varphi_{\lambda}\|_E}{M_{\alpha}} \leq\frac{\|D^{\alpha}\varphi_{\nu}-D^{\alpha}\varphi_{\lambda}\|_E}{M_{\alpha}}\leq \varepsilon
\eeqs
and for $|\alpha|\geq p_0$
\beqs
\frac{m_1^{|\alpha|}\|D^{\alpha}\varphi_{\nu}-D^{\alpha}\varphi_{\lambda}\|_E}{M_{\alpha}}\leq 2C\left(\frac{m_1}{m}\right)^{|\alpha|}\leq \varepsilon.
\eeqs
We obtain that for $\nu,\lambda\geq \nu_0$, $\|\varphi_{\nu}-\varphi_{\lambda}\|_{E,m_1}\leq \varepsilon$, i.e., $\varphi_{\nu}$ is a Cauchy net in the $(B)$-space $\DD^{\{M_p\},m_1}_E$; hence, it converges to $\varphi\in\DD^{\{M_p\},m_1}_E$ in it and thus also in $\DD^{\{M_p\}}_E$.
\end{proof}

Similarly as in the first part of the proof of this proposition one can prove, by using \cite[Lem. 3.4]{Komatsu3}, that $\DD^{\{M_p\}}_E$ and $\tilde{\DD}^{\{M_p\}}_E$ are equal as sets, i.e., the canonical inclusion $\DD^{\{M_p\}}_E\rightarrow \tilde{\DD}^{\{M_p\}}_E$ is surjective. We will actually show later (cf. Theorem \ref{1517}) that the equality $\tilde{\DD}^{\{M_p\}}_E=\DD^{\{M_p\}}_E$ also holds topologically; however we need to study intrinsic properties of their duals in Section \ref{subsection DE} in order to reach such a result.

\begin{pro}\label{970}
The following dense inclusions hold $\mathcal{S}^*(\mathbb{R}^d)\hookrightarrow\mathcal{D}^{*}_E\hookrightarrow E\hookrightarrow \mathcal{S}'^*(\mathbb{R}^d)$ and $\mathcal{D}^{*}_E$ is a topological module over the Beurling algebra
$L^{1}_{\omega}$, that is, the convolution $*:L^1_{\omega}\times\DD^*_E\rightarrow \DD^*_E$ is continuous. Moreover in the $(M_p)$ case the following estimate
\begin{equation}\label{eq5}
\|u\ast \varphi\|_{E,m}\leq \|u\|_{1,\omega}\| \varphi\|_{E,m},\,\,\, m>0
\end{equation}
holds. In the $\{M_p\}$ case, for each $m>0$ the convolution is also a continuous bilinear mapping $L^1_{\omega}\times\DD^{\{M_p\},m}_E\rightarrow \DD^{\{M_p\},m}_E$ and the inequality (\ref{eq5}) holds.
\end{pro}

\begin{proof} Clearly $\mathcal{D}_{E}^*$ is continuously injected into $E$. We will consider the $\{M_p\}$ case. We will prove that for every $h>0$, $\SSS^{\{M_p\},h}_{\infty}(\RR^d)$ is continuously injected into $\DD^{\{M_p\},h/H}_E$. From this it readily follows that $\SSS^{\{M_p\}}(\RR^d)$ is continuously injected into $\DD^{\{M_p\}}_E$. Denote by $\sigma_h$ the norm in $\SSS^{\{M_p\},h}_{\infty}(\RR^d)$ (see (\ref{pcnnk})). Since $\SSS^{\{M_p\}}(\RR^d)\rightarrow E$, it follows that $\SSS^{\{M_p\},h/H}_{\infty}(\RR^d)\rightarrow E$. Hence there exists $C_1>0$ such that $\|\varphi\|_E\leq C_1 \sigma_{h/H}(\varphi)$, $\forall\varphi\in \SSS^{\{M_p\},h/H}_{\infty}(\RR^d)$. Let $\psi\in \SSS^{\{M_p\},h}_{\infty}(\RR^d)$. It is easy to verify that for every $\beta\in\NN^d$, $D^{\beta}\psi\in\SSS^{\{M_p\},h/H}_{\infty}(\RR^d)$. We have
\beqs
\frac{h^{|\alpha|}\left\|D^{\alpha}\psi\right\|_E}{H^{|\alpha|}M_{\alpha}}&\leq& C_1\frac{h^{|\alpha|}}{H^{|\alpha|}M_{\alpha}} \sup_{\beta}\frac{h^{|\beta|}\left\|e^{M(\frac{h}{H}|\cdot|)}D^{\alpha+\beta}\psi\right\|_{L^{\infty}(\RR^d)}} {H^{|\beta|}M_{\beta}}\\
&\leq& c_0C_1 \sup_{\beta}\frac{h^{|\alpha|+|\beta|}\left\|e^{M(h|\cdot|)}D^{\alpha+\beta}\psi\right\|_{L^{\infty}(\RR^d)}} {M_{\alpha+\beta}}\leq c_0C_1 \sigma_h(\psi),
\eeqs
which proves the continuity of the inclusion $\SSS^{\{M_p\},h}_{\infty}(\RR^d)\rightarrow \DD^{\{M_p\},h/H}_E$. The proof that $\SSS^{(M_p)}(\RR^d)$ is continuously injected into $\DD^{(M_p)}_E$ is similar and we omit it. We have shown that $\mathcal{S}^*(\mathbb{R}^d)\rightarrow\mathcal{D}_{E}^*\hookrightarrow E\hookrightarrow \mathcal{S}'^*(\mathbb{R}^d)$. To prove that $\mathcal{D}_{E}^*$ is a module over the Beurling algebra $L^{1}_{\omega}$ we first consider the $(M_p)$ case. For $u\in \DD^{(M_p)}\left(\RR^d\right)$, $\varphi\in \DD^{(M_p)}_E$ and $m>0$ we have
\beqs
\frac{m^{|\gamma|}}{M_\gamma}\left\|D^\gamma(u*\varphi)\right\|_E= \left\|u*\frac{m^{|\gamma|}}{M_\gamma}D^\gamma\varphi\right\|_E\leq \|u\|_{1,\omega}\|\varphi\|_{E,m}.
\eeqs
By a density argument, the same inequality holds true for $u\in L^1_{\omega}$ and $\varphi\in \DD^{(M_p)}_E$. After taking that supremum over $\gamma\in\NN^d$, we obtain (\ref{eq5}). In the $\{M_p\}$ case, by a similar calculation as above, we again obtain (\ref{eq5}) for $\varphi\in \DD^{\{M_p\},m}_E$ and $u\in L^1_{\omega}$. Hence the convolution is a continuous bilinear mapping $L^1_{\omega}\times \DD^{\{M_p\},m}_E\rightarrow \DD^{\{M_p\},m}_E$. From this we obtain that the convolution is separately continuous mapping $L^1_{\omega}\times \DD^{\{M_p\}}_E\rightarrow \DD^{\{M_p\}}_E$ and since $L^1_{\omega}$ and $\DD^{\{M_p\}}_E$ are barreled $(DF)$-spaces, it follows that it is continuous \cite[Thm. 11, p.161]{kothe1}.

 It remains to prove the density of the injection $\mathcal{S}^*(\mathbb{R}^d)\hookrightarrow\mathcal{D}_E^*$. Let $\varphi\in\mathcal{D}_{E}^*$. Pick then $\phi\in\mathcal{D}^*(\mathbb{R}^{d})$ with support in the unit ball of $\RR^d$ with center at the origin such that $\phi(x)\geq 0$ and $\int_{\mathbb{R}^{d}}\phi(x)dx=1$ and set $\phi_{j}(x)=j^{d}\phi(j x)$. We only consider the $\{M_p\}$ case, the $(M_p)$ case is similar. There exists $m>0$ such that $\phi,\varphi\in \DD^{\{M_p\},m}_E$ and $\left|D^{\alpha}\phi(x)\right|\leq \tilde{C} M_{\alpha}/m^{|\alpha|}$, for some $\tilde{C}>0$. Let $0<m_1<m$ be arbitrary but fixed. We will prove that $\|\varphi-\varphi\ast \phi_{j}\|_{E,m_1}\rightarrow 0$. Let $\varepsilon>0$. Observe that there exists $C_1\geq1$ such that $\|\phi_j\|_{1,\omega}\leq C_1$, $\forall j\in\ZZ_+$ and $\|\phi\|_{1,\omega}\leq C_1$. Choose $p_0\in\ZZ_+$ such that $(m_1/m)^p\leq \varepsilon/(2C_2)$ for all $p\geq p_0$, $p\in\NN$, where $C_2=C_1(1+\|\varphi\|_{E,m})\geq 1$. By $(f)$ of Theorem \ref{ttt11} we can choose $j_0\in\ZZ_+$ such that $\ds\frac{m_1^{|\alpha|}}{M_{\alpha}}\left\|D^{\alpha}\varphi-D^{\alpha}\varphi\ast \phi_{j}\right\|_E\leq \varepsilon$ for all $|\alpha|\leq p_0$ and all $j\geq j_0$, $j\in\NN$. Observe that if $|\alpha|\geq p_0$ we have
\beqs
\frac{m_1^{|\alpha|}}{M_{\alpha}}\left\|D^{\alpha}\varphi-D^{\alpha}\varphi\ast \phi_{j}\right\|_E&\leq& \frac{m_1^{|\alpha|}}{M_{\alpha}}\left\|D^{\alpha}\varphi\right\|_E+\frac{m_1^{|\alpha|}}{M_{\alpha}} \left\|D^{\alpha}\varphi\right\|_E\|\phi_{j}\|_{1,\omega}\\
&\leq& \left(\frac{m_1}{m}\right)^{|\alpha|}\|\varphi\|_{E,m}+C_1\left(\frac{m_1}{m}\right)^{|\alpha|}\|\varphi\|_{E,m}\leq \varepsilon.
\eeqs
Hence, for $j\geq j_0$, $\|\varphi-\varphi\ast \phi_{j}\|_{E,m_1}\leq \varepsilon$, so $\varphi\ast \phi_{j}\rightarrow \varphi$ in $\DD^{\{M_p\},m_1}_E$ and consequently also in $\DD^{\{M_p\}}_E$. Let $V$ be a neighborhood of $0$ in $\DD^{\{M_p\}}_E$. Choose a neighborhood of $0$ in $\DD^{\{M_p\}}_E$ such that $W+W\subseteq V$. Then $W_{m_1}=W\cap \DD^{\{M_p\},m_1}_E$ is a neighborhood of $0$ in $\DD^{\{M_p\},m_1}_E$; hence, there exists $j_1\in \ZZ_+$ such that $\varphi\ast \phi_{j_1}-\varphi \in W_{m_1}\subseteq W$. Choose $m_2>0$ such that $m_2<m_1/j_1$. Then $W_{m_2}=W\cap \DD^{\{M_p\},m_2}_E$ is a neighborhood of $0$ in $\DD^{\{M_p\},m_2}_E$. So there exists $\varepsilon>0$ such that $\left\{\chi\in\DD^{\{M_p\},m_2}_E \Big|\, \|\chi\|_{E,m_2}\leq \varepsilon\right\}\subseteq W_{m_2}$. Since $j_1m_2<m$, $\left|D^{\alpha}\phi(x)\right|\leq \tilde{C} M_{\alpha}/(j_1m_2)^{|\alpha|}$. Pick $\psi\in \SSS^{\{M_p\}}$ such that $\|\varphi-\psi\|_E\leq \varepsilon/(\tilde{C}C')$ where $\ds C'=\sup_{j\in\ZZ_+}\int_{|x|\leq 1} \omega(x/j)dx$ which is finite by the growth estimate for $\omega$. Now we have
\beqs
\frac{m_2^{|\alpha|}}{M_{\alpha}} \left\|(\varphi-\psi)*D^{\alpha}\phi_{j_1}\right\|_E&\leq& \|\varphi-\psi\|_E\int_{\mathbb{R}^{d}} \frac{j_1^d(j_1m_2)^{|\alpha|}}{M_{\alpha}}\left|D^{\alpha}\phi(j_1x)\right|\omega(x)d x\\
&\leq& \tilde{C}\|\varphi-\psi\|_E\int_{|x|\leq 1} \omega(x/j_1)d x\leq \varepsilon.
\eeqs
We obtain that $\psi*\phi_{j_1}-\varphi*\phi_{j_1}\in W_{m_2}\subseteq W$. Hence $\psi*\phi_{j_1}-\varphi=\psi*\phi_{j_1}-\varphi*\phi_{j_1}+\varphi*\phi_{j_1}-\varphi\in W+W\subseteq V$. Since $\psi*\phi_j\in \SSS^{\{M_p\}}(\RR^d)$ we conclude that $\SSS^{\{M_p\}}(\RR^d)$ is dense in $\DD^{\{M_p\}}_E$.
\end{proof}

Let $P(D)$ be an ultradifferential operator of $*$ type. Via standard arguments, one can prove  that $P(D):\mathcal{D}^*_E\rightarrow \mathcal{D}^*_E$ is continuous.

In order to prove that  ultradifferential operators of class $\{M_p\}$ act continuously on $\tilde{\DD}^{\{M_p\}}_E$, we need the following technical result \cite[Lem. 2.3]{BojanL}:
Let $(k_p)\in\mathfrak{R}$. There exists $(k'_p)\in\mathfrak{R}$ such that $k'_p\leq k_p$ and
\begin{equation}\label{bp}
\ds\prod_{j=1}^{p+q}k'_j\leq 2^{p+q}\prod_{j=1}^{p}k'_j\cdot\prod_{j=1}^{q}k'_j, \mbox{ for all } p,q\in\ZZ_+.
\end{equation}

\begin{pro}\label{udozc}
Every ultradifferential operator of class $\{M_p\}$  acts continuously on $\tilde{\DD}^{\{M_p\}}_E$.
\end{pro}

\begin{proof} Since $P(D)=\sum_{\alpha} c_{\alpha} D^{\alpha}$ is of class $\{M_p\}$ for every $L>0$ there exists $C>0$ such that $\left|c_{\alpha}\right|\leq C L^{|\alpha|}/M_{\alpha}$. Now, \cite[Lem. 3.4]{Komatsu3} implies that there exist $(r_p)\in\mathfrak{R}$ and $C_1>0$ such that $\left|c_{\alpha}\right|\leq C_1 /\left(M_{\alpha}\prod_{j=1}^{|\alpha|}r_j\right)$. Let $(l_p)\in\mathfrak{R}$ be arbitrary but fixed. Define $k_p=\min\{r_p,l_p\}$, $p\in\ZZ_+$.
Then $(k_p)\in\mathfrak{R}$ and for this $(k_p)$ take $(k'_p)\in\mathfrak{R}$ as in (\ref{bp}). Then, there exists $C'>0$ such that $\left\|P(D)\varphi\right\|_{E,(l_p)}\leq C' \|\varphi\|_{E, (k'_p/(4H))}$ for $\varphi\in \tilde{\DD}^{\{M_p\}}_E$, which implies the continuity of $P(D)$. Indeed, for all $\beta\in\NN^d$,
\beqs
\frac{\left\|D^{\beta}P(D)\varphi\right\|_E}{M_{\beta}\prod_{j=1}^{|\beta|}l_j}&\leq& C\sum_{\alpha}\frac{\left\|D^{\alpha+\beta}\varphi\right\|_E}{M_{\alpha}M_{\beta}\prod_{j=1}^{|\alpha|}r_j\prod_{j=1}^{|\beta|}l_j} \leq C_1\sum_{\alpha}\frac{H^{|\alpha|+|\beta|}|\left\|D^{\alpha+\beta}\varphi\right\|_E} {M_{\alpha+\beta}\prod_{j=1}^{|\alpha|}k'_j\prod_{j=1}^{|\beta|}k'_j}\\
&\leq& C_2\|\varphi\|_{(k'_p/(4H))}\sum_{\alpha}2^{-|\alpha|}\leq C'\|\varphi\|_{(k'_p/(4H))}.
\eeqs
\end{proof}

Interestingly, all elements of our test space $\mathcal{D}^*_{E}$ are ultradifferentiable functions of class $*$. To establish this fact we need the following lemma.

\begin{lemma}\label{lemma4.2}
Let $K\subseteq \mathbb{R}^d$ be compact. There exists $m>0$ (there exists $(l_p)\in \mathfrak{R}$) such that $\mathcal{D}_{K,m}^{\{M_p\}}\subseteq E\cap E'_{\ast}$ ($\mathcal{D}_{K,(l_p)}^{\{M_p\}}\subseteq E\cap E'_{\ast}$). Moreover, the inclusion mappings $\mathcal{D}_{K,m}^{\{M_p\}}\to E$ and $\mathcal{D}_{K,m}^{\{M_p\}}\to E'_{*}$ ($\mathcal{D}_{K,(l_p)}^{\{M_p\}}\to E$ and $\mathcal{D}_{K,(l_p)}^{\{M_p\}}\to E'_{\ast}$) are continuous.
\end{lemma}

\begin{proof} We will give the proof in the Roumieu case; the Beurling case is similar. Let $U$ be a bounded open subset of $\RR^d$ such that $K\Subset U$ and set $K_1=\overline{U}$. Since the inclusion $\DD^{\{M_p\}}_{K_1}\rightarrow E$ is continuous and $\ds \DD^{\{M_p\}}_{K_1}=\lim_{\substack{\longleftarrow\\ (r_p)\in\mathfrak{R}}}\DD^{\{M_p\}}_{K_1,(r_p)}$ there exist $C>0$ and $(r_p)\in\mathfrak{R}$ such that $\|\varphi\|_E\leq C \|\varphi\|_{K_1,(r_p)}$. Let $\chi_m$, $m\in\ZZ_+$, be a $\delta$-sequence from $\DD^{\{M_p\}}$ such that $\mathrm{diam}(\mathrm{supp}\,\chi_m)\leq \mathrm{dist}(K,\partial U)/2$, for $m\in\ZZ_+$. Take $l_p= r_{p-1}/(2H)$, $p\geq 2$ and $l_1=r_1/(2H)$. Then $(l_p)\in\mathfrak{R}$. Let $\psi\in \DD^{\{M_p\}}_{K,(l_p)}$. Then $\psi*\chi_m\in \DD^{\{M_p\}}_{K_1}$ and one easily obtains that $\psi*\chi_m\rightarrow \psi$ in $\DD^{\{M_p\}}_{K_1,(r_p)}$. We have $\|\psi*\chi_m\|_E\leq C \|\psi*\chi_m\|_{K_1,(r_p)}$; hence,  $\psi*\chi_m$ is a Cauchy sequence in $E$, so it converges. Since $\psi*\chi_m\rightarrow \psi$ in $\DD'^{\{M_p\}}(\RR^d)$ and $E$ is continuously injected into $\DD'^{\{M_p\}}(\RR^d)$ the limit of $\psi*\chi_m$ in $E$ must be $\psi$. If we let $m\rightarrow\infty$ in the last inequality we have $\|\psi\|_E\leq C\|\psi\|_{K_1,(r_p)}$. Observe that $\|\psi\|_{K_1,(r_p)}\leq \|\psi\|_{K,(l_p)}$ (since $\psi\in\DD^{\{M_p\}}_{K,(l_p)}$, $\mathrm{supp}\, \psi\subseteq K$). Hence, $\|\psi\|_E\leq C\|\psi\|_{K,(l_p)}$, which gives the desired continuity of the inclusion $\mathcal{D}_{K,(l_p)}^{\{M_p\}}\to E$. Similarly, one obtains the continuous inclusion $\mathcal{D}_{K,(l'_p)}^{\{M_p\}}\to E'_{\ast}$ possibly with another $(l'_p)\in\mathfrak{R}$. The conclusion of the lemma now follows with $(\tilde{l}_p)\in\mathfrak{R}$ defined as $\tilde{l}_p=\min\{l_p,l'_p\}$, $p\in\ZZ_+$.
\end{proof}

\begin{pro}\label{smooth prop}
The embedding $\mathcal{D}^*_{E}\hookrightarrow \mathcal{O}^*_{C}(\mathbb{R}^{d})$ holds. Furthermore, for $\varphi\in\mathcal{D}^*_{E}$, $D^{\alpha}\varphi\in C_{\check{\omega}}$ for all $\alpha\in\NN^d$ and they satisfy the following growth condition: for every $m>0$ (for some $m>0$)
\begin{equation}\label{eqgrowth}
\sup_{\alpha\in\NN^d}\frac{m^{|\alpha|}}{M_{\alpha}} \left\|D^{\alpha}\varphi\right\|_{L^{\infty}_{\check{\omega}}{(\RR^d)}}<\infty.
\end{equation}
\end{pro}

\begin{proof} Let $U$ be the open unit ball in $\RR^d$ with center at $0$ and let $K=\overline{U}$. Let $r>0$ (let $(r_p)\in\mathfrak{R}$) be as in Lemma \ref{lemma4.2}, i.e., $\mathcal{D}_{K,r}^{\{M_p\}}\subseteq E\cap E'_{\ast}$ ($\mathcal{D}_{K,(r_p)}^{\{M_p\}}\subseteq E\cap E'_{\ast}$) and the inclusion mappings $\mathcal{D}_{K,r}^{\{M_p\}}\to E$ and $\mathcal{D}_{K,r}^{\{M_p\}}\to E'_{*}$ ($\mathcal{D}_{K,(r_p)}^{\{M_p\}}\to E$ and
$\mathcal{D}_{K,(r_p)}^{\{M_p\}}\to E'_{\ast}$) are continuous. By the parametrix of Komatsu, there exist $u\in\DD^{(M_p)}_{U,r}$,
$\psi\in \DD^{(M_p)}(U)$ and $P(D)$ of type $(M_p)$ ($u\in\DD^{\{M_p\}}_{U,(r_p)}$ satisfying $\ds \frac{\left\|D^{\alpha}u\right\|_{L^{\infty}}}{R_{\alpha}M_{\alpha}}\rightarrow 0$ when $|\alpha|\rightarrow\infty$, $\psi\in \DD^{\{M_p\}}(U)$ and
$P(D)$ of type $\{M_p\}$) such that $P(D)u=\delta+\psi$. Let $f\in\mathcal{D}^*_{E}$. Then $f=u*P(D)f-\psi*f$. Observe that $\psi*f\in\EE^*(\RR^d)$. For $\beta\in\NN^d$, $D^{\beta}P(D)f\in \DD_E^*$. By Lemma \ref{lemma4.2},
$\check{u}\in\mathcal{D}_{K,(r_p)}^{\{M_p\}}\subseteq E'$ and so $u\in (E')\check{}=\check{E}'$. Hence, by the discussion before Proposition \ref{convolution E E'}, all ultradistributional derivatives of $u*P(D)f$ are continuous functions on $\RR^d$. From this we obtain that $u*P(D)f\in C^{\infty}\left(\RR^d\right)$.
Indeed, this result is of local nature, so it is enough to use the Sobolev embedding theorem on an open disk $V$ of arbitrary point
$x\in\RR^d$ and the fact that $\DD^*(V)$ is dense in $\DD(V)$. Hence $f\in C^{\infty}\left(\RR^d\right)$. For $\beta\in\NN^d$,
$D^{\beta}f(x)=u*D^{\beta}P(D)f(x)-\psi*D^{\beta}f(x)=F_1(x)-F_2(x)$. By the above discussion, the last equality, and
Proposition \ref{convolution E E'}, it follows that $D^{\beta}f\in UC_{\check{\omega}}$. To prove the inclusion
$\mathcal{D}^*_{E}\rightarrow \mathcal{O}^*_{C}(\mathbb{R}^{d})$, we consider first the $(M_p)$ case. Let $m>0$ be
arbitrary but fixed. Since $P(D)=\sum_{\alpha}c_{\alpha}D^{\alpha}$ is of $(M_p)$ type, there exist $m_1,C'>0$
such that $|c_{\alpha}|\leq C'm_1^{|\alpha|}/M_{\alpha}$. Let $m_2=4\max\{m,m_1\}$. For $F_1$, since $P(D)$ acts continuously on $\mathcal{D}^*_E$, we have
\beqs
|F_1(x)|\leq \|u\|_{\check{E}'}\left\|D^{\beta}P(D)f(x)\right\|_E\omega(-x)\leq C_2\omega(-x)\|\check{u}\|_{E'}\|f\|_{E,m_2H}\frac{M_{\beta}}{(2m)^{|\beta|}}
\eeqs
and similarly
\beqs
|F_2(x)|\leq C_3\omega(-x)\|\check{\psi}\|_{E'}\|f\|_{E,2m}\frac{M_{\beta}}{(2m)^{|\beta|}}\leq C_3\omega(-x)\|\check{\psi}\|_{E'}\|f\|_{E,m_2H}\frac{M_{\beta}}{(2m)^{|\beta|}}.
\eeqs
Hence
\begin{equation}\label{bounds}
\frac{(2m)^{|\beta|}\left|D^{\beta}f(x)\right|}{M_{\beta}w(-x)}\leq C''\left(\|\check{u}\|_{E'}+\|\check{\psi}\|_{E'}\right)\|f\|_{E,m_2H}.
\end{equation}
Since there exist $\tau,C'''>0$ such that $\omega(x)\leq C''' e^{M(\tau|x|)}$, by using \cite[Prop. 3.6]{Komatsu1}, we obtain $\omega(-x)e^{M(\tau|x|)}\leq C_4 e^{M(\tau H|x|)}$. Hence
\beqs
\left(\sum_{\alpha}\frac{m^{2|\alpha|}}{M_{\alpha}^2}\left\|D^{\alpha}f e^{-M(\tau H|\cdot|)}\right\|^2_{L^2}\right)^{1/2}&\leq& C_5\left(\sum_{\alpha}\frac{m^{2|\alpha|}}{M_{\alpha}^2}\left\|\frac{D^{\alpha}f} {\omega(-\cdot)}\right\|^2_{L^{\infty}}\right)^{1/2}\\
&\leq& C\left(\|\check{u}\|_{E'}+\|\check{\psi}\|_{E'}\right)\|f\|_{E,m_2H},
\eeqs
which proves the continuity of the inclusion $\mathcal{D}^{(M_p)}_{E}\rightarrow \mathcal{O}_{C,\tau H}^{(M_p)}(\RR^d)$ and hence also the continuity of the inclusion $\mathcal{D}^{(M_p)}_{E}\rightarrow \mathcal{O}_C^{(M_p)}(\RR^d)$.

 In order to prove that the inclusion $\mathcal{D}^{\{M_p\}}_{E}\rightarrow \mathcal{O}_C^{\{M_p\}}(\RR^d)$ is continuous it is enough to prove that, for each $h>0$, $\mathcal{D}^{\{M_p\}}_{E}\rightarrow \mathcal{O}_{C,h}^{\{M_p\}}(\RR^d)$ is a continuous inclusion. And in order to prove this, it is enough to prove that for every $m>0$ there exists $m'>0$ such that we have the continuous inclusion $\mathcal{D}^{\{M_p\},m}_{E}\rightarrow \mathcal{O}_{C,m',h}^{\{M_p\}}(\RR^d)$. So, let $h,m>0$ be arbitrary but fixed. Take $m'\leq m/(4H)$. For $f\in \mathcal{D}^{\{M_p\},m}_{E}$, using the same technique as above, we have
\beq\label{boundss}
\frac{(2m')^{|\beta|}\left|D^{\beta}f(x)\right|}{M_{\beta}w(-x)}\leq C''\left(\|\check{u}\|_{E'}+\|\check{\psi}\|_{E'}\right)\|f\|_{E,m}.
\eeq
For the fixed $h$ take $\tau>0$ such that $\tau H\leq h$. Then there exists $C'''>0$ such that $\omega(x)\leq C''' e^{M(\tau|x|)}$ and by using \cite[Prop. 3.6]{Komatsu1} we obtain $\omega(x)e^{M(\tau|x|)}\leq C_4 e^{M(\tau H|x|)}$. Similarly as above, we have
\beqs
\left(\sum_{\alpha}\frac{m'^{2|\alpha|}}{M_{\alpha}^2}\left\|D^{\alpha}f e^{-M(h|\cdot|)}\right\|^2_{L^2}\right)^{1/2} \leq C\left(\|\check{u}\|_{E'}+\|\check{\psi}\|_{E'}\right)\|f\|_{E,m},
\eeqs
which proves the continuity of the inclusion $\mathcal{D}^{\{M_p\},m}_{E}\rightarrow \mathcal{O}_{C,m',h}^{\{M_p\}}(\RR^d)$.

 Observe that (\ref{eqgrowth}) follows by (\ref{bounds}) in the $(M_p)$ case and by (\ref{boundss}) in the $\{M_p\}$ case. It remains to prove that $D^{\alpha}f\in C_{\check{\omega}}$. We will prove this in the $\{M_p\}$ case, the $(M_p)$ case is similar. By using Proposition \ref{udozc}, with similar technique to that above, one can prove that for every $(k_p)\in\mathfrak{R}$ there exists $(l_p)\in\mathfrak{R}$ such that for $f\in\DD^{\{M_p\}}_E$ we have
\beq\label{15557}
\frac{\left|D^{\beta}f(x)\right|}{w(-x)M_{\beta}\prod_{j=1}^{|\beta|}k_j}\leq C''\left(\|\check{u}\|_{E'}+\|\check{\psi}\|_{E'}\right)\|f\|_{E,(l_p)}.
\eeq
Let $\varepsilon>0$. Since $\DD^{\{M_p\}}(\RR^d)$ is dense in $\DD^{\{M_p\}}_E$ (Proposition \ref{970}) it is dense in $\tilde{\DD}^{\{M_p\}}_E$. Pick $\chi\in\DD^{\{M_p\}}(\RR^d)$ such that $\|f-\chi\|_{E,(l_p)}\leq \varepsilon/\left(C''\left(\|\check{u}\|_{E'}+\|\check{\psi}\|_{E'}\right)\right)$. Then, by (\ref{15557}), for $x\in\RR^d\backslash \mathrm{supp}\,\chi$ we have
\beqs
\frac{\left|D^{\beta}f(x)\right|}{w(-x)M_{\beta}\prod_{j=1}^{|\beta|}k_j}= \frac{\left|D^{\beta}\left(f(x)-\chi(x)\right)\right|}{w(-x)M_{\beta}\prod_{j=1}^{|\beta|}k_j}\leq \varepsilon,
\eeqs
which proves that $D^{\beta}f\in C_{\check{\omega}}$.
\end{proof}

\begin{re} If $f\in\SSS^*(\RR^d)$, by the proof of the previous proposition (and (\ref{eqwconvolution1})), we have
\beqs
\left\|D^{\beta}f\right\|_E\leq \|u\|_E\left\|D^{\beta}P(D)f\right\|_{1,\omega}+\|\psi\|_E\left\|D^{\beta}f\right\|_{1,\omega},
\eeqs
since $u,\psi\in E$ (by their choice). Also, one easily verifies that (cf. the proof of Proposition \ref{udozc}) for every $m>0$ there exist $\tilde{m}>0$ and $C_1>0$ (for every $(k_p)\in \mathfrak{R}$ there exist $(l_p)\in\mathfrak{R}$ and $C_1>0$) such that
\beq\label{111111177}
\|f\|_{E,m}\leq C_1 \sup_{\alpha}\frac{\tilde{m}^{|\alpha|}\left\|D^{\alpha}f\right\|_{1,\omega}}{M_{\alpha}}\,\,\, \left(\|f\|_{E,(k_p)}\leq C_1 \sup_{\alpha}\frac{\left\|D^{\alpha}f\right\|_{1,\omega}}{M_{\alpha}\prod_{j=1}^{|\alpha|}l_j}\right).
\eeq
\end{re}

\section{The ultradistribution space $\mathcal{D}'^*_{E'_{\ast}}$}\label{subsection DE}

We denote by $\mathcal{D}'^*_{E'_{\ast}}$ the strong dual of $\mathcal{D}^*_{E}$. Then, $\mathcal{D}'^{(M_p)}_{E'_{\ast}}$ is a complete $(DF)$-space since $\mathcal{D}^{(M_p)}_{E}$ is an $(F)$-space. Also, $\mathcal{D}'^{\{M_p\}}_{E'_{\ast}}$ is an $(F)$-space as the strong dual of a $(DF)$-space. When $E$ is reflexive, we write $\mathcal{D}'^*_{E'}=\mathcal{D}'^*_{E'_{\ast}}$ in accordance with the last assertion of Theorem \ref{thgoodspace}. The notation $\mathcal{D}'^*_{E'_{\ast}}=(\mathcal{D}^*_{E})'$ is motivated by the next structural theorem.

\begin{te}\label{karak}
Let $f\in\mathcal{D}'^*(\mathbb{R}^d)$. The following statements are equivalent:
\begin{itemize}
\item [$(i)$] $f\in \mathcal{D}'^*_{E'_{\ast}}$.
\item [$(ii)$] $f*\psi\in E'$ for all $\psi\in \mathcal{D}^*(\mathbb{R}^d)$.
\item [$(iii)$] $f*\psi\in E'_{\ast}$ for all $\psi\in \mathcal{D}^*(\mathbb{R}^d)$.
\item [$(iv)$] $f$ can be expressed as $f=P(D)g+g_1$, where $P(D)$ is an ultradifferential operator of $\ast$ type  with $g,g_1\in E'$.
\item [$(v)$] There exist ultradifferential operators  $P_{k}(D)$ of $*$ type and $f_{k}\in E'_{\ast}\cap UC_{\omega}$ for $k$ in a finite set $J$ such that
\begin{equation}\label{eq:representation}
 f=\sum_{k\in J}P_k(D)f_k.
\end{equation}
Moreover, if $E$ is reflexive, then we may choose $f_k\in E'\cap C_{\omega}$.
\end{itemize}
\end{te}

\begin{re}\label{r1}
One can replace $\mathcal{D}'^*(\mathbb{R}^d)$ and $\mathcal{D}^*(\mathbb{R}^d)$ by $\mathcal{S}'^*(\mathbb{R}^d)$ and $\mathcal{S}^*(\mathbb{R}^d)$ in every statement of Theorem \ref{karak}.
\end{re}

\begin{proof}
We denote $B_{E}=\{ \varphi\in\mathcal{D}^*(\mathbb{R}^d)|\, \|\varphi\|_E\leq1\}$.

$(i)\Rightarrow (ii)$. Fix first $\psi\in\mathcal{D}^*(\mathbb{R}^d)$. By (\ref{eq5}) the set $\check{\psi}\ast B_{E}=\{\check{\psi}*\varphi|\: \varphi \in B_{E}\}$ is bounded in $\mathcal{D}^*_{E}$. Hence, $|\langle f*\psi,\varphi\rangle|=|\langle f,\check{\psi}*\varphi\rangle|\leq C_{\psi}$ for $\varphi\in B_{E}$. So, $|\langle f*\psi,\varphi\rangle|\leq C_{\psi}\|\varphi\|_E,$ for all $\varphi\in\mathcal{D}^*(\mathbb{R}^{d})$. Since $\mathcal{D}^*(\mathbb{R}^{d})$ is dense in $E$, we obtain $f*\psi\in E'$, for each $\psi\in \mathcal{D}^*(\mathbb{R}^{d})$.

 $(ii)\Rightarrow(iv)$. Let $\Omega$ be a bounded open symmetric neighborhood of $0$ in $\RR^d$ and set $K=\overline{\Omega}$. For arbitrary but fixed $\psi\in\mathcal{D}^*_K$ we have $\langle f*\check{\varphi},\check{\psi}\rangle=\langle f*\psi,\varphi\rangle$. We obtain that the set $\{\langle f*\check{\varphi},\check{\psi}\rangle|\, \varphi\in B_{E}\}$ is bounded in $\mathbb{C}$, that is, $\{f*\check{\varphi}|\, \varphi\in B_{E}\}$ is weakly bounded in $\DD'^*_K$; hence, it is equicontinuous. Using the same technique as in the proof of Proposition \ref{S'}, we obtain that there exists $r>0$ such that for each $\rho\in \mathcal{D}^{(M_p)}_{\Omega,r}$ there exists $C_{\rho}>0$ (there exists $(r_p)\in \mathfrak{R}$ such that for each $\rho\in \mathcal{D}^{\{M_p\}}_{\Omega,(r_p)}$ there exists $C_{\rho}>0$) satisfying $|\langle f*\rho,\varphi\rangle|\leq C_{\rho}$ for all $\varphi\in B_{E}$. The density of $\mathcal{D}^*(\mathbb{R}^d)$ in $E$ implies that $f*\rho\in E'$ for each $\rho\in \mathcal{D}^{(M_p)}_{\Omega,r}$ (for each $\rho\in \mathcal{D}^{\{M_p\}}_{\Omega,(r_p)}$). The parametrix of Komatsu implies the existence of $u\in \mathcal{D}^{(M_p)}_{\Omega,r}$, $\psi\in \DD^{(M_p)}(\Omega)$ and ultradifferential operator $P(D)$ of class $(M_p)$ (the existence of $u\in \mathcal{D}^{\{M_p\}}_{\Omega,(r_p)}$, $\psi \in \mathcal{D}^{\{M_p\}}(\Omega)$ and ultradifferential operator $P(D)$ of class $\{M_p\}$) satisfying $f=P(D)(u*f)+\psi*f$. This gives the desired representation.

 $(iv)\Rightarrow (i)$. Obvious.

 $(ii)\Rightarrow (v)$. Proceed as in $(ii)\Rightarrow(iv)$ to obtain $f=P(D)(u*f)+\psi*f$ for some $u\in \mathcal{D}^{(M_p)}_{\Omega,r}$, $\psi\in \DD^{(M_p)}(\Omega)$ and ultradifferential operator $P(D)$ of class $(M_p)$ (for some $u\in \mathcal{D}^{\{M_p\}}_{\Omega,(r_p)}$, $\psi \in \mathcal{D}^{\{M_p\}}(\Omega)$ and ultradifferential operator $P(D)$ of class $\{M_p\}$). Moreover, by using Lemma \ref{lemma4.2}, one can easily see from the proof of $(ii)\Rightarrow(iv)$ that we can choose $r$ such that $\mathcal{D}^{(M_p)}_{\Omega,r}\subseteq \check{E}$ (we can choose $(r_p)$ such that $\mathcal{D}^{\{M_p\}}_{\Omega,(r_p)}\subseteq \check{E}$). Observe that the composition of ultradifferential operators of class * is again an ultradifferential operator of class *. We obtain
\beqs
f&=&P(D)(u*(P(D)(u*f)+\psi*f))+\psi*(P(D)(u*f)+\psi*f)\\
&=&P(D)(P(D)(u*(u*f)))+P(D)(u*(\psi*f))+P(D)(\psi*(u*f))+\psi*(\psi*f)
\eeqs
and $u*(u*f),u*(\psi*f),\psi*(u*f),\psi*(\psi*f)\in E'_*\cap UC_{\omega}$ by the definition of $E'_*$ and Proposition \ref{convolution E E'}. If $E$ is reflexive, all of these are in fact elements of $C_{\omega}$ by the same proposition.\\

 The implications $(v)\Rightarrow (i)$, $(iv)\Rightarrow(iii)$ and $(iii)\Rightarrow(ii)$ are obvious.
\end{proof}

\begin{pro}\label{com}
Let $\mathbf{f}:\mathcal{D}^*(\RR^d)\rightarrow \mathcal{D}'^*(\RR^d)$ be linear and continuous. The following statements are equivalent:
\begin{itemize}
\item[$(i)$] $\mathbf{f}$ commutates with every translation, i.e., $\left\langle \mathbf{f},T_{-h}\varphi\right\rangle = T_{h} \left\langle \mathbf{f},\varphi\right\rangle$, for all $h\in\mathbb{R}^{d}$ and $\varphi\in \mathcal{D}^*(\mathbb{R}^{d}).$
\item[$(ii)$] $\mathbf{f}$ commutates with every convolution, i.e., $\left\langle\mathbf{f},\psi*\varphi\right\rangle=\check{\psi}*\left\langle \mathbf{f},\varphi\right\rangle$, for all $\psi,\varphi\in \mathcal{D}^*(\mathbb{R}^{d}).$
\item[$(iii)$] There exists $f\in \mathcal{D}'^*(\RR^d)$ such that  $\langle \mathbf{f},\varphi\rangle=f*\check{\varphi}$ for every $\varphi\in \mathcal{D}^*(\RR^d)$.
\end{itemize}
\end{pro}

\begin{proof} $(i)\Rightarrow (ii)$ Let $\varphi,\psi\in\DD^*(\RR^d)$ and denote $K=\mathrm{supp}\, \psi$. Then the Riemann sums
\beqs
L_{\varepsilon}(\cdot)=\sum_{y\in\ZZ^d,\, \varepsilon y\in K}\psi(\varepsilon y)\varphi(\cdot-\varepsilon y)\varepsilon^d =\sum_{y\in\ZZ^d,\, \varepsilon y\in K}\psi(\varepsilon y)T_{-\varepsilon y}\varphi \varepsilon ^d
\eeqs
converge to $\psi*\varphi$ in $\DD^*(\RR^d)$, when $\varepsilon \rightarrow0^+$. The continuity of $\mathbf{f}$ implies
\beqs
\langle\mathbf{f},\psi*\varphi\rangle =\lim_{\varepsilon\rightarrow 0^+}\sum_{y\in\ZZ^d,\, \varepsilon y\in K} \psi(\varepsilon y)\langle \mathbf{f},T_{-\varepsilon y} \varphi\rangle \varepsilon ^d =\lim_{\varepsilon\rightarrow 0^+}\sum_{y\in\ZZ^d,\, \varepsilon y\in K} \psi(\varepsilon y) T_{\varepsilon y}\langle \mathbf{f},\varphi\rangle \varepsilon ^d,
\eeqs
in $\DD'^*(\RR^d)$. Let $\chi\in \mathcal{D}^*(\RR^d)$. Then
\beqs
\left\langle\lim_{\varepsilon\rightarrow 0^+}\sum_{y\in\ZZ^d,\, \varepsilon y\in K} \psi(\varepsilon y) T_{\varepsilon y} \langle \mathbf{f},\varphi\rangle \varepsilon ^d,\chi\right\rangle=\langle\langle\mathbf{f},\varphi\rangle,\psi*\chi\rangle= \langle\check{\psi}*\langle\mathbf{f},\varphi\rangle,\chi\rangle.
\eeqs

 $(ii)\Rightarrow (iii)$. Let $\Omega$ be an arbitrary symmetric bounded open neighborhood of $0$ in $\RR^d$ and set $K=\overline{\Omega}$. Take $\delta_m\in\DD^*\left(\RR^d\right)$ as in the proof of Proposition \ref{S'}. For every $\psi\in \mathcal{D}^*(\RR^d)$ we have that $\psi*\delta_m\rightarrow \psi$ in $\mathcal{D}^*(\RR^d)$ when $m\rightarrow\infty$. Also,
\begin{equation}\label{delta}
\check{\psi}*\langle\mathbf{f},\delta_m\rangle=\langle \mathbf{f},\psi*\delta_m\rangle\rightarrow \langle\mathbf{f},\psi\rangle\mbox{ when }m\rightarrow\infty.
\end{equation}
First we will prove that the set $\{\langle \mathbf{f},\delta_m\rangle|m\in\mathbb{Z}_+\}$ is equicontinuous subset of $\mathcal{D}'^*(\RR^d)$, or equivalently bounded in $\DD'^*\left(\RR^d\right)$ (since $\DD^*(\RR^d)$ is barreled). By (\ref{delta}), for each fixed $\psi\in\DD^*\left(\RR^d\right)$, the set $\{\psi*\langle \mathbf{f},\delta_m\rangle|m\in\ZZ_+\}$ is bounded in $\DD'^*\left(\RR^d\right)$. Denote by $G_m$ the bilinear mapping $(\varphi,\psi)\mapsto \langle \mathbf{f},\delta_m\rangle*\varphi*\psi|_{K}$, $G_m: \DD^*_K\times \DD^*_K\rightarrow C(K)$. For fixed $\psi\in\DD^*_K$, the mappings $G_{m,\psi}$ defined by $\varphi\mapsto \langle \mathbf{f},\delta_m\rangle*\varphi*\psi|_K$, $\DD^*_K\rightarrow C(K)$ are linear and continuous and the set $\{G_{m,\psi}|\, m\in\ZZ_+\}$ is pointwise bounded in $\mathcal{L}\left(\DD^*_K,C(K)\right)$. Since $\DD^*_K$ is barreled, this set is equicontinuous. Similarly, for each fixed $\varphi\in\DD^*_K$, the mappings $\psi\mapsto \langle \mathbf{f},\delta_m\rangle*\varphi*\psi|_K$, $\DD^*_K\rightarrow C(K)$ form an equicontinuous subset of $\mathcal{L}\left(\DD^*_K,C(K)\right)$. We obtain that the set of bilinear mappings $\{G_m|\, m\in\ZZ_+\}$ is separately equicontinuous and since $\DD^{(M_p)}_K$ is an $(F)$-space and $\DD^{\{M_p\}}_K$ is a barreled $(DF)$-space, it is equicontinuous (see \cite[Thm. 2, p. 158]{kothe1} for the case of $(F)$-spaces and \cite[Thm. 11, p. 161]{kothe1} for the case of barreled $(DF)$-spaces). We will continue the proof by considering only the $\{M_p\}$ case; the $(M_p)$ case can be treated similarly. By the equicontinuity of the mappings $G_m$, $m\in\ZZ_+$, there exist $C>0$ and $(k_p)\in\mathfrak{R}$ such that for all $\varphi,\psi\in\DD^{\{M_p\}}_K$, $m\in \ZZ_+$, we have $\left\|G_m(\varphi,\psi)\right\|_{L^{\infty}(K)}\leq C\|\varphi\|_{K,(k_p)}\|\psi\|_{K,(k_p)}$. Let $r_p=k_{p-1}/H$, for $p\in\NN$, $p\geq 2$ and set $r_1=\min\{1,r_2\}$. Then $(r_p)\in \mathfrak{R}$. For $\chi\in \DD^{\{M_p\}}_{\Omega, (r_p)}$, for large enough $j$, $\chi*\delta_j\in \DD^{\{M_p\}}_K$ and by similar technique as in the proof of Proposition \ref{S'} one can prove that $\chi*\delta_j \rightarrow \chi$ in $\DD^{\{M_p\}}_{K,(k_p)}$, where $\delta_j\in\DD^*(\RR^d)$, $j\in\ZZ_+$, is the same sequence as that used in the proof of Proposition \ref{S'}. Let $\varphi,\psi\in \DD^{\{M_p\}}_{\Omega, (r_p)}$ and set $\varphi_j=\varphi*\delta_j$, $\psi_j=\psi*\delta_j$. Since\\
$\|G_m(\varphi_j,\psi_j)-G_m(\varphi_s,\psi_s)\|_{L^{\infty}(K)}$
\beqs
&\leq& \|G_m(\varphi_j,\psi_j-\psi_s)\|_{L^{\infty}(K)}+\|G_m(\varphi_j-\varphi_s,\psi_s)\|_{L^{\infty}(K)}\\
&\leq&C\left(\|\varphi_j\|_{K,(k_p)}\|\psi_j-\psi_s\|_{K,(k_p)}+\|\varphi_j-\varphi_s\|_{K,(k_p)}\|\psi_s\|_{K,(k_p)}\right),
\eeqs
it follows that for each fixed $m$, $G_m(\varphi_j,\psi_j)$ is a Cauchy sequence in $C(K)$; hence, it must converge. On the other hand, $\langle \mathbf{f},\delta_m\rangle*\varphi_j*\psi_j\rightarrow\langle \mathbf{f},\delta_m\rangle*\varphi*\psi$ in $\DD'^{\{M_p\}}\left(\RR^d\right)$ and, since $C(K)$ is continuously injected into $\DD'^{\{M_p\}}_K$, it follows that $G_m(\varphi_j,\psi_j)$ converges to $\langle \mathbf{f},\delta_m\rangle*\varphi*\psi|_{K}$ in $\DD'^{\{M_p\}}_K$ (here the restriction to $K$ is in fact the transposed mapping of the inclusion $\DD^{\{M_p\}}_K\rightarrow\DD^{\{M_p\}}\left(\RR^d\right)$). Thus, $G_m(\varphi_j,\psi_j)\rightarrow \langle \mathbf{f},\delta_m\rangle*\varphi*\psi|_{K}$ in $C(K)$. By the arbitrariness of $\varphi,\psi\in \DD^{\{M_p\}}_{\Omega, (r_p)}$ and by passing to the limit in the inequality $\left\|G_m(\varphi_j,\psi_j)\right\|_{L^{\infty}(K)}\leq C\|\varphi_j\|_{K,(k_p)}\|\psi_j\|_{K,(k_p)}$, we have $\left\|\langle \mathbf{f},\delta_m\rangle*\varphi*\psi|_K\right\|_{L^{\infty}(K)}\leq C\|\varphi\|_{K,(k_p)}\|\psi\|_{K,(k_p)}$ for all $m\in \ZZ_+$, $\varphi,\psi\in \DD^{\{M_p\}}_{\Omega, (r_p)}$. For the fixed $(r_p)\in\mathfrak{R}$, by the parametrix of Komatsu, there exist ultradifferential operator $P(D)$ of class $\{M_p\}$, $u\in\DD^{\{M_p\}}_{\Omega, (r_p)}$ and $\psi\in \DD^{\{M_p\}}(\Omega)$ such that $\langle \mathbf{f},\delta_m\rangle=P(D)\left(\langle \mathbf{f},\delta_m\rangle*u\right)+\langle \mathbf{f},\delta_m\rangle*\psi$. Applying again the parametrix we have
\beqs
\langle \mathbf{f},\delta_m\rangle=P(D)P(D)\left(\langle \mathbf{f},\delta_m\rangle*u*u\right)+2P(D)\left(\langle \mathbf{f},\delta_m\rangle*\psi*u\right)+\langle \mathbf{f},\delta_m\rangle*\psi*\psi.
\eeqs
Since each of the sets $\{\langle \mathbf{f},\delta_m\rangle*u*u|_K\,|\, m\in\ZZ_+\}$, $\{\langle \mathbf{f},\delta_m\rangle*\psi*u|_K\,|\, m\in\ZZ_+\}$, and $\{\langle \mathbf{f},\delta_m\rangle*\psi*\psi|_K\,|\, m\in\ZZ_+\}$ is bounded in $\DD'^{\{M_p\}}_K$ and, hence, also in $\DD'^{\{M_p\}}(\Omega)$, we obtain that $\{\langle \mathbf{f},\delta_m\rangle|_{\Omega}\,|\,m\in\mathbb{N}\}$ is bounded in $\DD'^{\{M_p\}}(\Omega)$. By the arbitrariness of $\Omega$ it follows that this set is bounded in $\DD'^{\{M_p\}}\left(\RR^d\right)$. Hence, it is relatively compact ($\DD'^{\{M_p\}}(\RR^d)$ is Montel); thus, there exists subsequence $\langle \mathbf{f},\delta_{m_s}\rangle$ which converges to an $f$ in $\DD'^{\{M_p\}}\left(\RR^d\right)$. Since $\langle \mathbf{f},\delta_{m_s}*\chi\rangle=\langle \mathbf{f},\delta_{m_s}\rangle*\check{\chi}$ for each $\chi\in\DD^{\{M_p\}}\left(\RR^d\right)$, after passing to the limit we have $\langle \mathbf{f},\chi\rangle= f*\check{\chi}$.\\

 The implication $(iii)\Rightarrow (i)$ is obvious.
\end{proof}

We also have the following interesting corollary.

\begin{co}\label{convolution corollary}
Let $\mathbf{f}\in\mathcal{D}'^*(\mathbb{R}^{d},E'_{\sigma(E',E)})$, that is, a continuous linear mapping $\mathbf{f}:\mathcal{D}^*(\mathbb{R}^{d})\rightarrow E'_{\sigma(E',E)}$. If $\mathbf{f}$ commutes with every translation in the sense of Proposition \ref{com}, then there exists $f\in \mathcal{D}_{E'_{\ast}}'^*$ such that $\mathbf{f}$ is of the form
\begin{equation}\label{eq:8}
\left\langle\mathbf{f},\varphi\right\rangle=f\ast \check{\varphi}, \ \ \ \varphi\in\mathcal{D}^*(\mathbb{R}^{d}).
\end{equation}
\end{co}

\begin{proof}
Since the inclusion $E'_{\sigma(E',E)}\to\mathcal{D}'^*_{\sigma}(\mathbb{R}^{d})$ is continuous (as the transposed mapping of $\DD^*(\RR^d)\hookrightarrow E$), $\mathbf{f}: \mathcal{D}^*(\mathbb{R}^{d})\to\mathcal{D}'^*_{\sigma}(\mathbb{R}^{d})$ is also a continuous. For $B$ bounded in $\mathcal{D}^*(\mathbb{R}^{d})$, $\mathbf{f}(B)$ is bounded in $\mathcal{D}'^*_{\sigma}(\mathbb{R}^{d})$ and hence bounded in $\mathcal{D}'^*(\mathbb{R}^{d})$. Since $\DD^*\left(\RR^d\right)$ is bornological, $\mathbf{f}:\mathcal{D}^*(\mathbb{R}^{d})\rightarrow \mathcal{D}'^*(\mathbb{R}^{d})$ is continuous. Now the claim follows from Proposition \ref{com} and Theorem \ref{karak}.
\end{proof}

If $F$ is a complete l.c.s., we define $\mathcal{S}'^*\left(\mathbb{R}^{d},F\right)= \SSS'^*(\RR^d)\varepsilon F$. Since $\SSS'^*(\RR^d)$ is nuclear, it satisfies the weak approximation property and we obtain $\mathcal{L}_b\left(\SSS^*(\RR^d),F\right)\cong\SSS'^*(\RR^d)\varepsilon F\cong \SSS'^*(\RR^d) \hat{\otimes} F$. (For the definition of the $\varepsilon$ tensor product, the definition of the weak approximation property, and their connection, we refer to \cite{SchwartzV} and \cite{Komatsu3}.)

 We now embed the ultradistribution space $\mathcal{D}'^*_{E'_{\ast}}$ into the space of $E'$-valued tempered ultradistributions as follows. Define first the continuous injection $\iota:\mathcal{S}'^*(\mathbb{R}^{d})\to\mathcal{S}'^*(\mathbb{R}^{d},\mathcal{S}'^*(\mathbb{R}^{d}))$, where $\iota(f)=\mathbf{f}$ is given by (\ref{eq:8}). Observe the restriction of $\iota$ to $\mathcal{D}'^*_{E'_{\ast}}$, $\iota:\mathcal{D}'^*_{E'_{\ast}}\to\mathcal{S}'^*(\mathbb{R}^{d},E')$ .(The range of $\iota$ is a subset of $\mathcal{S}'^*(\mathbb{R}^{d},E')$ by Theorem \ref{karak} and the remark after it.) Let $B_1$ be an arbitrary bounded subset of $\mathcal{S}^*(\mathbb{R}^d)$. The set $B=\{\psi*\varphi|\,\varphi \in B_1,\|\psi\|_E\leq 1\}$ is bounded in $\mathcal{D}^*_{E}$ (by $(e)$ of Theorem \ref{ttt11}). For  $f\in \mathcal{D}'^*_{E'_{\ast}}$,
\beqs
\sup_{\varphi\in B_1} \|\langle\mathbf{f},\varphi\rangle\|_{E'}= \sup_{\varphi\in B_1} \|f*\check{\varphi}\|_{E'}=\sup_{\varphi\in B_1}\sup_{\|\psi\|_E\leq1}|\langle f,\psi*\varphi\rangle|=\sup_{\chi\in B}|\langle f,\chi\rangle|.
\eeqs
Hence, the mapping $\iota$ is continuous. Furthermore, by $(iii)$ of Theorem \ref{karak}, $\iota(\mathcal{D}_{E'_{\ast}}')\subseteq\mathcal{S}'^*(\mathbb{R}^{d},E'_{\ast})$ and Proposition \ref{com} tells us that $\iota(\mathcal{D}_{E'_{\ast}}'^*)$ is precisely the subspace of $\mathcal{S}'(\mathbb{R}^{d},E'_{\ast})$ consisting of those $\mathbf{f}$ which commute with all translations in the sense of Proposition \ref{com}. Since the translations $T_{h}$ are continuous operators on $E'_{\ast}$, we actually obtain that the range $\iota(\mathcal{D}_{E'_{\ast}}'^*)$ is a closed subspace of $\mathcal{S}'^*(\mathbb{R}^{d},E'_{\ast})$. Note that we may consider $\mathcal{D}'^*(\mathbb{R}^{d})$ instead of $\mathcal{S}'^*(\mathbb{R}^{d})$ in these embeddings.

\begin{co}\label{cor:bounded}
Let $B'\subseteq \DD'^*_{E'_{\ast}}$. The following properties are equivalent:
\begin{itemize}
\item [$(i)$] $B'$ is a bounded subset of $\mathcal{D}_{E'_{\ast}}'^*$.
\item [$(ii)$] $\iota (B')$ is bounded in $\mathcal{S}'^*(\mathbb{R}^{d},E')$ (or equivalently in
    $\mathcal{S}'^*(\mathbb{R}^{d},E'_{\ast})$).
\item [$(iii)$] There exist a bounded subset $\tilde{B}$ of $E'$ and an ultradifferential operator $P(D)$ of class * such that each $f\in B'$ can be represented as $f=P(D)g+g_1$ for some $g,g_1\in \tilde{B}$.
\item [$(iv)$] There are $C>0$ and a finite set $J$ such that every $f\in B'$ admits a representation
    (\ref{eq:representation}) with continuous functions $f_k\in E'_{\ast}\cap UC_{\omega}$ satisfying the uniform
    bounds $\|f_k\|_{E'}\leq C$ and $\|f_k\|_{\infty,\omega}\leq C$. (If $E$ is reflexive one may choose
    $f_k\in E'\cap C_{\omega}$.)
\end{itemize}
\end{co}

\begin{proof} $(i)\Rightarrow(ii)$. It follows from continuity of the mapping $\iota$.

 $(ii)\Rightarrow (iii)$. Let $\Omega$ be bounded open symmetric neighborhood of $0$ in $\RR^d$ and set $K=\overline{\Omega}$. Let $\iota (B')$ be bounded in $\mathcal{S}'^*(\mathbb{R}^{d},E')=\mathcal{L}_b\left(\SSS^*\left(\RR^d\right), E'\right)$. Then it is equicontinuous subset of $\mathcal{L}_b\left(\DD^*_K, E'\right)$. We will continue the proof in the $\{M_p\}$ case, the $(M_p)$ case is similar. There exist $(k_p)\in\mathfrak{R}$ and $C>0$ such that $\|\langle \mathbf{f},\varphi\rangle\|_{E'}\leq C\|\varphi\|_{K,(k_p)}$ for all $\mathbf{f}\in \iota(B')$ and $\varphi\in\DD^{\{M_p\}}_K$, i.e., $\|f*\check{\varphi}\|_{E'}\leq C\|\varphi\|_{K,(k_p)}$ for all $f\in B'$ and $\varphi\in\DD^{\{M_p\}}_K$. By a similar technique as in the proof of Proposition \ref{S'}, one obtains that there exists $(r_p)\in \mathfrak{R}$ such that $\|f*\check{\varphi}\|_{E'}\leq C\|\varphi\|_{K,(k_p)}$ for all $f\in B'$, $\varphi\in \DD^{\{M_p\}}_{\Omega, (r_p)}$. For the fixed $(r_p)\in\mathfrak{R}$, by the parametrix of Komatsu, there exist an ultradifferential operator $P(D)$ of class $\{M_p\}$, $u\in\DD^{\{M_p\}}_{\Omega, (r_p)}$ and $\psi\in \DD^{\{M_p\}}(\Omega)$ such that $f=P(D)(f*u)+f*\psi$. By what we proved above $\{f*u|\, f\in B'\}$ and $\{f*\psi|\, f\in B'\}$ are bounded in $E'$ and $(iii)$ follows.

 $(ii)\Rightarrow (iv)$. Proceed as in $(ii)\Rightarrow (iii)$ and then use the same technique as in the proof of $(ii)\Rightarrow (v)$ of Theorem \ref{karak}.

  $(iii)\Rightarrow (i)$ and $(iv)\Rightarrow (i)$ are obvious.
\end{proof}

\begin{co}\label{cor:sequence}
Let $\left\{f_{j}\right\}_{j=0}^{\infty}\subseteq\mathcal{D}_{E'_{\ast}}'^*$ $($or similarly, a filter with a countable or bounded basis$)$. The following three statements are equivalent:
\begin{itemize}
\item [$(i)$]  $\left\{f_{j}\right\}_{j=0}^{\infty}$ is (strongly) convergent in $\mathcal{D}_{E'_{\ast}}'^*$.
\item [$(ii)$]  $\left\{\iota(f_{j})\right\}_{j=0}^{\infty}$ is convergent in $\mathcal{S}'^*(\mathbb{R}^{d},E')$ $($or
    equivalently in $\mathcal{S}'^*(\mathbb{R}^{d},E'_{\ast})$$)$.
\item [$(iii)$] There exist convergent sequences $\{g_j\}_j,\{\tilde{g}_j\}_j$ in $E'$ and an ultradifferential operator $P(D)$ of class $\ast$ such that each $f_j=P(D) g_j+\tilde{g}_j$.
\item [$(iv)$] There exist $N\in\ZZ_+$, sequences $\{g^{(k)}_j\}_j$, $k=1,...,N$, in $E'_{\ast}\cap UC_{\omega}$ each convergent in $E'_*$ and in $L^{\infty}_{\omega}$ and ultradifferential operators $P_k(D)$, $k=1,...,N$, of class $\ast$ such that $\ds f_j=\sum_{k=1}^{N}P_k(D) g^{(k)}_j$. (If $E$ is reflexive one may choose $g^{(k)}_j\in E'\cap C_{\omega}$.)
\end{itemize}
\end{co}

\begin{proof} The proof is similar to the proof of the above corollary and we omit it.
\end{proof}

Observe that Corollaries \ref{cor:bounded} and \ref{cor:sequence} are still valid if  $\mathcal{S}'^*(\RR^d)$ is replaced by
$\mathcal{D}'^*(\RR^d)$.

 At the beginning of Section \ref{tspace}, we defined the spaces $\tilde{\DD}^{\{M_p\},(r_p)}_E$ and $\tilde{\DD}^{\{M_p\}}_E$. As we saw there $\DD^{\{M_p\}}_E$ and $\tilde{\DD}^{\{M_p\}}_E$ are equal as sets and the former has a stronger topology than the latter. In fact we will prove that these are also topologically isomorphic.

\begin{te}\label{1517}
The spaces $\DD^{\{M_p\}}_E$ and $\tilde{\DD}^{\{M_p\}}_E$ are isomorphic as l.c.s.
\end{te}

\begin{proof} By the above considerations, it is enough to prove that the topology of $\tilde{\DD}^{\{M_p\}}_E$ is stronger than the topology of $\DD^{\{M_p\}}_E$. Let $V$ be a neighborhood of zero in $\DD^{\{M_p\}}_E$. Since $\DD^{\{M_p\}}_E$ is complete and barreled, its topology is in fact the topology $b\left(\DD'^{\{M_p\}}_{E'_*}, \DD^{\{M_p\}}_E\right)$. Hence, we can assume that $V=B^{\circ}$ for a bounded set $B$ in $\DD'^{\{M_p\}}_{E'_*}$ ($B^{\circ}$ is the polar of $B$), that is, $\ds V=\left\{\varphi\in\DD^{\{M_p\}}_E\Big|\, \sup_{T\in B}|\langle T,\varphi\rangle|\leq 1\right\}$. By Corollary \ref{cor:bounded} there exists $C>0$ and a finite set $J$ such that every $T\in B$ admits a representation (\ref{eq:representation}) with continuous functions $f_k\in E'_{\ast}\cap UC_{\omega}$ satisfying the uniform bounds $\|f_k\|_{E'}\leq C$. Since $P_k(D)$ are continuous on $\tilde{\DD}^{\{M_p\}}_E$ (Proposition \ref{udozc}), there exist $(r_p)\in\mathfrak{R}$ and $C_1>0$ such that $\|P_k(-D)\varphi\|_E\leq C_1 \|\varphi\|_{E,(r_p)}$ for all $k\in J$, $\varphi\in \tilde{\DD}^{\{M_p\}}_E$. Set $N=|J|$ and let $W=\left\{\varphi\in\tilde{\DD}^{\{M_p\}}_E\big|\, \|\varphi\|_{E,(r_p)}\leq 1/(CC_1N)\right\}$ be a neighborhood of zero in $\tilde{\DD}^{\{M_p\}}_E$. If $\varphi\in W$, then for $T\in B$ one easily obtains $|\langle T,\varphi\rangle|\leq 1$, that is, $\varphi\in V$. Hence, we obtain the desired result.
\end{proof}

When $E$ is reflexive, the space $\mathcal{D}^*_{E}$ is also reflexive. Furthermore, we have:

\begin{pro}\label{prop:reflexive}
If $E$ is reflexive, then $\mathcal{D}^{(M_p)}_{E}$ and $\mathcal{D}'^{\{M_p\}}_{E'}$ are $(FS^*)$-spaces, and $\mathcal{D}^{\{M_p\}}_{E}$ and $\mathcal{D}'^{(M_p)}_{E}$ are $(DFS^*)$-spaces. Consequently, they are reflexive. In addition, $\mathcal{S}^*(\mathbb{R}^{d})$ is dense in $\mathcal{D}'^*_{E'}$.
\end{pro}

\begin{proof} Let $\tilde{\tilde{\DD}}^{\{M_p\},m}_E$ be the $(B)$-space of all $\varphi\in\DD'^*\left(\RR^d\right)$ such that $D^{\alpha}\varphi\in E$, $\forall\alpha\in\NN^d$ and
\beqs
\|\varphi\|_{E,m}=\left(\sum_{\alpha}\frac{m^{2|\alpha|}}{M_{\alpha}^2}\|D^{\alpha}\varphi\|^2_E\right)^{1/2}<\infty.
\eeqs
Then, we have the obvious continuous inclusions $\tilde{\tilde{\DD}}^{\{M_p\},m}_E\rightarrow \DD^{\{M_p\},m}_E$ and $\DD^{\{M_p\},2m}_E\rightarrow \tilde{\tilde{\DD}}^{\{M_p\},m}_E$. Hence, $\ds\mathcal{D}_E^{(M_p)}=\lim_{\substack{\longleftarrow\\ m\rightarrow\infty}}\tilde{\tilde{\mathcal{D}}}_E^{\{M_p\},m}$ and $\ds\mathcal{D}_E^{\{M_p\}}=\lim_{\substack{\longrightarrow\\ m\rightarrow 0}}\tilde{\tilde{\mathcal{D}}}_E^{\{M_p\},m}$. If $l^2_m(E)$ is the $(B)$-space of all $(\psi_{\alpha})_{\alpha\in\NN^d}$ with $\psi_{\alpha}\in E$ and norm $\ds\|(\psi_{\alpha})_{\alpha}\|_{l^2_m(E)}=\left(\sum_{\alpha\in\NN^d} \frac{m^{2|\alpha|}}{M^2_{\alpha}}\|\psi_{\alpha}\|_E^2\right)^{1/2}$, then $l^2_m(E)$ is reflexive since $E$ is (cf. \cite[Thm. 2, p. 360]{kothe11}). Observe that $\tilde{\tilde{\DD}}^{\{M_p\},m}_E$ is isometrically injected into a closed subspace of $l^2_m(E)$ by the mapping $\varphi\mapsto \left(D^{\alpha}\varphi\right)_{\alpha}$; hence, $\tilde{\tilde{\DD}}^{\{M_p\},m}_E$ is reflexive. Thus, $\DD^{(M_p)}_E$ is an $(FS^*)$-space and $\DD^{\{M_p\}}_E$ is a $(DFS^*)$-space. In particular, they are reflexive, $\DD'^{(M_p)}_E$ is a $(DFS^*)$-space, and $\DD'^{\{M_p\}}_E$ is an $(FS^*)$-space. Now, the density of $\SSS^*(\RR^d)$ in $\DD'^*_{E'}$ is an easy consequence of the Hahn-Banach theorem.
\end{proof}

%

\section{The weighted spaces $\mathcal{D}^{\ast}_{L^{p}_{\eta}}$ and $\mathcal{D}'^{\ast}_{L^{p}_{\eta}}$}
\label{examples}
As examples, in this section we discuss the weighted spaces $\mathcal{D}^{\ast}_{L^{p}_{\eta}}$ and $\mathcal{D}'^{\ast}_{L^{p}_{\eta}},$ which are particular examples of the spaces $\mathcal{D}^*_{E}$ and $\mathcal{D}'^*_{E'_{\ast}}$. They turn out to be important in the study of properties of the general $\mathcal{D}'^*_{E'_{\ast}}$ and general convolution in $\mathcal{D}'^{\ast}(\mathbb{R}^{d})$ (cf. Section \ref{section convolution}).

Let $\eta$ be an \emph{ultrapolynomially bounded weight of class $\ast$}, that is, a (Borel) measurable function $\eta:\mathbb{R}^d\rightarrow (0,\infty)$ that fulfills the requirement $\eta(x+h)\leq C\eta(x)e^{M(\tau|h|)}$, for some $C,\tau>0$ (for every $\tau>0$ and a corresponding $C=C_{\tau}>0$). An interesting nontrivial example in the $(M_p)$ case is given by the function $\eta(x)=e^{\tilde{\eta}(|x|)}$ where $\tilde{\eta}:[0,\infty)\rightarrow[0,\infty)$ is defined by $\ds\tilde{\eta}(\rho)=\rho\int_{\rho}^{\infty}\frac{M(s)}{s^2}ds$. To see this, observe that $\tilde{\eta}$ is a differentiable function with nonnegative monotonically decreasing derivative. Hence $\tilde{\eta}$ is a concave monotonically increasing function and $\tilde{\eta}(0)=0$. Also, it is easy to see that $M(\rho)\leq \tilde{\eta}(\rho)$ and $\tilde{\eta}(\rho+\lambda)\leq \tilde{\eta}(\rho)+\tilde{\eta}(\lambda)$, for all $\rho,\lambda>0$. By $(M.3)$ and \cite[Prop. 4.4]{Komatsu1} there exist $C,C_1>0$ such that $\tilde{\eta}(\rho)\leq M(C\rho)+C_1$, for all $\rho>0$. For the $\{M_p\}$ case take $(r_p)\in\mathfrak{R}$ and perform the same construction with the sequence $N_p$ defined by $N_0=1$ and $N_p=M_p\prod_{j=1}^p r_j$, $p\in\ZZ_+$, which obviously satisfies $(M.1)$ and $(M.3)$ since $M_p$ does.

 For $1\leq p<\infty$ we denote as $L^p_{\eta}$ the space of measurable functions $g$ such that $\|\eta g\|_{p}<\infty$. Clearly $L^p_{\eta}$ are translation-invariant spaces of tempered ultradistributions for $p\in[1,\infty)$. In the case $p=\infty$, we define $L^\infty_{\eta}$ via the norm $\|g/\eta\|_{\infty}$; the axiom (I) clearly fails for $L^\infty_{\eta}$ since  $\mathcal{D}^*(\mathbb{R}^{d})$ is not dense in $L^{\infty}_{\eta}$. In the next considerations the number $q$ always stands for $p^{-1}+q^{-1}=1$ ($p\in[1,\infty]$). Of course $(L^{p}_{\eta})'=L^{q}_{\eta^{-1}}$ if $1< p<\infty$ and $(L^{1}_{\eta})'=L^{\infty}_{\eta}$. In view of Proposition \ref{thgoodspace}, the space $E'_{\ast}$ corresponding to $E=L^{p}_{\eta^{-1}}$ is $E'_{\ast}=L^{q}_{\eta}$ whenever $1<p<\infty$. On the other hand, $(iii)$ of Theorem \ref{thgoodspace} gives
that $E'_{\ast}=UC_{\eta}$ for  $E=L^{1}_{\eta}$, where $UC_{\eta}$ is defined as in (\ref{UC}) with $\omega$ replaced by
$\eta$. We will also consider the Banach space $\ds C_{\eta}=\left\{g\in C(\mathbb{R}^{d})\big|\, \lim_{|x|\to \infty}g(x)/\eta(x)=0\right\}\subset UC_{\eta}\subset L^{\infty}_{\eta}.$

The weight function of $L^{p}_{\eta}$ can be explicitly determined as in \cite[Prop. 10]{DPV}.

\begin{pro}\label{weight prop}
Let $\omega_{\eta}(h):= \operatorname*{ess}\sup_{x\in\mathbb{R}^{d}} \eta(x+h)/\eta(x)$. Then
$$\|T_{-h}\|_{\mathcal{L}(L^{p}_{\eta})}=
\begin{cases}
\omega_{\eta}(h) & \mbox{ if } p\in[1,\infty),
\\ \omega_{\eta}(-h)
  & \mbox{ if } p=\infty.
\end{cases}
$$
Consequently, the Beurling algebra associated to $L_{\eta}^{p}$ is $L_{\omega_{\eta}}^{1}$ if $p=[1,\infty)$ and $L^{1}_{\check{\omega}_{\eta}}$ if $p=\infty$.
\end{pro}

\begin{proof} See the proof of \cite[Prop. 10]{DPV}.
\end{proof}

Observe that when the logarithm of $\eta$ is a subadditive function with $\eta(0)=1$, one easily obtains from Proposition \ref{weight prop} that $\omega_{\eta}=\eta$ (a.e.).

Consider now the spaces $\mathcal{D}^*_{L^{p}_{\eta}}$ for $p\in[1,\infty]$ and $\tilde{\DD}^{\{M_p\}}_{L^{\infty}_{\eta}}$ defined as in Section \ref{tspace} by taking $E=L^{p}_{\eta}$. Once again, the case $p=\infty$ is an exception since $\mathcal{D}^*(\mathbb{R}^{d})$ is not dense in $\mathcal{D}^*_{L^{\infty}_{\eta}}$ nor in $\tilde{\DD}^{\{M_p\}}_{L^{\infty}_{\eta}}$. Nevertheless, we can repeat the proof of Proposition \ref{regular} to prove that $\DD^{\{M_p\}}_{L^{\infty}_{\eta}}$ is regular and complete. One can show that each ultradifferential operator of $\ast$ class acts continuously on $\DD^*_{L^{\infty}_{\eta}}$ and each ultradifferential operator of class $\{M_p\}$ acts continuously on $\tilde{\DD}^{\{M_p\}}_{L^{\infty}_{\eta}}$ (cf. the proof of Proposition \ref{udozc}). Obviously $\DD^{\{M_p\}}_{L^{\infty}_{\eta}}$ is continuously injected into $\tilde{\DD}^{\{M_p\}}_{L^{\infty}_{\eta}}$ and by using \cite[Lem. 3.4]{Komatsu3} and employing a similar technique as in the proof of Proposition \ref{regular}, one can prove that this inclusion is in fact surjective. We will also use the notation $\mathcal{B}^*_{\eta}$ for the space $\DD^*_{L^{\infty}_{\eta}}$ and we denote by $\dot{\mathcal{B}}^*_{\eta}$ the closure of $\DD^*(\RR^d)$ in $\mathcal{B}^*_{\eta}$. We denote by $\dot{\tilde{\mathcal{B}}}^{\{M_p\}}_{\eta}$ the closure of $\DD^{\{M_p\}}(\RR^d)$ in $\tilde{\DD}^{\{M_p\}}_{L^{\infty}_\eta}$. It is important to notice that in the case $\eta=1$ these spaces were considered in \cite{Pilipovic} (see also \cite{PilipovicK}).

We immediately see that $\dot{\mathcal{B}}^{(M_p)}_{\eta}=\mathcal{D}^{(M_p)}_{C_{\eta}}$. In the $\{M_p\}$ case this is not trivial. The following theorem gives that result.

\begin{te}\label{edn}
The spaces $\mathcal{D}^{\{M_p\}}_{C_{\eta}}$, $\dot{\mathcal{B}}^{\{M_p\}}_{\eta}$ and $\dot{\tilde{\mathcal{B}}}^{\{M_p\}}_{\eta}$ are isomorphic to each other as l.c.s..
\end{te}

\begin{proof} By Proposition \ref{regular}, $\mathcal{D}^{\{M_p\}}_{C_{\eta}}$ is a complete barreled $(DF)$-space. First we prove that $\mathcal{D}^{\{M_p\}}_{C_{\eta}}$ and $\dot{\tilde{\mathcal{B}}}^{\{M_p\}}_{\eta}$ are isomorphic l.c.s.. Observe that $\mathcal{D}^{\{M_p\}}_{C_{\eta}}\subseteq \tilde{\DD}^{\{M_p\}}_{L^{\infty}_{\eta}}$. Moreover, by Theorem \ref{1517}, the topology of $\mathcal{D}^{\{M_p\}}_{C_{\eta}}$ is the same as the induced topology on $\mathcal{D}^{\{M_p\}}_{C_{\eta}}$ by $\tilde{\DD}^{\{M_p\}}_{L^{\infty}_{\eta}}$. Since $\DD^{\{M_p\}}(\RR^d)$ is dense in $\mathcal{D}^{\{M_p\}}_{C_{\eta}}$ and $\dot{\tilde{\mathcal{B}}}^{\{M_p\}}_{\eta}$ is the closure of $\DD^{\{M_p\}}(\RR^d)$ in the complete l.c.s. $\tilde{\DD}^{\{M_p\}}_{L^{\infty}_{\eta}}$, the spaces $\mathcal{D}^{\{M_p\}}_{C_{\eta}}$ and $\dot{\tilde{\mathcal{B}}}^{\{M_p\}}_{\eta}$ are isomorphic l.c.s. and the canonical inclusion $\mathcal{D}^{\{M_p\}}_{C_{\eta}}\rightarrow \tilde{\DD}^{\{M_p\}}_{L^{\infty}_{\eta}}$ gives the isomorphism. Now, observe that the inclusion $\mathcal{D}^{\{M_p\}}_{C_{\eta}}\rightarrow \DD^{\{M_p\}}_{L^{\infty}_{\eta}}$ is continuous. Since $\DD^{\{M_p\}}(\RR^d)$ is dense in $\mathcal{D}^{\{M_p\}}_{C_{\eta}}$ and $\dot{\mathcal{B}}^{\{M_p\}}_{\eta}$, $\mathcal{D}^{\{M_p\}}_{C_{\eta}}\subseteq \dot{\mathcal{B}}^{\{M_p\}}_{\eta}$ and the inclusion is continuous. Also, since the inclusion $\DD^{\{M_p\}}_{L^{\infty}_{\eta}}\rightarrow \tilde{\DD}^{\{M_p\}}_{L^{\infty}_{\eta}}$ is continuous and $\DD^{\{M_p\}}(\RR^d)$ is dense in $\dot{\mathcal{B}}^{\{M_p\}}_{\eta}$ and $\dot{\tilde{\mathcal{B}}}^{\{M_p\}}_{\eta}$, we obtain that $\dot{\mathcal{B}}^{\{M_p\}}_{\eta}\subseteq \dot{\tilde{\mathcal{B}}}^{\{M_p\}}_{\eta}$ and the inclusion is continuous. But, since we already proved that the inclusion $\mathcal{D}^{\{M_p\}}_{C_{\eta}}\rightarrow \dot{\tilde{\mathcal{B}}}^{\{M_p\}}_{\eta}$ is a topological isomorphism onto, we obtain that the inclusion $\mathcal{D}^{\{M_p\}}_{C_{\eta}}\rightarrow \dot{\mathcal{B}}^{\{M_p\}}_{\eta}$ is as well.
\end{proof}

Proposition \ref{smooth prop} together with the estimate (\ref{bounds}) in the $(M_p)$ case and with the estimate (\ref{boundss}) in the $\{M_p\}$ case implies $\mathcal{D}^*_{L^{p}_{\eta}}\hookrightarrow\dot{\mathcal{B}}^*_{\check{\omega}_{\eta}}$ for every $p\in[1,\infty)$. It follows from Proposition \ref{prop:reflexive} that $\mathcal{D}^{*}_{L^p_\eta}$ is reflexive when $p\in(1,\infty)$.

 In accordance with Section \ref{subsection DE}, the weighted spaces $\mathcal{D}'^*_{L_{\eta}^p}$ are defined as
$\mathcal{D}'^*_{L_{\eta}^p}= (\mathcal{D}^*_{L_{\eta^{-1}}^q})'$ where $p^{-1}+q^{-1}=1$ if $p\in (1,\infty]$; if $p=1$, then
$\mathcal{D}'^*_{L_{\eta}^1}=(\mathcal{D}^*_{C_{\eta}})'=(\dot{\mathcal{B}}^*_\eta)'$. We write
$\mathcal{B}'^*_{\eta}=\mathcal{D}'^*_{L^{\infty}_{\eta}}$ and $\dot{\mathcal{B}}'^*_{\eta}$ for the closure of
$\mathcal{D}^*(\mathbb{R}^{n})$ in $\mathcal{B}'^*_{\eta}$.

The dual of $E=C_{\eta}$ is the space $\mathcal{M}^1_{\eta}$ consisting of all elements $\nu\in \left(C_c\left(\RR^d\right)\right)'$ which are of the form $d\nu= \eta^{-1} d\mu$, for $\mu\in\mathcal{M}^1$ (i.e., a finite measure) and the norm is $\|\nu\|_{\mathcal{M}^1_{\eta}}=\|\mu\|_{\mathcal{M}^1}$. Observe that then $E'_{\ast}=L^1_{\eta}$. In this case, by using Theorem \ref{karak}, similarly as in the case of distributions (see \cite[p. 99]{S1955}, \cite[p. 196]{SchwartzV}), one can prove that the bidual of $\dot{\mathcal{B}}^{(M_p)}_{\eta}$ is isomorphic to $\DD^{(M_p)}_{L^{\infty}_\eta}$ as l.c.s. and that $\dot{\mathcal{B}}^{(M_p)}_{\eta}$ is a distinguished $(F)$-space, namely, $\DD'^{(M_p)}_{L^1_{\eta}}$ is barreled and bornological. In the $\{M_p\}$ case, observe that $\DD'^{\{M_p\}}_{L^1_{\eta}}$ is an $(F)$-space as the strong dual of a barreled $(DF)$-space. Moreover, we have the following theorem.

\begin{te}\label{te B}
The bidual of $\dot{\mathcal{B}}^{\{M_p\}}_{\eta}$ is isomorphic to $\DD^{\{M_p\}}_{L^{\infty}_{\eta}}$ as l.c.s.. Moreover $\DD^{\{M_p\}}_{L^{\infty}_{\eta}}$ and $\tilde{\DD}^{\{M_p\}}_{L^{\infty}_{\eta}}$ are isomorphic l.c.s..
\end{te}

\begin{proof} First note that $\eta$ can be assumed to be continuous. (The continuous weight $\eta_{1}=\eta\ast \varphi$ defines equivalent norms if we choose $\varphi\in\mathcal{D}(\mathbb{R}^{d})$ to be nonnegative with $\int_{\mathbb{R}^{d}}\varphi(x)dx=1$.) We already saw that $\DD^{\{M_p\}}_{L^{\infty}_{\eta}}$ and $\tilde{\DD}^{\{M_p\}}_{L^{\infty}_{\eta}}$ are equal as sets. First we prove that the bidual of $\dot{\mathcal{B}}^{\{M_p\}}_{\eta}$ is isomorphic to $\tilde{\DD}^{\{M_p\}}_{L^{\infty}_{\eta}}$. Since $\EE'^{\{M_p\}}(\RR^d)$ is continuously and densely injected into $\DD'^{\{M_p\}}_{L^1_{\eta}}$ (the density can be proved by using cut-off functions and Theorem \ref{karak}) we have the continuous inclusion $\left(\DD'^{\{M_p\}}_{L^1_{\eta}}\right)'_{b}\rightarrow \EE^{\{M_p\}}(\RR^d)$ (where $b$ stands for the strong topology). Let $(r_p)\in\mathfrak{R}$ and set $R_{\alpha}=\prod_{j=1}^{|\alpha|}r_j$. Observe the set $\ds B=\left\{\frac{(\eta(a))^{-1}D^{\alpha}\delta_a}{M_{\alpha}R_{\alpha}}\Big|\, a\in\RR^d,\, \alpha\in\NN^d\right\}$. One easily proves that it is a bounded subset of $\DD'^{\{M_p\}}_{L^1_{\eta}}$. Hence, if $\psi\in \left(\DD'^{\{M_p\}}_{L^1_{\eta}}\right)'_{b}$, then $\psi(B)$ is bounded in $\CC$ and hence
\beqs
\sup_{a,\alpha}\frac{\left|(\eta(a))^{-1}D^{\alpha}\psi(a)\right|}{M_{\alpha}R_{\alpha}}=\sup_{T\in B}\left|\langle \psi, T\rangle\right|<\infty.
\eeqs
We obtain that $\left(\DD'^{\{M_p\}}_{L^1_{\eta}}\right)'\subseteq \DD^{\{M_p\}}_{L^{\infty}_{\eta}}$ and the inclusion $\left(\DD'^{\{M_p\}}_{L^1_{\eta}}\right)'_{b}\rightarrow \tilde{\DD}^{\{M_p\}}_{L^{\infty}_{\eta}}$ is continuous.

 Let $\psi\in \DD^{\{M_p\}}_{L^{\infty}_{\eta}}$. If $T\in \DD'^{\{M_p\}}_{L^1_{\eta}}$, then by Theorem \ref{karak} there exist an ultradifferential operator $P(D)$ of class $\{M_p\}$ and $f,f_1\in \mathcal{M}^1_{\eta}$ such that $T=P(D)f+f_1$. Let $df= \eta^{-1}dg$ and $df_1= \eta^{-1}dg_1$ for $g,g_1\in\mathcal{M}^1$. Define $S_{\psi}$ by
\beqs
S_{\psi}(T)=\int_{\RR^d}\frac{P(-D)\psi(x)}{\eta(x)}dg+\int_{\RR^d}\frac{\psi(x)}{\eta(x)}dg_1.
\eeqs
Obviously, the integrals on the right-hand side are absolutely convergent. We will prove that $S_{\psi}$ is a well defined element of $\left(\DD'^{\{M_p\}}_{L^1_{\eta}}\right)'$. Let $\tilde{P}(D)$, $\tilde{f},\tilde{f}_1\in\mathcal{M}^1_{\eta}$ be such that $T=\tilde{P}(D)\tilde{f}+\tilde{f}_1$ and let $d\tilde{f}= \eta^{-1}d\tilde{g}$ and $d\tilde{f}_1= \eta^{-1}d\tilde{g}_1$ for $\tilde{g},\tilde{g}_1\in\mathcal{M}^1$. Let $\chi\in\DD^{\{M_p\}}(\RR^d)$ be a function such that $\chi=1$ on the closed unit ball with center at $0$ and $\chi=0$ on $\{x\in\RR^d|\, |x|>2\}$. Set $\psi_n(x)=\chi(x/n)\psi(x)$, $n\in\ZZ_+$. Then it is easy to verify that
\beqs
\int_{\RR^d}\frac{P(-D)\psi_n(x)}{\eta(x)}dg\rightarrow \int_{\RR^d}\frac{P(-D)\psi(x)}{\eta(x)}dg&,&\, \int_{\RR^d}\frac{\psi_n(x)}{\eta(x)}dg_1\rightarrow\int_{\RR^d}\frac{\psi(x)}{\eta(x)}dg_1,\\
\int_{\RR^d}\frac{\tilde{P}(-D)\psi_n(x)}{\eta(x)}d\tilde{g}\rightarrow \int_{\RR^d}\frac{\tilde{P}(-D)\psi(x)}{\eta(x)}d\tilde{g}&,&\, \int_{\RR^d}\frac{\psi_n(x)}{\eta(x)}d\tilde{g}_1\rightarrow\int_{\RR^d}\frac{\psi(x)}{\eta(x)}d\tilde{g}_1,
\eeqs
when $n\rightarrow\infty$. Also, observe that for each $n\in\ZZ_+$
\beqs
\int_{\RR^d}\frac{P(-D)\psi_n(x)}{\eta(x)}dg+\int_{\RR^d}\frac{\psi_n(x)}{\eta(x)}dg_1= \int_{\RR^d}\frac{\tilde{P}(-D)\psi_n(x)}{\eta(x)}d\tilde{g}+\int_{\RR^d}\frac{\psi_n(x)}{\eta(x)} d\tilde{g}_1,
\eeqs
since both terms are equal to $\langle T,\psi_n\rangle$ in the sense of the duality $\left\langle \DD^{\{M_p\}}(\RR^d),\DD'^{\{M_p\}}(\RR^d)\right\rangle$. Hence, $S_{\psi}$ is a well defined mapping $\DD'^{\{M_p\}}_{L^1_{\eta}}\rightarrow \CC$, since it does not depend on the representation of $T$. To prove that it is continuous it is enough to prove that it maps bounded sets into bounded sets, since $\DD'^{\{M_p\}}_{L^1_{\eta}}$ is an $(F)$-space. Let $B$ be a bounded set in $\DD'^{\{M_p\}}_{L^1_{\eta}}$. By Corollary \ref{cor:bounded}, there exist an ultradifferential operator $P(D)$ of class $\{M_p\}$ and bounded subset $B_1$ of $\mathcal{M}^1_{\eta}$ such that each $T\in B$ can be represented by $T=P(D)f+f_1$ for some $f,f_1\in B_1$. By the way we defined $S_{\psi}$, it is easy to verify that $S_{\psi}(B)$ is bounded in $\CC$, so $S_{\psi}\in \left(\DD'^{\{M_p\}}_{L^1_{\eta}}\right)'$. We obtain that $\left(\DD'^{\{M_p\}}_{L^1_{\eta}}\right)'=\tilde{\DD}^{\{M_p\}}_{L^{\infty}_{\eta}}$ as sets and $\left(\DD'^{\{M_p\}}_{L^1_{\eta}}\right)'_{b}$ has stronger topology than the latter. Let $V=B^{\circ}$ be a neighborhood of zero in $\left(\DD'^{\{M_p\}}_{L^1_{\eta}}\right)'_{b}$ for $B$ a bounded subset of $\DD'^{\{M_p\}}_{L^1_{\eta}}$. By Corollary \ref{cor:bounded}, there exist an ultradifferential operator $P(D)$ of class $\{M_p\}$ and a bounded subset $B_1$ of $\mathcal{M}^1_{\eta}$ such that each $T\in B$ can be represented by $T=P(D)f+f_1$ for some $f,f_1\in B_1$. There exists $C_1\geq 1$ such that $\|\tilde{g}\|_{\mathcal{M}^1_{\eta}}\leq C_1$ for all $\tilde{f}\in B_1$. Also, since $P(D)=\sum_{\alpha}c_{\alpha}D^{\alpha}$ is of class $\{M_p\}$, there exist $(r_p)\in\mathfrak{R}$ and $C_2\geq 1$ such that $\left|c_{\alpha}\right|\leq C_2/(M_{\alpha}R_{\alpha})$ (see the proof of Proposition \ref{udozc}). Observe the neighborhood of zero $\ds W=\left\{\psi\in \tilde{\DD}^{\{M_p\}}_{L^{\infty}_{\eta}}\Big|\, \sup_{x,\alpha}\frac{\left|(\eta(x))^{-1}D^{\alpha}\psi(x)\right|}{M_{\alpha}\prod_{j=1}^{|\alpha|}(r_j/2)}\leq \frac{1}{2C_1C_2C_3}\right\}$ in $\tilde{\DD}^{\{M_p\}}_{L^{\infty}_{\eta}}$, where we set $C_3=\sum_{\alpha}2^{-|\alpha|}$. One easily verifies that $W\subseteq V$. We obtain that $\left(\DD'^{\{M_p\}}_{L^1_{\eta}}\right)'_{b}$ and $\tilde{\DD}^{\{M_p\}}_{L^{\infty}_{\eta}}$ are isomorphic l.c.s.. Hence, $\tilde{\DD}^{\{M_p\}}_{L^{\infty}_{\eta}}$ is a complete $(DF)$-space (since $\DD'^{\{M_p\}}_{L^1_{\eta}}$ is an $(F)$-space). Obviously, the identity mapping $\DD^{\{M_p\}}_{L^{\infty}_{\eta}}\rightarrow\tilde{\DD}^{\{M_p\}}_{L^{\infty}_{\eta}}$ is continuous and bijective. Since $\tilde{\DD}^{\{M_p\}}_{L^{\infty}_{\eta}}$ is a $(DF)$-space, to prove the continuity of the inverse mapping it is enough to prove that its restriction to every bounded subset of $\tilde{\DD}^{\{M_p\}}_{L^{\infty}_{\eta}}$ is continuous (see \cite[Cor. 6.7, p. 155]{Sch}). If $B$ is a bounded subset of $\tilde{\DD}^{\{M_p\}}_{L^{\infty}_{\eta}}$ then for every $(r_p)\in\mathfrak{R}$, $\ds \sup_{\psi\in B} \sup_{\alpha}\frac{\left\|D^{\alpha}\psi\right\|_{L^{\infty}_{\eta}(\RR^d)}}{M_{\alpha}R_{\alpha}}<\infty$. Hence, by \cite[Lem. 3.4]{Komatsu3}, there exists $h>0$ such that $\ds \sup_{\psi\in B} \sup_{\alpha}\frac{h^{|\alpha|}\left\|D^{\alpha}\psi\right\|_{L^{\infty}_{\eta}(\RR^d)}}{M_{\alpha}}<\infty$, that is, $B$ is bounded in $\DD^{\{M_p\}}_{L^{\infty}_{\eta}}$. Since every bounded subset of $\DD^{\{M_p\}}_{L^{\infty}_{\eta}}$ is obviously bounded in $\tilde{\DD}^{\{M_p\}}_{L^{\infty}_{\eta}}$, $\DD^{\{M_p\}}_{L^{\infty}_{\eta}}$ and $\tilde{\DD}^{\{M_p\}}_{L^{\infty}_{\eta}}$ have the same bounded sets. Let $\psi_{\lambda}$ be a bounded net in $\tilde{\DD}^{\{M_p\}}_{L^{\infty}_{\eta}}$ which converges to $\psi$ in $\tilde{\DD}^{\{M_p\}}_{L^{\infty}_{\eta}}$. Then there exist $0<h\leq 1$ and $C>0$ such that
\beqs
\sup_{\lambda} \sup_{\alpha}\frac{h^{|\alpha|}\left\|D^{\alpha}\psi_{\lambda}\right\|_{L^{\infty}_{\eta}}}{M_{\alpha}}\leq C \mbox{ and } \sup_{\alpha}\frac{h^{|\alpha|}\left\|D^{\alpha}\psi\right\|_{L^{\infty}_{\eta}}}{M_{\alpha}}\leq C.
\eeqs
Fix $0<h_1<h$. Let $\varepsilon>0$ be arbitrary but fixed. Take $p_0\in\ZZ_+$ such that $(h_1/h)^{|\alpha|}\leq \varepsilon/(2C)$ for all $|\alpha|\geq p_0$. Since $\psi_{\lambda}\rightarrow \psi$ in $\tilde{\DD}^{\{M_p\}}_{L^{\infty}_{\eta}}$, for the sequence $r_p=p$, $p\in\ZZ_+$, there exists $\lambda_0$ such that for all $\lambda\geq \lambda_0$ we have $\ds \sup_{\alpha}\frac{\left\|D^{\alpha}\left(\psi_{\lambda}-\psi\right)\right\|_{L^{\infty}_{\eta}}}{M_{\alpha}R_{\alpha}}\leq \frac{\varepsilon}{p_0!}$. Then for $|\alpha|<p_0$, we have $\ds \frac{h_1^{|\alpha|}\left\|D^{\alpha}\left(\psi_{\lambda}-\psi\right)\right\|_{L^{\infty}_{\eta}}}{M_{\alpha}}\leq \varepsilon$. For $|\alpha|\geq p_0$, we have
\beqs
\frac{h_1^{|\alpha|}\left\| D^{\alpha}\left(\psi_{\lambda}-\psi\right)\right\|_{L^{\infty}_{\eta}}}{M_{\alpha}}\leq 2C\left(\frac{h_1}{h}\right)^{|\alpha|}\leq \varepsilon.
\eeqs
It follows that $\psi_{\lambda}\rightarrow \psi$ in $\DD^{\{M_p\},h_1}_{L^{\infty}_{\eta}}$ and hence in $\DD^{\{M_p\}}_{L^{\infty}_{\eta}}$. We obtain that the topology induced by $\tilde{\DD}^{\{M_p\}}_{L^{\infty}_{\eta}}$ on every bounded subset of $\tilde{\DD}^{\{M_p\}}_{L^{\infty}_{\eta}}$ is stronger than the topology induced by $\DD^{\{M_p\}}_{L^{\infty}_{\eta}}$. Hence, the identity mapping $\tilde{\DD}^{\{M_p\}}_{L^{\infty}_{\eta}}\rightarrow \DD^{\{M_p\}}_{L^{\infty}_{\eta}}$ is continuous.
\end{proof}

\section{Convolution of ultradistributions}
We now apply our results to the study of the convolution of ultradistributions.
\label{section convolution}
\subsection{Convolution of Roumieu ultradistributions}
\label{subsection convolution 1}

As an application of Theorem \ref{te B} when $\eta=1$, we obtain a significant improvement to the following theorem from \cite[Thm. 1]{PB} for existence of convolution of Roumieu ultradistributions. For the sake of completeness, we recall the definition of the space $\dot{\mathcal{B}}^{\{M_p\}}_{\Delta}$ (cf. \cite[p. 97]{PB}) involved in this result. For $a>0$, we define $\dot{\mathcal{B}}^{\{M_p\}}_a=\{\varphi\in\dot{\mathcal{B}}^{\{M_p\}}(\RR^{2d})|\, \mathrm{supp}\:\varphi\subseteq \Delta_a\}$, where $\Delta_a=\{(x,y)\in\RR^{2d}|\, |x+y|\leq a\}$. Provided with family of seminorms
\beqs
\varphi\mapsto \sup_{\alpha,\beta\in\NN^d}\sup_{(x,y)\in\RR^{2d}} \frac{|D^{\alpha}_xD^{\beta}_y\varphi(x,y)|}{M_{\alpha+\beta}\prod_{j=1}^{|\alpha|+|\beta}r_j},\,\,\,\,\,\, \mbox{for}\,\, (r_p)\in\mathfrak{R},
\eeqs
$\dot{\mathcal{B}}^{\{M_p\}}_a$ becomes a l.c.s.. We define as l.c.s. $\ds\dot{\mathcal{B}}^{\{M_p\}}_{\Delta}=\lim_{\substack{\longrightarrow\\ a\rightarrow \infty}} \dot{\mathcal{B}}^{\{M_p\}}_a$.

\begin{te}[\cite{PB})]
Let $S,T\in\DD'^{\{M_p\}}\left(\RR^d\right)$. The following statements are equivalent:
\begin{itemize}
\item[$(i)$] the convolution of $S$ and $T$ exists;
\item[$(ii)$] $S\otimes T\in \left(\dot{\mathcal{B}}^{\{M_p\}}_{\Delta}\right)'$;
\item[$(iii)$] for all $\varphi\in\DD^{\{M_p\}}\left(\RR^d\right)$, $\left(\varphi*\check{S}\right)T\in\tilde{\DD}'^{\{M_p\}}_{L^1}$ and for every compact subset $K$ of $\RR^d$,  $(\varphi,\chi)\mapsto \left\langle\left(\varphi*\check{S}\right)T,\chi\right\rangle$, $\DD^{\{M_p\}}_K\times\dot{\tilde{\mathcal{B}}}^{\{M_p\}}\longrightarrow \CC$, is a  continuous bilinear mapping;
\item[$(iv)$] for all $\varphi\in\DD^{\{M_p\}}\left(\RR^d\right)$, $\left(\varphi*\check{T}\right)S\in\tilde{\DD}'^{\{M_p\}}_{L^1}$ and for every compact subset $K$ of $\RR^d$,
$(\varphi,\chi)\mapsto \left\langle\left(\varphi*\check{T}\right)S,\chi\right\rangle$, $\DD^{\{M_p\}}_K\times\dot{\tilde{\mathcal{B}}}^{\{M_p\}}\longrightarrow \CC$, is a continuous bilinear mapping;
\item[$(v)$] for all $\varphi,\psi\in\DD^{\{M_p\}}\left(\RR^d\right)$, $\left(\varphi*\check{S}\right)(\psi*T)\in L^1\left(\RR^d\right)$.
\end{itemize}
\end{te}

We now have:
\begin{te}
Let $S,T\in\DD'^{\{M_p\}}\left(\RR^d\right)$. Then the following conditions are equivalent
\begin{itemize}
\item[$(i)$] the convolution of $S$ and $T$ exists;
\item[$iii)'$] for all $\varphi\in\DD^{\{M_p\}}\left(\RR^d\right)$, $\left(\varphi*\check{S}\right)T\in \DD'^{\{M_p\}}_{L^1}$;
\item[$iv)'$] for all $\varphi\in\DD^{\{M_p\}}\left(\RR^d\right)$, $\left(\varphi*\check{T}\right)S\in \DD'^{\{M_p\}}_{L^1}$.
\end{itemize}
\end{te}

\begin{proof} We will prove that $(iii)\Leftrightarrow (iii)'$; the prove that $(iv)\Leftrightarrow (iv)'$ is similar. Observe that $(iii)\Rightarrow (iii)'$ is trivial. Let $(iii)'$ hold. Then, by Theorem \ref{edn}, $\DD'^{\{M_p\}}_{L^1}$ is an $(F)$-space as the strong dual of a $(DF)$-space. The mapping $\chi\mapsto \left\langle\left(\varphi*\check{S}\right)T,\chi\right\rangle$, $\dot{\mathcal{B}}^{\{M_p\}}\rightarrow\CC$ is continuous for each fixed $\varphi\in\DD^{\{M_p\}}_K$ since $\left(\varphi*\check{S}\right)T\in \DD'^{\{M_p\}}_{L^1}$. Fix $\chi\in\dot{\mathcal{B}}^{\{M_p\}}$. Then the mapping $\varphi\mapsto \left(\varphi*\check{S}\right)T$, $\DD^{\{M_p\}}_K\rightarrow \DD'^{\{M_p\}}\left(\RR^d\right)$ is continuous; hence, it has a closed graph. But $\left(\varphi*\check{S}\right)T\in \DD'^{\{M_p\}}_{L^1}$ and $\DD'^{\{M_p\}}_{L^1}$ is continuously injected into $\DD'^{\{M_p\}}\left(\RR^d\right)$; hence, the mapping $\varphi\mapsto \left(\varphi*\check{S}\right)T$, $\DD^{\{M_p\}}_K\rightarrow \DD'^{\{M_p\}}_{L^1}$ has a closed graph. We have that $\DD^{\{M_p\}}_K$ is barreled (in fact it is a $(DFS)$-space). Since $\DD'^{\{M_p\}}_{L^1}$ is an $(F)$-space it is a Pt\'{a}k space hence this mapping is continuous by the Pt\'{a}k closed graph theorem (cf. \cite[Thm. 8.5, p. 166]{Sch}). We obtain that for each fixed $\chi\in\dot{\mathcal{B}}^{\{M_p\}}$, the mapping $\varphi\mapsto \left\langle\left(\varphi*\check{S}\right)T,\chi\right\rangle$, $\DD^{\{M_p\}}_K\rightarrow\CC$ is continuous. Hence, the bilinear mapping $(\varphi,\chi)\mapsto \left\langle\left(\varphi*\check{S}\right)T,\chi\right\rangle$, $\DD^{\{M_p\}}_K\times\dot{\mathcal{B}}^{\{M_p\}}\rightarrow \CC$ is separately continuous. Since $\DD^{\{M_p\}}_K$ and $\dot{\mathcal{B}}^{\{M_p\}}$ are barreled $(DF)$-spaces, this mapping is continuous.
\end{proof}

\subsection{Relation between $\mathcal{D}'^*_{E'_{\ast}}$, $\mathcal{B}'^*_{\omega}$, and
$\mathcal{D}'^*_{L^{1}_{\check{\omega}}}$ -- Convolution and multiplication}
\label{subsection convolution 2}

We now study convolution and multiplicative products on $\mathcal{D}'^*_{E'_{\ast}}$. For it, we first need the following proposition.

\begin{pro}\label{prop4.13}
The following dense and continuous inclusions hold
$\mathcal{D}^*_{L^1_\omega}\hookrightarrow\mathcal{D}^*_E\hookrightarrow\dot{\mathcal{B}}^*_{\check{\omega}}$ and the inclusions $\mathcal{D}'^*_{L^1_{\check{\omega}}}\rightarrow\mathcal{D}'^*_{E'_{\ast}}\rightarrow\mathcal{B}'^*_{\omega}$ are continuous. If $E$ is reflexive, one has $\mathcal{D}'^*_{L^{1}_{\check{\omega}}}\hookrightarrow\mathcal{D}'^*_{E'}\hookrightarrow\dot{\mathcal{B}}'^*_{\omega}$.
\end{pro}

\begin{proof} The proof follows the same lines as in the distribution case \cite[Thm. 4]{DPV} (by using the analogous results for ultradistributions); we therefore omit it.
\end{proof}

By the above proposition and the fact $\DD^*(\RR^d)\hookrightarrow \mathcal{D}'^*_{L^{1}_{\eta}}$ (which is easily obtainable by direct inspection) we have $\mathcal{D}^*_{L^{1}_{\omega_{\eta}}}\hookrightarrow\mathcal{D}^*_{L^{p}_{\eta}}\hookrightarrow \dot{\mathcal{B}}^*_{\check{\omega}_{\eta}}$ and $\mathcal{D}'^*_{L^{1}_{\check{\omega}_{\eta}}}\hookrightarrow\mathcal{D}'^*_{L^{p}_{\eta}}\hookrightarrow \dot{\mathcal{B}}'^*_{\omega_{\eta}}$ for $1\leq p <\infty$.

In addition, a direct consequence of this proposition is that the spaces $\mathcal{D}_E^*$ are never Montel spaces when $\omega$ is a bounded weight. In fact, if $\varphi\in\mathcal{D}^*(\mathbb{R}^{d})$ is nonnegative with $\varphi(x)=0$ for $|x|\geq 1/2$ and $\theta\in\mathbb{R}^{d}$ is a unit vector, then $\{(T_{-j\theta}\varphi)/\omega(j\theta)\}_{j=0}$ is a bounded sequence in $\mathcal{D}^*_{L^1_\omega}$ and, hence, in $\mathcal{D}^*_{E}$ without any accumulation point.

It is also easy to verify that $\dot{\mathcal{B}}^*_{{\eta}}\hookrightarrow
\dot{\mathcal{B}}^*_{\omega_{\eta}}$ and $\dot{\mathcal{B}}'^*_{{\eta}}\hookrightarrow \dot{\mathcal{B}}'^*_{\omega_{\eta}}$.

 The multiplicative product mappings $\cdot:\mathcal{D}'^*_{L^{p}_{\eta}}\times
\mathcal{B}^*_{\eta}\to \mathcal{D}'^*_{L^{p}}$ and $\cdot:\mathcal{B}'^*_{\eta}\times
\mathcal{D}^*_{L^{p}_{\eta}}\to \mathcal{D}'^*_{L^{p}}$ are well defined and hypocontinuous for $1\leq p<\infty$. In particular, $f\varphi$ is an integrable ultradistribution whenever $f\in\mathcal{B}'^*_{\eta}$ and $\varphi\in\mathcal{D}^*_{L^{1}_{\eta}}$ or $f\in\mathcal{D}'^*_{L^{1}_{\eta}}$ and $\varphi\in\mathcal{B}^*_{\eta}$. If $(1/r)=(1/p_{1})+(1/p_{2})$ with $1\leq r, p_1,p_2<\infty$, it is also clear that the multiplicative product $\cdot:\mathcal{D}'^*_{L^{p_1}_{\eta_{1}}}\times \mathcal{D}^*_{L_{\eta_{2}}^{p_2}}\to \mathcal{D}'^*_{L^{r}_{\eta_{1}\eta_{2}}}$ is hypocontinuous. Clearly, the convolution product can always be canonically defined as a hypocontinuous mapping in the following situations, $\ast:\mathcal{D}'^*_{L^{p}_{\eta}}\times \mathcal{D}'^*_{L_{\omega}^{1}}\to \mathcal{D}'^*_{L^{p}_{\eta}}$, $1\leq p\leq \infty$, and $\ast:\dot{\mathcal{B}}'^*_{\eta}\times \mathcal{D}'^*_{L_{\omega}^{1}}\to \dot{\mathcal{B}}'^*_{\eta}$. Furthermore, such convolution products are continuous bilinear mappings. In fact, in the Roumieu case these spaces are $(F)$-spaces, and therefore, continuity is equivalent to separate continuity; for the Beurling case, it follows from the equivalence between hypocontinuity and continuity for bilinear mappings on $(DF)$-spaces (cf. \cite[Thm. 10, p. 160]{kothe1}).

We can now define multiplication and convolution operations on $\mathcal{D}'^*_{E'_{\ast}}$. In the next proposition we denote by $\mathcal{O}'^*_{C,b}(\RR^d)$ the space $\mathcal{O}'^*_C(\RR^d)$ equipped with the strong topology from the duality $\left\langle \mathcal{O}^*_C(\RR^d),\mathcal{O}'^*_C(\RR^d)\right\rangle$.

\begin{pro}\label{conv-prod} The convolution mappings $\ast:\mathcal{D}'^*_{E'_{\ast}}\times\mathcal{D}'^*_{L^{1}_{\check{\omega}}}\to \mathcal{D}'^*_{E'_{\ast}}$ and  $
\ast:\mathcal{D}'^*_{E'_{\ast}}\times\mathcal{O}'^*_{C,b}(\RR^d)\to \mathcal{D}'^*_{E'_{\ast}}
$ are continuous. The convolution and multiplicative products are hypocontinuous in the following cases:
$
\cdot:\mathcal{D}'^*_{E'_{\ast}}\times\mathcal{D}^*_{L^{1}_{\omega}}\to \mathcal{D}'^*_{L^{1}}
$,
$
\cdot:\mathcal{D}'^*_{L^{1}_{\check{\omega}}}\times\mathcal{D}^*_{E}\to \mathcal{D}'^*_{L^{1}}
$,
and
$
\ast:\mathcal{D}'^*_{E'_{\ast}}\times\mathcal{D}^*_{\check{E}}\to \mathcal{B}^*_{\omega}
$. If $E$ is reflexive, we have $
\ast:\mathcal{D}'^*_{E'}\times\mathcal{D}^*_{\check{E}}\to \dot{\mathcal{B}^*_{\omega}}
$.
\end{pro}

\begin{proof} The proof goes along the same lines as in the distribution case \cite[Prop. 11]{DPV} (again, by using the analogous results for ultradistributions).
\end{proof}

Note that, as a consequence of Proposition \ref{conv-prod}, $f\varphi$ is an integrable ultradistribution (i.e., an element of $\DD'^*_{L^1}$) if $f\in\mathcal{D}'^*_{E'_{\ast}}$ and $\varphi\in\mathcal{D}^*_{L^{1}_{\omega}}$ or if
$f\in\mathcal{D}'^*_{L^{1}_{\check{\omega}}}$ and $\varphi\in\mathcal{D}^*_{E}$.

\end{document}